%% file: Neurips_NcvxMinMax_CameraReady.tex
\documentclass{article}

 \PassOptionsToPackage{numbers, compress}{natbib}
\usepackage[final]{neurips_2020}

\usepackage[utf8]{inputenc} % allow utf-8 input
\usepackage[T1]{fontenc}    % use 8-bit T1 fonts
\usepackage{hyperref}       % hyperlinks
\usepackage{url}            % simple URL typesetting
\usepackage{booktabs}       % professional-quality tables
\usepackage{amsfonts}       % blackboard math symbols
\usepackage{nicefrac}       % compact symbols for 1/2, etc.
\usepackage{microtype}      % microtypography

\input{my_preamble}

\title{Hybrid Variance-Reduced SGD Algorithms For Minimax Problems with Nonconvex-Linear Function}

\author{%
Quoc Tran-Dinh$^{*}$ \quad\qquad Deyi Liu$^{*}$ \quad\qquad Lam M. Nguyen$^{\dagger}$\vspace{0.5ex}\\
  $^{*}$Department of Statistics and Operations Research\\
  The University of North Carolina at Chapel Hill, Chapel Hill, NC 27599 \\
  Emails: \texttt{\{quoctd@email.unc.edu, deyi.liu@live.unc.edu\}} \vspace{0.25ex}\\
$^{\dagger}$IBM Research, Thomas J. Watson Research Center \\ Yorktown Heights, NY10598, USA.\\
Email: \texttt{lamnguyen.mltd@ibm.com}
}

\begin{document}

\maketitle

\begin{abstract}
We develop a novel and single-loop variance-reduced algorithm to solve a class of stochastic nonconvex-convex minimax problems involving a nonconvex-linear objective function, which has various applications in different fields such as machine learning and robust optimization.
This problem class has several computational challenges due to its nonsmoothness, nonconvexity, nonlinearity, and non-separability of the objective functions.
Our approach relies on a new combination of recent ideas, including smoothing and hybrid biased variance-reduced techniques.
Our algorithm and its variants can achieve $\mathcal{O}(T^{-2/3})$-convergence rate and the best known oracle complexity under standard assumptions, where $T$ is the iteration counter.
They have several computational advantages compared to existing methods such as simple to implement and less parameter tuning requirements. 
They can also work with both single sample or mini-batch on derivative estimators, and with constant or diminishing step-sizes.
We demonstrate the benefits of our algorithms over existing methods through two numerical examples, including a nonsmooth and nonconvex-non-strongly concave minimax model.
\end{abstract}

%%%%%%%%%%%%%%%%%%%%%%%%%%%%%%%%%%%%%%%%%%%%%
%%%% 1. Introduction.
%%%%%%%%%%%%%%%%%%%%%%%%%%%%%%%%%%%%%%%%%%%%%
\beforesec
\section{Introduction}
\aftersec
We study the following stochastic minimax problem with nonconvex-linear objective function, which covers various practical problems in different fields, see, e.g., \cite{Ben-Tal2009,Facchinei2003,goodfellow2014generative}:
\begin{equation}\label{eq:min_max_form}
\min_{x\in\R^p}\max_{y\in\R^n}\Big\{ \Psi(x, y) := \Rc(x) + \Exps{\xi}{\iprods{Ky, \Fb(x,\xi)}} - \psi(y) \Big\},
\end{equation}
where $\Fb : \R^p \times \Omega \to \R^q$ is a stochastic vector function defined on a probability space $(\Omega, \mathbb{P})$, $K\in\R^{q\times n}$ is a given matrix, $\iprods{\cdot,\cdot}$ is an inner product, and $\psi: \R^n \to \R \cup \{ +\infty\}$ and $\Rc : \R^p\to\Rext$ are proper, closed, and convex functions \cite{Bauschke2011}.
Problem \eqref{eq:min_max_form} is a special case of the nonconvex-concave minimax problem, where  $\mathcal{H}(x, y) := \Exps{\xi}{\iprods{Ky, \Fb(x,\xi)}}$ is nonconvex in $x$ and linear in $y$.

Due to the linearity of $\mathcal{H}$ w.r.t. $y$, \eqref{eq:min_max_form} can be reformulated into a general stochastic compositional nonconvex problem of the form:
\begin{equation}\label{eq:com_nlp}
\min_{x\in\R^p} \Big\{ \Psi_0(x) := \phi_0(F(x)) + \Rc(x) \equiv \phi_0\left(\Exps{\xi}{\Fb(x,\xi)} \right) + \Rc(x) \Big\},
\end{equation}
where $\phi_0$ is a convex, but possibly nonsmooth function, defined as 
\begin{equation}\label{eq:phi_func}
\phi_0(u) := \max_{y\in\R^n}\set{ \iprods{K^{\top}u, y} - \psi(y)} \equiv \psi^{*}(K^{\top}u),
\end{equation}
with $\psi^{*}$ being the Fenchel conjugate of $\psi$ \cite{Bauschke2011}, and we define $\Phi_0(x) := \phi_0(F(x))$.
%The reformulation \eqref{eq:com_nlp} has been widely studied in the literature, see, e.g., \cite{lian2017finite,wang2017stochastic,wang2017accelerating,zhang2019multi} under specific assumptions.
Note that problem \eqref{eq:com_nlp} is completely different from existing models such as \cite{drusvyatskiy2019efficiency,duchi2018stochastic}, where the expectation is inside the outer function $\phi_0$, i.e., $\phi_0\left(\Exps{\xi}{\Fb(x,\xi)} \right)$.
We refer to this setting as a ``non-separable" model.

\noindent\textbf{Challenges:}
Developing numerical methods for solving \eqref{eq:min_max_form} or \eqref{eq:com_nlp} faces several challenges.
First, it is often nonconvex, i.e., $F$ is not affine. 
Many recent papers consider special cases of \eqref{eq:com_nlp} when $\Psi_0$ in \eqref{eq:com_nlp} is convex by imposing restrictive conditions, which are unfortunately not realistic in applications. 
Second, the max-form $\phi_0$ in \eqref{eq:phi_func} is often nonsmooth if $\psi$ is not strongly convex.
This prevents the use of gradient-based methods.
Third, since the expectation is inside $\phi_0$, it is very challenging to form an unbiased estimate for [sub]gradients of $\Phi_0$, making classical stochastic gradient-based methods inapplicable.
Finally, prox-linear operator-based methods as in \cite{drusvyatskiy2019efficiency,duchi2018stochastic,tran2020stochastic,zhang2020stochastic} require large mini-batch evaluations of  both function value $\Fb$ and its Jacobian $\Fb'$, see  \cite{tran2020stochastic,zhang2019multi,zhang2020stochastic}, instead of single sample or small mini-batch, making them less flexible and more expensive  than gradient-based methods.

%%% Related work.
\noindent\textbf{Related work:}
Problem \eqref{eq:min_max_form} has recently attracted considerable attention due to key applications, e.g., in game theory, robust optimization, distributionally robust optimization, and generative adversarial nets (GANs) \cite{Ben-Tal2009,Facchinei2003,goodfellow2014generative,rahimian2019distributionally}. 
Various first-order methods have been developed to solve \eqref{eq:min_max_form} during the past decades for both convex-concave models , e.g., \cite{Bauschke2011,Korpelevic1976,Nemirovskii2004,tseng2008accelerated} and  nonconvex-concave settings \cite{lin2018solving,lin2019gradient,loizou2020stochastic,ostrovskii2020efficient,thekumparampil2019efficient}.
Some recent works consider a nonnonvex-nonconcave formulation, e.g., \cite{nouiehed2019solving,yang2020global}.
However, they still rely on additional assumptions to guarantee that the maximization problem in \eqref{eq:phi_func} can globally be solved.
One well-known assumption is the Polyak-{\L}ojasiewicz (PL) condition, which is rather strong and often used to guarantee linear convergence rates.
A majority of these works focus on deterministic models, while some methods have been extended to stochastic settings, e.g., \cite{lin2018solving,yang2020global}.
Although  \eqref{eq:min_max_form} is a special case of a general model in \cite{lin2018solving,lin2019gradient,yang2020global}, it almost covers all examples in \cite{lin2018solving,yang2020global}.
Compared to these, we only consider a special class of minimax problems where the function $\mathcal{H}$ is linear in $y$.
However, our algorithm is rather simple with a single loop, and our oracle complexity is significantly improved over the ones in \cite{lin2018solving,yang2020global}.

In a very recent work \cite{luo2020stochastic}, which is concurrent to our paper, the authors develop a double-loop algorithm, called SREDA, to handle a more general case than \eqref{eq:min_max_form} where $\mathcal{H}$ is strongly concave in $y$. 
Their method exploits the SARAH estimator introduced in \cite{nguyen2017sarah} and can achieve the same $\BigO{\varepsilon^{-3}}$ oracle complexity as ours in Theorem~\ref{th:convergence2_scvx} below.
Compared to our work here, though the problem setting in \cite{luo2020stochastic} is more general than \eqref{eq:min_max_form}, it does not cover the non-strongly convex case.
This is important to handle stochastic constrained optimization problems, where $\psi$ is nonsmooth and convex, but not necessarily strongly convex (see, e.g., \eqref{eq:min_max_stochastic_opt} below as an example). 
Moreover, the SREDA algorithm in  \cite{luo2020stochastic} requires double loops with large mini-batch sizes in both function values and derivatives and uses small learning rates to achieve the desired oracle complexity.

It is interesting that the minimax problem \eqref{eq:min_max_form} can be reformulated into a nonconvex compositional optimization problem of the form \eqref{eq:com_nlp}. 
The formulation \eqref{eq:com_nlp} has been broadly studied in the literature under both deterministic and stochastic settings, see, e.g.,  \cite{drusvyatskiy2019efficiency,duchi2018stochastic,Lewis2008,Nesterov2007g,Tran-Dinh2011,wang2017stochastic}.
If  $q=1$ and $\phi_0(u) = u$, then \eqref{eq:com_nlp} reduces to the standard stochastic optimization model studied e.g., in \cite{ghadimi2016accelerated,Pham2019}.
In the deterministic setting, one common method to solve \eqref{eq:com_nlp} is the prox-linear-type method, which is also known as a Gauss-Newton method \cite{Lewis2008,Nesterov2007g}.
This method has been studied in several papers, including \cite{drusvyatskiy2019efficiency,duchi2018stochastic,Lewis2008,Nesterov2007g,Tran-Dinh2011}.
However, the prox-linear operator often does not have a closed form expression, and its evaluation may require solving a general nonsmooth strongly convex  subproblem.

In the stochastic setting as \eqref{eq:com_nlp}, \cite{wang2017stochastic,wang2017accelerating} proposed stochastic compositional gradient methods to solve more general forms than \eqref{eq:com_nlp}, but they required a set of stronger assumptions than Assumptions~\ref{ass:A1}-\ref{ass:A2} below, including the smoothness of $\phi_0$.
Recent related works include \cite{lian2017finite,liu2017variance,xu2019katyusha,yang2019multilevel,yu2017fast}, which also rely on similar ideas.
For instance, \cite{lin2018solving} proposed a double loop subgradient-based method with $\BigO{\varepsilon^{-6}}$ oracle complexity.
Another subgradient-based method was recently proposed in \cite{yang2020global} based on a two-side PL condition.
Stochastic methods exploiting prox-linear operators have also been recently proposed in \cite{tran2020stochastic,zhang2020stochastic}, which are essentially extensions of existing deterministic methods to \eqref{eq:com_nlp}.
Together with algorithms, convergence guarantees, stochastic oracle complexity bounds have also been estimated.
For instance, \cite{wang2017stochastic} obtained $\BigO{\varepsilon^{-8}}$ oracle complexity for \eqref{eq:com_nlp}, while it was improved to $\BigO{\varepsilon^{-4.5}}$ in \cite{wang2017accelerating}.
Recent works \cite{zhang2019multi,zhang2019stochastic} further improved the complexity to $\BigO{\varepsilon^{-3}}$.
These methods require the smoothness of both $\phi_0$ and $F$, use large batch sizes, and need a double-loop scheme. 
In contrast, \textit{\textbf{our method has single loop, can work with either single sample or mini-batch, and allows both constant or diminishing step-sizes}}.
For nonsmooth $\phi_0$, under the same assumptions as \cite{tran2020stochastic,zhang2020stochastic}, our methods achieve $\BigO{\varepsilon^{-3}}$ Jacobian and $\BigO{\varepsilon^{-5}}$ function evaluation complexity as in those papers.
However, our method is gradient-based, which only uses proximal operator of $\psi$ and $\Rc$ instead of a complex prox-linear operator as in \cite{tran2020stochastic,zhang2020stochastic}.
Note that even if $\psi$ and $\Rc$ have closed-form proximal operator, the prox-linear operator still does not have a closed-form solution, and requires to solve a composite and possibly nonsmooth strongly convex subproblem involving a linear operator, see, e.g., \cite{tran2020stochastic}.
Moreover, our method can work with both single sample and mini-batch for Jacobian $\Fb'$ compared to a large batch size as in \cite{tran2020stochastic,zhang2020stochastic}.

%%% Our contribution
\noindent\textbf{Our contribution:}
Our main contribution in this paper can be summarized as follows:
\begin{compactitem}
\item[(a)] We develop a new single-loop hybrid variance-reduced SGD algorithm to handle \eqref{eq:min_max_form} under Assumptions~\ref{ass:A1} and \ref{ass:A2} below.
Under the strong convexity of  $\psi$, our algorithm has $\BigO{(bT)^{-2/3}}$ convergence rate to approximate a KKT (Karush-Kuhn-Tucker) point of \eqref{eq:min_max_form}, where $b$ is the  batch size and $T$ is the iteration counter.
We also estimate an $\BigO{\varepsilon^{-3}}$-oracle complexity to obtain an $\varepsilon$-KKT point, matching the best known one as, e.g., in \cite{luo2020stochastic,zhang2019multi,zhang2019stochastic}.
Our complexity bound holds for a wide range of $b$ as opposed to a specific choice as in \cite{luo2020stochastic,zhang2019multi,zhang2019stochastic}.
Moreover, our algorithm has only a single loop compared to \cite{luo2020stochastic,zhang2019multi}..

\item [(b)] When $\psi$ is non-strongly convex, we combine our approach with a smoothing technique to develop a gradient-based variant, that can achieve the best-known $\BigO{\varepsilon^{-3}}$ Jacobian and $\BigO{\varepsilon^{-5}}$ function evaluations of $\Fb$ for finding an $\varepsilon$-KKT point of \eqref{eq:min_max_form}. 
Moreover, our algorithm does not require prox-linear operators and large batches for Jacobian as in \cite{tran2020stochastic,zhang2020stochastic}.

\item [(c)] We also propose a simple restarting technique without sacrificing convergence guarantees to accelerate the practical performance of both cases (a) and (b) (see Supp. Doc. \ref{apdx:sec:restarting_hSGD}).
\end{compactitem}
Our methods exploit a recent biased hybrid estimators introduced in \cite{Tran-Dinh2019a} as opposed to SARAH ones in \cite{tran2020stochastic,zhang2019multi,zhang2020stochastic}.
This allows us to simplify our algorithm with a single loop and without large batches at each iteration compared to \cite{zhang2019multi}.
As indicated in \cite{arjevani2019lower}, our $\BigO{\varepsilon^{-3}}$ oracle complexity is optimal under the considered assumptions.
If $\psi$ is non-strongly convex (i.e. $\phi_0$ in \eqref{eq:com_nlp} can be nonsmooth), then our algorithm is fundamentally different from the ones in \cite{tran2020stochastic,zhang2020stochastic} as it does not use prox-linear operator. Note that evaluating a prox-linear operator requires to solve a general strongly convex but possible nonsmooth subproblem.
In addition, they only work with large batch sizes of both $\Fb$ and $\Fb'$.

\noindent\textbf{Content:}
Section~\ref{sec:math_tools} states our assumptions and recalls some mathematical tools. 
Section~\ref{sec:alg_and_theory} develops a new algorithm and analyzes its convergence.
Section~\ref{sec:num_exps} provides two numerical examples to compare our methods.
All technical details and proofs are deferred to Supplementary Document (Supp. Doc.).

%%%%%%%%%%%%%%%%%%%%%%%%%%%%%%%%%%%%%%%%%%%%%
%%%% 2. Basic Assumptions and Mathematical Tools.
%%%%%%%%%%%%%%%%%%%%%%%%%%%%%%%%%%%%%%%%%%%%%
\beforesec
\section{Basic assumptions, KKT points and smoothing technique}\label{sec:math_tools}
\aftersec
\textbf{Notation:}
We work with finite-dimensional space $\R^p$ equipped with standard inner product $\iprods{\cdot,\cdot}$ and Euclidean norm $\norms{\cdot}$.
For a function $\phi :\R^p\to\Rext$, $\dom{\phi}$ denotes its domain.
If $\phi$ is convex, then $\prox_{\phi}$ denotes its proximal operator, $\partial{\phi}$ denotes its subdifferential, and $\nabla{\phi}$ is its [sub]gradient, see, e.g., \cite{Bauschke2011}.
$\phi$ is $\mu_{\phi}$-strongly convex with a strongly convex parameter $\mu_{\phi} > 0$ if $\phi(\cdot) - \frac{\mu_{\phi}}{2}\norms{\cdot}^2$ remains convex.
For a smooth vector function $F:\R^p\to\R^q$, $F'$ denotes its Jacobian.
We use $\dist{x,\Xc} := \inf_{y\in\Xc}\norms{x-y}$ to denote the Euclidean distance from $x$ to a convex set $\Xc$.

\beforesubsec
\subsection{Model assumptions}
\aftersubsec
Let $F(x) := \Exps{\xi}{\Fb(x,\xi)}$ denote the expectation function of $\Fb$ and $\dom{\Psi_0}$ denote the domain of $\Psi_0$.
Throughout this paper, we always assume that $\Psi_0^{\star} := \inf_{x\in\R^p}\set{ \Psi_0(x) := \phi_0(F(x)) + \Rc(x) }> -\infty$ in \eqref{eq:com_nlp} and $\Rc$ is proper, closed, and convex without recalling them in the sequel.
Our goal is to develop stochastic gradient-based algorithms to solve \eqref{eq:min_max_form} relying on the following assumptions:
\begin{assumption}\label{ass:A1}
The function $\Fb$ in problem \eqref{eq:min_max_form} or \eqref{eq:com_nlp} satisfies the following assumptions:
\begin{compactitem}
    \item[$\mathrm{(a)}$]\mytxtbi{Smoothness:} 
    $\Fb(\cdot,\cdot)$ is $L_{F}$-average smooth with $L_F \in (0, +\infty)$, i.e.:
    \begin{equation}\label{eq:F_smooth}	
	\Exps{\xi}{\norm{\Fb'(x,\xi) - \Fb'(y,\xi)}^2} \leq L_{F}^2\norm{x - y}^2,~~\forall x, y\in\dom{\Psi_0}.
    \end{equation}
    
     \item[$\mathrm{(b)}$]\mytxtbi{Bounded variance:}
    There exists two constants $\sigma_F, \sigma_J\in (0, +\infty)$ such that
    \begin{equation*}
    \Exps{\xi}{\norm{\Fb(x,\xi) - F(x)}^2} \leq \sigma_F^2 \quad \text{and}\quad \Exps{\xi}{\norm{\Fb'(x,\xi) - F'(x)}^2} \leq \sigma_{J}^2,~~~\forall x\in\dom{\Psi_0}.
    \end{equation*}
    \item[$\mathrm{(c)}$]\mytxtbi{Lipschitz continuity:}
     $F(\cdot)$ is $M_F$-average Lipschitz continuous with  $M_{F} \in (0, +\infty)$, i.e.:
    \begin{equation}\label{eq:F_Lipschitz}	
   	\Exps{\xi}{\norm{\Fb'(x,\xi)}^2} \leq M_F^2,~~~\forall x\in\dom{\Psi_0}.
    \end{equation}
\end{compactitem}
\end{assumption}
Note that Assumptions~\ref{ass:A1} are standard in stochastic nonconvex optimization, see \cite{tran2020stochastic,zhang2019multi,zhang2019stochastic,zhang2020stochastic}.
If $\dom{\Rc}$ is bounded, then $\dom{\Psi_0}$ is bounded, and this assumption automatically holds.

For $\psi$, we only require the following assumption, which is mild and holds for many applications.
\begin{assumption}\label{ass:A2}
The function $\psi$ in \eqref{eq:min_max_form} is proper, closed, and convex.
Moreover, $\dom{\psi}$ is bounded by $M_{\psi}\in (0,+\infty)$, i.e.: $\sup\set{\norms{y} : y\in\dom{\psi}}\leq M_{\psi}$.
\end{assumption}
An important special case of $\psi$ is  the indicator of convex and bounded sets.
Hitherto, we do not require $\phi_0$ and $\Rc$ in \eqref{eq:com_nlp} to be smooth or strongly convex. 
They can be nonsmooth so that \eqref{eq:com_nlp} can also cover constrained problems.
Note that the boundedness of $\dom{\psi}$ is equivalent to the Lipschitz continuity of $\phi_0$ (Lemma~\ref{le:properties_of_phi}).
Simple examples of $\phi_0$ include norms and gauge functions.

\beforesubsec
\subsection{KKT points and approximate KKT points}
\aftersubsec
Since \eqref{eq:min_max_form} is nonconvex-concave, a pair $(x^{\star}, y^{\star})$ is said to be a KKT point of  \eqref{eq:min_max_form} if
\begin{equation}\label{eq:KKT_point}
0 \in F'(x^{\star})^{\top}Ky^{\star} + \partial{\Rc}(x^{\star})\quad\quad\text{and}\quad\quad 0 \in K^{\top}F(x^{\star}) - \partial{\psi}(y^{\star}).
\end{equation}
From \eqref{eq:KKT_point}, we have $y^{\star} \in \partial{\psi^{*}}(K^{\top}F(x^{\star}))$.
Substituting this $y^{\star}$ into the first expression, we get 
\begin{equation}\label{eq:stationary_point}
0 \in F'(x^{\star})^{\top}\partial{\phi_0}(F(x^{\star})) + \partial{\Rc}(x^{\star}).
\end{equation}
Here, we have used $K^{\top}\partial{\psi^{*}}(K^{\top}u) =  \partial{\phi_0}(u)$, where $\phi_0$ is given by \eqref{eq:phi_func}
This inclusion shows that $x^{\star}$ is a stationary point of \eqref{eq:com_nlp}.
In the convex-concave case, under mild assumptions, a KKT point is also a saddle-point of \eqref{eq:min_max_form}.
In particular, if \eqref{eq:com_nlp} is convex, then $x^{\star}$ is also its global optimum of \eqref{eq:com_nlp}. 

However, in practice, we can only find an approximation $(\tilde{x}_0^{*},\tilde{y}^{*}_0)$ of a KKT point $(x^{\star}, y^{\star})$ for \eqref{eq:min_max_form}.
%%% Definition 2.1.
\begin{definition}\label{de:approx_KKT_point}
Given any tolerance $\varepsilon > 0$,  $(\tilde{x}_0^{*},\tilde{y}^{*}_0)$ is called an $\varepsilon$-KKT point of \eqref{eq:min_max_form} if 
\begin{equation}\label{eq:approx_KKT_point}
\begin{array}{ll}
&\qquad\qquad\qquad \Exp{\Ec(\tilde{x}_0^{*}, \tilde{y}_0^{*})} \leq\varepsilon, \vspace{1ex}\\
\text{where}~~\Ec(x, y) &:= \dist{0, F'(x)^{\top}Ky + \partial{\Rc}(x)} + \dist{0, K^{\top}F(x) - \partial{\psi}(y)}.
\end{array}
\end{equation}
\end{definition}
Here, the  expectation is taken overall the randomness from both model \eqref{eq:min_max_form} and the algorithm.
Clearly, if $ \Exp{\Ec(\tilde{x}_0^{*}, \tilde{y}_0^{*})} = 0$, then $(\tilde{x}_0^{*}, \tilde{y}_0^{*})$ is a KKT point of \eqref{eq:min_max_form} as characterized by \eqref{eq:KKT_point}.

\beforesubsec
\subsection{Smoothing techniques}
\aftersubsec
Under Assumption~\ref{ass:A2}, $\phi_0$ defined by \eqref{eq:phi_func} can be nonsmooth.
Hence, we can smooth $\phi_0$ as follows:
\begin{equation}\label{eq:smoothed_phi}
\phi_{\gamma}(u) := \max_{y\in\R^n}\set{ \iprods{u, Ky} - \psi(y) - \gamma b(y) },
\end{equation}
where $b : \dom{\psi}\to\R_{+}$ is a continuously differentiable and $1$-strongly convex function such that $\min_yb(y)  = 0$, and $\gamma > 0$ is a smoothness parameter.
For example, we can choose $b(y) := \frac{1}{2}\norms{y - \dot{y}}^2$ for a fixed $\dot{y}$ or $b(y) := \log(n) + \sum_{j=1}^ny_j\log(y_j)$ defined on a standard simplex $\Delta_n$ \cite{Nesterov2005c}.
Under Assumption~\ref{ass:A2}, $\phi_{\gamma}$ possesses some useful properties as stated in Lemma~\ref{le:properties_of_phi} (Supp. Doc. \ref{apdx:subsec:smooth_properties}).

Let $y^{\ast}_{\gamma}(u)$ be an optimal solution of the maximization problem in \eqref{eq:smoothed_phi}, which always exists and is unique.
In particular, if $b(y) := \frac{1}{2}\norms{y - \dot{y}}^2$, then 
\begin{equation}\label{eq:prox_psi}
y^{*}_{\gamma}(u) := \mathrm{arg}\max_{y\in\R^n}\set{ \iprods{u, Ky} - \psi(y) - \tfrac{\gamma}{2}\norms{y - \dot{y}}^2 } \equiv \prox_{\psi/\gamma}\left(\dot{y} + \gamma^{-1} K^{\top} u \right).
\end{equation}
Hence, when $\psi$ is proximally tractable (i.e., its proximal operator can be computed in a closed-form or with a low-order polynomial time algorithm),  computing $y^{*}_{\gamma}(u)$ reduces to evaluating the proximal operator of $\psi$ as opposed to solving a complex subproblem as in prox-linear methods \cite{tran2020stochastic,zhang2020stochastic}.

Given $\phi_{\gamma}$ defined by \eqref{eq:smoothed_phi}, we consider the following functions:
\begin{equation}\label{eq:smoothed_compositional_func}
\Phi_{\gamma}(x) := \phi_{\gamma}(F(x)) = \phi_{\gamma}\left(\Exps{\xi}{\Fb(x, \xi)}\right)\quad\text{and}\quad \Psi_{\gamma}(x) := \Phi_{\gamma}(x) + \Rc(x).
\end{equation}
In this case, under Assumptions~\ref{ass:A1} and \ref{ass:A2}, $\Phi_{\gamma}$ is continuously differentiable, and 
\begin{equation}\label{eq:deri_Phi}
\nabla{\Phi}_{\gamma}(x) = F'(x)^{\top}\nabla{\phi}_{\gamma}(F(x)) = F'(x)^{\top}Ky_{\gamma}^{*}(F(x)).
\end{equation}
%%%
\noindent\textbf{Smoothness:}
Moreover, $\Phi_{\gamma}(\cdot)$ is $L_{\Phi_{\gamma}}$-smooth with $L_{\Phi_{\gamma}} := M_{\phi_{\gamma}}L_F + M_F^2L_{\phi_{\gamma}}$ (see \cite{zhang2019stochastic}), i.e.:
\begin{equation}\label{eq:Phi_smoothness}
\norms{\nabla{\Phi}_{\gamma}(x) - \nabla{\Phi}_{\gamma}(\hat{x})} \leq L_{\Phi_{\gamma}}\norms{x-\hat{x}},~~\forall x,\hat{x}\in\dom{\Psi_0},
\end{equation}
where $M_{\phi_{\gamma}}:= M_{\psi}\norms{K}$ and $L_{\phi_{\gamma}}:= \frac{\norms{K}^2}{\gamma + \mu_{\psi}}$ are given in Lemma~\ref{le:properties_of_phi}.

%%%
\noindent\textbf{Gradient mapping:}
Let us recall the following gradient mapping of $\Psi_{\gamma}(\cdot)$ given in \eqref{eq:smoothed_compositional_func} as  
\begin{equation}\label{eq:grad_map}
\Gc_{\eta}(x) := \tfrac{1}{\eta}\left(x - \prox_{\eta\Rc}(x - \eta\nabla{\Phi}_{\gamma}(x))\right), \quad \text{for any $\eta > 0$}.
\hspace{-2ex}
\end{equation}
This mapping will be used to characterize approximate KKT points of \eqref{eq:min_max_form} in Definition~\ref{de:approx_KKT_point}.

%%%%%%%%%%%%%%%%%%%%%%%%%%%%%%%%%%%%%%%%%%%%%
%%%% 3. Algorithms and Convergence Analysis.
%%%%%%%%%%%%%%%%%%%%%%%%%%%%%%%%%%%%%%%%%%%%%
\beforesec
\section{The proposed algorithm and its convergence analysis}\label{sec:alg_and_theory}
\aftersec
First, we introduce a stochastic estimator for $\nabla{\Phi_{\gamma}}$.
Then, we develop our main algorithm and analyze its convergence and oracle complexity.
Finally, we show how to construct an $\epsilon$-KKT point of \eqref{eq:min_max_form}.

\beforesubsec
\subsection{Stochastic estimators  and the algorithm}
\aftersubsec
Since $F$ is the expectation of a stochastic function $\Fb$, we exploit the hybrid stochastic estimators for $F$ and its Jacobian $F'$ introduced in \cite{Tran-Dinh2019a}.
More precisely, given a sequence $\set{x_t}$ generated by a stochastic algorithm, our hybrid stochastic estimators $\tilde{F}_t$ and $\tilde{J}_t$ are defined as follows:
\begin{equation}\label{eq:est_update}
\hspace{-0.25ex}
\arraycolsep=0.2em
\left\{\begin{array}{lcl}
\tilde{F}_t & := & \beta_{t-1} \tilde{F}_{t-1} + \frac{\beta_{t-1}}{b_1}\sum_{\xi_i \in \Bc_t^1} \left[ \Fb(x_t,\xi_i) - \Fb(x_{t-1},\xi_i) \right] +  \frac{(1-\beta_{t-1})}{b_2} \sum_{\zeta_i \in \Bc_t^2}\Fb(x_t,\zeta_i) \vspace{1.2ex}\\
\tilde{J}_t & := & \hat{\beta}_{t-1} \tilde{J}_{t-1} +  \frac{\hat{\beta}_{t-1}}{\hat{b}_1}\sum_{\hat{\xi}_i \in \hat{\Bc}_t^1} \big[ \Fb'(x_t, \hat{\xi}_i) - \Fb'(x_{t-1}, \hat{\xi}_i) \big] +  \frac{(1-\hat{\beta}_{t-1})}{\hat{b}_2} \sum_{\hat{\zeta}_i \in \hat{\Bc}_t^2}\Fb'(x_t, \hat{\zeta}_i),
\end{array}\right.
\hspace{-4ex}
\end{equation}
where $\beta_{t-1}, \hat{\beta}_{t-1}\in [0, 1]$ are given weights, and the initial estimators $\tilde{F}_0$ and $\tilde{J}_0$ are defined as
\begin{equation}\label{eq:est_snap_def}
\begin{array}{ll}
\tilde{F}_0 := \frac{1}{b_0} \sum_{\xi_i \in \Bc^0} \Fb(x_0,\xi_i) ~~~~\text{and}~~~~\tilde{J}_0 := \frac{1}{\hat{b}_0} \sum_{\hat{\xi}_i \in \hat{\Bc}^0} \Fb'(x_0, \hat{\xi}_i).
\end{array}
\end{equation}
Here, $\Bc^0$, $\hat{\Bc}^0$, $\Bc^1_t$, $\hat{\Bc}^1_t$, $\Bc^2_t$, and $\hat{\Bc}^2_t$ are mini-batches of sizes $b_0$, $\hat{b}_0$, $b_1$, $\hat{b}_1$, $b_2$, and $\hat{b}_2$, respectively.
We allow $\Bc_t^1$ to be \textbf{correlated} with $\Bc^2_t$, and $\hat{\Bc}^1_t$ to be \textbf{correlated} with $\hat{\Bc}_t^2$.
We also do not require any independence between these mini-batches. 
When $\Bc_t^1 \equiv \Bc_t^2$ and $\hat{\Bc}_t^1 \equiv \hat{\Bc}_t^2$, our estimators reduce the STORM estimators studied in \cite{Cutkosky2019} as a special case.
Clearly, with the choices $\Bc_t^1 \equiv \Bc_t^2$ and $\hat{\Bc}_t^1 \equiv \hat{\Bc}_t^2$, we can save $b_1$ function value evaluations and $\hat{b}_1$ Jacobian evaluations at each iteration. 

For $\tilde{F}_t$ and $\tilde{J}_t$ defined by \eqref{eq:est_update}, we introduce a stochastic estimator for the gradient $\nabla{\Phi_{\gamma}}(x_t) = F'(x_t)^{\top}\nabla{\phi_{\gamma}}(F(x_t))$ of  $\Phi_{\gamma}(\cdot)$ in \eqref{eq:smoothed_compositional_func} at $x_t$ as follows:
\begin{equation}\label{eq:grad_estimators}
v_t :=  \tilde{J}_t^{\top}\nabla{\phi}_{\gamma}(\tilde{F}_t) \equiv  \tilde{J}_t^{\top}Ky^{*}_{\gamma}(\tilde{F}_t).
\end{equation}
To evaluate $v_t$, we need to compute $y^{*}_{\gamma}(\tilde{F}_t)$, which requires just one $\prox_{\gamma\psi}$ if we use \eqref{eq:prox_psi}.
%When $\psi$ is proximally tractable, computing $y^{*}_{\gamma}(\tilde{F}_t)$ reduces to evaluating the proximal operator of $\psi$.
Moreover, due to \eqref{eq:est_snap_def} and \eqref{eq:grad_estimators}, 
evaluating $v_0$ does not require the full matrix $\tilde{J}_0$, but a matrix-vector product $\tilde{J}_0^{\top}Ky^{*}_{\gamma}(\tilde{F}_0)$, which is often cheaper than evaluating $\tilde{J}_0$.

Using the new estimator $v_t$ of $\nabla{\Phi_{\gamma}}(x_t)$ in \eqref{eq:grad_estimators}, we propose Algorithm~\ref{alg:A1} to solve \eqref{eq:min_max_form}.

\begin{algorithm}[hpt!]\caption{(Smoothing Hybrid Variance-Reduced SGD Algorithm for solving \eqref{eq:min_max_form})}\label{alg:A1}
\normalsize
\begin{algorithmic}[1]
   \State{\bfseries Inputs:} An arbitrarily initial point $x_0 \in\dom{\Psi_0}$.
   \vspace{0.5ex}
   \State\hspace{3ex}\label{step:o1} Input $\beta_0, \hat{\beta}_0 \in (0, 1)$, $\gamma_0\geq 0$, $\eta_0 > 0$, and $\theta_0\in (0,1]$ (specified in Subsection~\ref{subsec:convergence}).{\!\!\!}
    \vspace{0.5ex}
    \State{\bfseries Initialization:}\label{step:o2} Generate $\tilde{F}_0$ and  $\tilde{J}_0$ as in \eqref{eq:est_snap_def} with mini-batch sizes $b_0$ and $\hat{b}_0$, respectively.
     \vspace{0.5ex}
   \State\hspace{3ex}\label{step:o2} Solve \eqref{eq:smoothed_phi} to obtain $y^{*}_{\gamma_0}(\tilde{F}_0)$.
   Then, evaluate $v_0 := \tilde{J}_0^{\top}Ky^{*}_{\gamma_0}(\tilde{F}_0)$.
    \vspace{0.5ex}
   \State\hspace{3ex}\label{step:o3} Update $\hat{x}_1 := \prox_{\eta_0\Rc}\left(x_0 - \eta_0v_0\right)$ and $x_1 := (1-\theta_0)x_0 + \theta_0\hat{x}_1$.
    \vspace{0.5ex}
   \State\hspace{0ex}\label{step:o4}{\bfseries For $t := 1,\cdots, T$ do}
    \vspace{0.5ex}
   \State\hspace{3ex}\label{step:i2} Construct $\tilde{F}_t$ and $\tilde{J}_t$ as in \eqref{eq:est_update} and $v_t := \tilde{J}_t^{\top}Ky^{*}_{\gamma_t}(\tilde{F}_t)$, 
   where $y^{*}_{\gamma_t}(\tilde{F}_t)$ solves \eqref{eq:smoothed_phi}.
    \vspace{0.5ex}
   \State\hspace{3ex}\label{step:i4} Update $\hat{x}_{t+1} := \prox_{\eta_t \Rc}\left( x_{t} - \eta_t v_{t}\right)$ and $x_{t+1} := (1-\theta_t)x_t + \theta_t\hat{x}_{t+1}$.
    \vspace{0.5ex}
   \State\hspace{3ex}\label{step:i5} Update $\beta_{t+1}, \hat{\beta}_{t+1}, \theta_{t+1} \in (0,1)$,  $\eta_{t+1} > 0$, and $\gamma_{t+1} \geq 0$ if necessary.
    \vspace{0.5ex}
   \State\hspace{0ex}{\bfseries EndFor}
    \vspace{0.5ex}
   \State\hspace{0ex}\label{step:o5}\textbf{Output:} Choose $\bar{x}_T$ randomly from $\set{x_0, x_1, \cdots, x_T}$ with $\mathbf{Prob}\{\bar{x}_T = x_t\} = \frac{\theta_t/L_{\Phi_{\gamma_t}}}{\sum_{t=0}^T\theta_t/L_{\Phi_{\gamma_t}}}$. 
\end{algorithmic}
\end{algorithm}

Algorithm~\ref{alg:A1} is designed by adopting the idea in \cite{Tran-Dinh2019a}, where it can start from two initial batches $\Bc^0$ and $\hat{\Bc}^0$ to generate a good approximation for the search direction $v_0$ before getting into the main loop.
But if diminishing step-sizes are use, it does not require such initial batchs.
However, it has $3$ major differences compared to  \cite{Tran-Dinh2019a}: the dual step $y_{\gamma_t}^{*}(\tilde{F}_t)$, the estimator $v_t$, and the dynamic parameter updates.
Note that, as explained in \eqref{eq:prox_psi}, since the dual step $y_{\gamma_t}^{*}(\tilde{F}_t)$ can be computed using $\prox_{\gamma \psi}$,  Algorithm~\ref{alg:A1} is single loop, making it easy to implement in practice compared to methods based on SVRG \cite{johnson2013accelerating} and SARAH \cite{nguyen2017sarah} such as \cite{luo2020stochastic,zhang2019multi}. 

%%% 3.3. Convergence Analysis of Algorithm 1.
\beforesubsec
\subsection{Convergence analysis of Algorithm~\ref{alg:A1}}\label{subsec:convergence}
\aftersubsec
Let  $\Fc_t$ be the $\sigma$-field generated by Algorithm~\ref{alg:A1} up to the $t$-th iteration, which is defined as follows:
\begin{equation}\label{eq:filtration}
\Fc_t := \sigma\big(x_0, \Bc^0,\hat{\Bc}^0, \Bc_1^1, \hat{\Bc}_1^1, \Bc_1^2, \hat{\Bc}_1^2, \cdots, \Bc^1_{t}, \hat{\Bc}_{t}^1, \Bc_{t}^2, \hat{\Bc}_{t}^2\big).
\end{equation}
If $\psi$ is strongly convex, then, without loss of generality, we can assume $\mu_{\psi} := 1$.
Otherwise, we can rescale it.
Moreover, for the sake of our presentation, for a given $c_0 > 0$, we introduce:
\begin{equation}\label{eq:constant_defs}
\hspace{-0ex}
\arraycolsep=0.05em
\begin{array}{lclclcl}
P &:= & \frac{\sqrt{26}\norms{K}}{3\sqrt{c_0}}\sqrt{\kappa M_F^4\norms{K}^2  + c_0\hat{\kappa} L_F^2M_{\psi}^2}, &  & Q &:=& \frac{26}{9c_0}\norms{K}^2\big(\kappa M_F^4\norms{K}^2\sigma_F^2 +  c_0\hat{\kappa} M_{\psi}^2\sigma_J^2\big), \vspace{1ex}\\
 L_{\Phi_0} &:= & L_FM_{\psi}\norms{K}  + M_F^2\norms{K}^2, &{\hspace{-3ex}}\text{and}~~ & L_{\Phi_{\gamma}} & := & L_FM_{\psi}\norms{K} +  \frac{M_F^2\norms{K}^2}{\gamma},
\end{array}
\hspace{-3ex}
\end{equation}
where $\gamma > 0$, $M_F$, $L_F$, $\sigma_F$, and $\sigma_J$ are given in Assumption~\ref{ass:A1} and $M_{\psi}$ is in Assumption~\ref{ass:A2}.
Here, $\kappa := 1$ if the mini-batch $\Bc_t^1$ is independent of $\Bc_t^2$, and $\kappa := 2$, otherwise.
Similarly, $\hat{\kappa} := 1$ if $\hat{\Bc}_t^1$ is independent of $\hat{\Bc}_t^2$, and $\hat{\kappa} := 2$, otherwise.

%%% a. The strongly convex case.
\subsubsection{The strongly concave case}
Theorem~\ref{th:convergence2_scvx}, whose proof is in Supp. Doc.~\ref{apdx:subsec:th:convergence2_scvx}, analyzes convergence rate and complexity of Algorithm~\ref{alg:A1} for the smooth case of $\phi_0$ in \eqref{eq:com_nlp} (i.e., $\psi$ is strongly convex).

%%% Theorem 3.1.
\begin{theorem}[\textbf{Constant step-size}]\label{th:convergence2_scvx}
Suppose that Assumptions~\ref{ass:A1} and \ref{ass:A2} hold, $\psi$ is $\mu_{\psi}$-strongly convex with $\mu_{\psi} := 1$, and $P$, $Q$, and $L_{\Phi_0}$ are defined in \eqref{eq:constant_defs}.
Given a mini-batch $0 < b \leq \hat{b}_0(T+1)$, let $b_0 := c_0\hat{b}_0$, $\hat{b}_1 = \hat{b}_2 := b$, and $b_1 = b_2 := c_0b$.
Let $\sets{x_t}_{t=0}^T$ be generated by Algorithm~\ref{alg:A1} using
\begin{equation}\label{eq:para_config0}
\hspace{-0.5ex}
\begin{array}{l}
\gamma_t := 0, ~~ \beta_t = \hat{\beta}_t := 1 - \frac{b^{1/2}}{[\hat{b}_0(T+1)]^{1/2}}, ~~  \theta_t =  \theta :=  \frac{L_{\Phi_0} b^{3/4}}{P [\hat{b}_0(T+1)]^{1/4}}, ~~\text{and}~~ \eta_t  = \eta := \frac{2}{L_{\Phi_0}(3 + \theta)},
\end{array}
\hspace{-2ex}
\end{equation}
provided that $\frac{\hat{b}_0(T+1)}{b^3} > \frac{L_{\Phi_0}^4}{P^4}$. % $P > \frac{L_{\Phi_0} b^{3/4}}{[\hat{b}_0(T+1)]^{1/4}}$.  
Let $b_0 := c_1^2[b(T+1)]^{1/3}$ for some $c_1 > 0$.
Then, we have
\begin{equation}\label{eq:convergence_rate1_b}
\Exp{\norms{\Gc_{\eta}(\bar{x}_T)}^2} \leq \frac{\Delta_0}{[b(T+1)]^{2/3}}, \quad\text{where}\quad \Delta_0  :=  16P \sqrt{c_1}\big[\Psi_0(x_0) - \Psi_0^{\star}\big]  +  \frac{24Q}{c_1}.
\end{equation}
For a given tolerance $\varepsilon > 0$, the total number of iterations $T$ to obtain $\Exp{\norms{\Gc_{\eta}(\bar{x}_T)}^2} \leq \varepsilon^2$ is at most $T := \big\lfloor \frac{\Delta_0^{3/2}}{b\varepsilon^3} \big\rfloor$.
The total numbers of function evaluation $\Fb(x_t, \xi)$ and its Jacobian evaluations $\Fb'(x_t,\xi)$ are at most $\Tc_{F} := \big\lfloor \frac{c_0c_1^2\Delta_0^{1/2}}{\varepsilon} + \frac{3c_0\Delta_0^{3/2}}{\varepsilon^3}\big\rfloor$ and $\Tc_{J} := \big\lfloor \frac{c_1^2\Delta_0^{1/2}}{\varepsilon} + \frac{3\Delta_0^{3/2}}{\varepsilon^3}\big\rfloor$, respectively.
\end{theorem}
% End of Theorem 3.1.

Theorem~\ref{th:convergence2_scvx_diminishing} states convergence of Algorithm~\ref{alg:A1} using diminishing step-size (see Supp. Doc.~\ref{apdx:subsec:th:convergence2_scvx_diminishing}).

%%% Theorem 3.2.
\begin{theorem}[\textbf{Diminishing step-size}]\label{th:convergence2_scvx_diminishing}
Suppose that Assumptions~\ref{ass:A1} and \ref{ass:A2} hold, $\psi$ is $\mu_{\psi}$-strongly convex with $\mu_{\psi} := 1$ $($i.e., $\phi_0$ in \eqref{eq:com_nlp} is smooth$)$.
Let $\sets{x_t}_{t=0}^T$ be generated by Algorithm~\ref{alg:A1} using the mini-batch sizes as in Theorem~\ref{th:convergence2_scvx}, and increasing weight and diminishing step-sizes as
\begin{equation}\label{eq:para_config0_a0}
\hspace{-2ex}
\begin{array}{l}
\gamma_t := 0, \quad \beta_t = \hat{\beta}_t := 1 - \frac{1}{(t+2)^{2/3}},\quad \theta_t :=  \frac{L_{\Phi_0}\sqrt{b}}{P(t+2)^{1/3}},\quad\text{and}\quad \eta_t :=  \frac{2}{L_{\Phi_0}(3 + \theta_t)}.
\end{array}
\hspace{-2ex}
\end{equation}
Then, for all $T \geq 0$, and $(\bar{x}_T,\bar{\eta}_T)$ chosen as $\mathbf{Prob}\big\{ \Gc_{\bar{\eta}_T}(\bar{x}_T) = \Gc_{\eta_t}(x_t) \big\} = \frac{\theta_t}{\sum_{t=0}^T\theta_t}$, we have
\begin{equation}\label{eq:convergence_rate1_b_a0} 
\arraycolsep=0.2em
\hspace{-1ex}
\begin{array}{lcl}
\Exp{\norms{\Gc_{\bar{\eta}_T}(\bar{x}_T)}^2}  \leq  \frac{32P [ \Psi_0(x_0)-\Psi^{\star}_0 ] }{3\sqrt{b}\big[(T+3)^{2/3} - 2^{2/3}\big]}   +  \frac{32 Q}{3\big[(T+3)^{2/3} - 2^{2/3}\big]}\left[\frac{2^{1/3}}{\hat{b}_0} + \frac{2(1 + \log(T+1))}{b}\right].
\end{array}
\hspace{-4ex}
\end{equation}
\end{theorem}
%%% End of Theorem 3.2.

If we set $b = \hat{b}_0 = 1$, then our convergence rate is $\BigO{\frac{\log(T)}{T^{2/3}}}$ with a $\log(T)$ factor slower than \eqref{eq:convergence_rate1_b}.
However, it does not require a large initial mini-batch $\hat{b}_0$ as in Theorem~\ref{th:convergence2_scvx}.
In Theorems~\ref{th:convergence2_scvx} and \ref{th:convergence2_scvx_diminishing}, we do not need to smooth $\phi_0$. 
Hence, $\gamma_t$ is absent in Algorithm~\ref{alg:A1}, i.e., $\gamma_t = 0$ for $t\geq 0$.

%The main update of Algorithm~\ref{alg:A1} can be written as $x_{t+1} := x_t - \eta_t\theta_t\widetilde{\Gc}_{\eta_t}(x_t)$, where $\widetilde{\Gc}_{\eta_t}(x_t) := \frac{1}{\eta_t}\left(x_t - \prox_{\eta_t\Rc}\left(x_t - \eta_tv_t\right)\right)$.
%Clearly, if $\Rc = 0$, then $\widetilde{\Gc}_{\eta_t}(x_t) = v_t$, which reduces to an approximation of the gradient $\nabla{\Phi}_{\gamma_t}(x_t)$.
%Thus we can refer to $\hat{\theta}_t := \theta_t\eta_t$ as a combined step-size.
%Since $\eta_t :=  \frac{2}{L_{\Phi_{\gamma_t}}(3 + \theta_t)}$ we have $\hat{\theta}_t =   \frac{2\theta_t}{L_{\Phi_{\gamma_t}}(3 + \theta_t)} \leq \frac{2\theta_t}{3L_{\Phi_{\gamma_t}}}$, which is \textbf{diminishing} to zero in \eqref{eq:para_config0_a0} or \eqref{eq:para_config0_a2}.

%%% b. The nonsmooth case.
\subsubsection{The non-strongly concave case}
Now, we consider the case $\mu_{\psi}  = 0$, i.e., $\psi$ is non-strongly convex (or equivalently, \eqref{eq:min_max_form} is non-strongly concave in $y$), leading to the nonsmoothness of $\phi_0$ in \eqref{eq:com_nlp}.
Theorem~\ref{th:convergence2} states convergence of Algorithm~\ref{alg:A1} in this case, whose proof is in Supp. Doc.~\ref{apdx:subsec:th:convergence2}.

%%% Theorem 3.2.
\begin{theorem}[\textbf{Constant step-size}]\label{th:convergence2}
Assume that Assumptions~\ref{ass:A1} and \ref{ass:A2} hold, $\psi$ in \eqref{eq:min_max_form} is non-strongly convex $($i.e., $\phi_0$ is nonsmooth$)$, and $P$, $Q$, and $L_{\Phi_{\gamma}}$ are defined in \eqref{eq:constant_defs}.
Let $b$ and $\hat{b}_0$ be two positive integers, $c_0 > 0$, and $\sets{x_t}_{t=0}^T$ be generated by Algorithm~\ref{alg:A1} after $T$ iterations using:
\begin{equation}\label{eq:choice_of_para_ncvx}
\hspace{-0.25ex}
\arraycolsep=0.1em
\left\{\begin{array}{ll}
&\hat{b}_1 = \hat{b}_2 := b, \quad b_1 = b_2 := \frac{c_0b}{\gamma^2}, \quad \hat{b}_0 := c_1^2[b(T+1)]^{1/3}, \quad  b_0 := \frac{c_0\hat{b}_0}{\gamma^2}, \quad \gamma_t := \gamma \in (0,1], \vspace{1ex}\\
&\beta_t = \hat{\beta}_t = 1 - \frac{b^{1/2}}{[\hat{b}_0(T+1)]^{1/2}},  \quad \theta_t =  \theta :=  \frac{L_{\Phi_{\gamma}} b^{3/4}}{P [\hat{b}_0(T+1)]^{1/4}}, \quad\text{and}\quad \eta_t  = \eta := \frac{2}{L_{\Phi_{\gamma}}(3 + \theta)}.
\end{array}\right.
\hspace{-3ex}
\end{equation}
Then, with $B_{\psi}$ defined in Lemma~\ref{le:properties_of_phi},  the following bound holds
\begin{equation}\label{eq:key_est5}
\Exp{\norms{\Gc_{\eta}(\bar{x}_T)}^2} \leq \frac{\hat{\Delta}_0}{[b(T+1)]^{2/3}}, 
\quad\text{where}\quad 
\hat{\Delta}_0 := 16\sqrt{c_1}P\big( \Psi_{0}(x_0) \! - \! \Psi^{\star}_0 \! + \! B_{\psi}\big) + \frac{24Q}{c_1}.
\end{equation}
The total number of iterations $T$ to achieve $\Exp{\norms{\Gc_{\eta}(\bar{x}_T)}^2}\leq\varepsilon^2$ is at most $T := \big\lfloor\frac{\hat{\Delta}_0^{3/2}}{b\varepsilon^3}\big\rfloor$.
The total numbers of function evaluations $\Tc_F$ and Jacobian evaluations $\Tc_J$ are respectively at most
\begin{equation*}
\begin{array}{l}
\Tc_F := \frac{c_0\hat{\Delta}_0^{1/2}}{\gamma^2\varepsilon} + \frac{3c_0\hat{\Delta}_0^{3/2}}{\gamma^2\varepsilon^3} = \mathcal{O}\Big( \frac{\hat{\Delta}_0^{3/2}}{\gamma^2\varepsilon^{3}} \Big)
\quad\text{and}\quad
\Tc_J :=  \frac{\hat{\Delta}_0^{1/2}}{\varepsilon} + \frac{3\hat{\Delta}_0^{3/2}}{\varepsilon^{3}} = \BigO{\frac{\hat{\Delta}_0^{1.5}}{\varepsilon^{3}}}.
\end{array}
\end{equation*}
If we choose $\gamma := c_2\varepsilon$ for some $c_2 > 0$, then $\Tc_F =  \frac{c_0\hat{\Delta}_0^{1/2}}{c_2^2\varepsilon^3} + \frac{3c_0\hat{\Delta}_0^{3/2}}{c_2^2\varepsilon^5}  = \mathcal{O}\big( \frac{\hat{\Delta}_0^{3/2}}{\varepsilon^{5}} \big)$.
\end{theorem}
%%% End of Theorem 3.3.

Alternatively, we can also establish convergence and estimate the complexity of Algorithm~\ref{alg:A1} with diminishing step-size in Theorem~\ref{th:nonsmooth_diminishing}, whose proof is in Supp. Doc.~\ref{apdx:subsec:th:convergence2_diminishing}.

%%% Theorem 3.4.
\begin{theorem}[\textbf{Diminishing step-size}]\label{th:nonsmooth_diminishing}
Suppose that Assumptions~\ref{ass:A1} and \ref{ass:A2} hold, $\psi$ is non-strongly convex $($i.e., $\phi_0$ is possibly nonsmooth$)$, and $P$, $Q$, $L_{\Phi_{\gamma_t}}$ are defined by \eqref{eq:constant_defs}.
Given mini-batch sizes $b > 0$ and $\hat{b}_0 > 0$, let $b_0 := \frac{c_0\hat{b}_0}{\gamma_0^2}$, $b_1^t = b_2^t := \frac{c_0b}{\gamma_{t}^2}$, and $\hat{b}_1 = \hat{b}_2 := b$ for some $c_0 > 0$.
Let $\sets{x_t}_{t=0}^T$ be generated by Algorithm~\ref{alg:A1} using increasing weight and diminishing step-sizes as
\begin{equation}\label{eq:para_config0_a2}
\begin{array}{l}
\gamma_t := \frac{1}{(t+2)^{1/3}}, \quad
\beta_t = \hat{\beta}_t := 1 - \frac{1}{(t+2)^{2/3}},\quad 
\theta_t :=  \frac{L_{\Phi_{\gamma_t}}b^{1/2}}{P(t+2)^{1/3}},
\quad\text{and}\quad \eta_t :=  \frac{2}{L_{\Phi_{\gamma_t}}(3 + \theta_t)}.
\end{array}
\hspace{-2ex}
\end{equation}
For $(\bar{x}_T,\bar{\eta}_T)$ chosen as $\mathbf{Prob}\big\{ \Gc_{\bar{\eta}_T}(\bar{x}_T) \!=\! \Gc_{\eta_t}(x_t) \big\} \!=\! \big[\sum_{t=0}^T (\theta_t/L_{\Phi_{\gamma_t}})\big]^{-1}(\theta_t/L_{\Phi_{\gamma_t}})$, we have{\!\!}
\begin{equation}\label{eq:convergence_rate1_b_a2} 
\hspace{0ex}
\begin{array}{lcl}
\Exp{\norms{\Gc_{\bar{\eta}_T}(\bar{x}_T)}^2}  & \leq & \frac{32P}{3\sqrt{b}[(T+3)^{2/3} - 2^{2/3}]}\big( \Psi_0(x_0) - \Psi^{\star}_0  + \frac{B_{\psi}}{(T+2)^{1/3}} \big) \vspace{1ex}\\
&& + {~}   \frac{16Q}{3[(T+3)^{2/3} - 2^{2/3}]}\Big(\frac{2^{1/3}}{\hat{b}_0} + \frac{2(1 + \log(T+1))}{b} \Big) = \BigO{\frac{\log(T)}{T^{2/3}}}.
\end{array}
\hspace{-2ex}
\end{equation}
\end{theorem}
%%% End of Theorem 3.4.

Note that since $\gamma_t := \frac{1}{(t+2)^{1/3}}$ (diminishing) and $b_1^t = b_2^t := \frac{c_0b}{\gamma_{t}^2}$, we have $b_1^t = b_2^t = c_0b(t+2)^{2/3}$, which shows that the mini-batch sizes of the function estimation $\tilde{F}_t$ are chosen in increasing manner (not fixed at a large size for all $t$), which can save computational cost for $F$.
The batch sizes $b$ and $\hat{b}_0$ in Theorems~\ref{th:convergence2} and \ref{th:nonsmooth_diminishing} must be chosen to guarantee $\beta_t, \theta_t \in (0, 1]$.

%%%% 2.3. Construction of approximate KKT points.
\beforesubsec
\subsection{Constructing approximate KKT point for \eqref{eq:min_max_form} from Algorithm~\ref{alg:A1}}
\aftersubsec
Existing works such as \cite{luo2020stochastic,zhang2019multi,zhang2020stochastic} do not show how to construct an $\epsilon$-KKT point of \eqref{eq:min_max_form}  or an $\epsilon$-stationary point of \eqref{eq:com_nlp} from $\bar{x}_T$ with $\Exp{\norms{\Gc_{\bar{\eta}_T}(\bar{x}_T)}^2} \leq\varepsilon^2$.
Lemma~\ref{le:approx_KKT0}, whose proof is in Supp. Doc.~\ref{apdx:le:approx_KKT0}, shows one way to construct an $\epsilon$-KKT point of \eqref{eq:min_max_form} in the sense of Definition~\ref{de:approx_KKT_point} with $\epsilon := \BigO{\varepsilon}$ from the output $\bar{x}_T$ of Theorems~\ref{th:convergence2_scvx}, \ref{th:convergence2_scvx_diminishing}, \ref{th:convergence2}, and \ref{th:nonsmooth_diminishing}.

%%% Lemma A.2.
\begin{lemma}\label{le:approx_KKT0}
Let $\bar{x}_T$ be computed by Algorithm~\ref{alg:A1} up to an accuracy $\varepsilon > 0$ after $T$ iterations.
Assume that we can approximate $F'(\bar{x}_T)$, $F(\bar{x}_T)$, and $F(\tilde{x}^{*}_{\gamma_T})$, respectively such that
\begin{equation}\label{eq:approx_oralce10}
\hspace{-0.5ex}
\arraycolsep=0.2em
\begin{array}{ll}
&\Exp{\norms{\tilde{F}(\bar{x}_T) - F(\bar{x}_T)}} \leq  (\mu_{\psi}+\gamma_T)\varepsilon, \quad \Exp{\norms{(\tilde{J}(\bar{x}_T) - F'(\bar{x}_T))^{\top}\nabla{\phi_{\gamma_T}}(\tilde{F}(\bar{x}_T))}} \leq \varepsilon,\vspace{1ex}\\
\text{and}~&\Exp{\norms{\tilde{F}(\tilde{x}^{*}_{\gamma_T}) - F(\tilde{x}^{*}_{\gamma_T})}} \leq \varepsilon.
\end{array}
\hspace{-3ex}
\end{equation}
Let us denote $\widetilde{\nabla}{\Phi}_{\gamma_T}(\bar{x}_T) := \tilde{J}(\bar{x}_T)^{\top}\nabla{\phi_{\gamma}}(\tilde{F}(\bar{x}_T))$ and compute $(\tilde{x}^{*}_{\gamma_T}, \tilde{y}^{*}_{\gamma_T})$ as
\begin{equation}\label{eq:KKT_point_t}
\tilde{x}^{*}_{\gamma_T} := \prox_{\bar{\eta}_T\Rc}(\bar{x}_T - \bar{\eta}_T\widetilde{\nabla}{\Phi}_{\gamma_T}(\bar{x}_T)) \quad\text{and}\quad \tilde{y}^{*}_{\gamma_T} := y^{*}_{\gamma_T}(\tilde{F}(\tilde{x}^{*}_{\gamma_T}))~\text{by \eqref{eq:smoothed_phi}}.
\end{equation}
Suppose that $\Exp{\norms{\Gc_{\bar{\eta}_T}(\bar{x}_T)}^2} \leq\varepsilon^2$ and $0 \leq \gamma_T \leq c_2\varepsilon$ for a constant $c_2 \geq 0$.
Then
\begin{equation}\label{eq:approx_KKT2_main}
\Exp{\Ec(\tilde{x}^{*}_{\gamma_T}, \tilde{y}^{*}_{\gamma_T})} \leq \epsilon, \quad\text{where}\quad \epsilon := \big[\tfrac{13}{3} + \tfrac{8}{3}M_F\norms{K}^2 + c_2D_{\psi}\big]\varepsilon,
\end{equation}
where $D_{\psi}$ is  in Lemma~\ref{le:properties_of_phi} and  $\Ec(\cdot)$ is given by \eqref{eq:approx_KKT_point}.
In other words, $(\tilde{x}^{*}_{\gamma_T}, \tilde{y}^{*}_{\gamma_T})$ is an $\epsilon$-KKT of \eqref{eq:min_max_form}.
\end{lemma}

If we use stochastic estimators as in \eqref{eq:est_snap_def} to form $\tilde{F}(\bar{x}_T)$ , $\tilde{J}(\bar{x}_T)$, and $\tilde{F}(\tilde{x}^{*}_{\gamma_T})$ with batch sizes $b_T$, $\hat{b}_T$, and $\tilde{b}_T$, respectively, then \eqref{eq:approx_oralce10} holds if we choose $b_T {\!}:= {\!} \rounds{\frac{\sigma_F^2}{ (\mu_{\psi}+\gamma_T)^2\varepsilon^2}}$,  $\hat{b}_T := \rounds{\frac{\sigma_J^2}{\varepsilon^2}}$, and $\tilde{b}_T := \rounds{\frac{\sigma_F^2}{\varepsilon^2}}$.
We do not explicitly compute Jacobian $\tilde{J}(\bar{x}_T)$, but its matrix-vector product $\tilde{J}(\bar{x}_T)^{\top}\nabla{\phi_{\gamma_T}}(\tilde{F}(\bar{x}_T))$.
This extra cost is dominated by $\Tc_J$ and $\Tc_F$ in Theorems~\ref{th:convergence2_scvx}, \ref{th:convergence2_scvx_diminishing}, \ref{th:convergence2}, and \ref{th:nonsmooth_diminishing}.
For $\bar{x}_T$ computed by Theorems~\ref{th:convergence2_scvx} and \ref{th:convergence2_scvx_diminishing}, we can set $\gamma_T := 0$, or equivalently, $c_2 := 0$.
For $\bar{x}_T$ computed by Theorem~ \ref{th:convergence2}, since $\gamma_T := c_2\varepsilon$ and $\mu_{\psi} = 0$, we have $b_T = \rounds{\frac{\sigma_F^2}{ c_2^2\varepsilon^4}} < \Tc_F = \mathcal{O}\big( \frac{\hat{\Delta}_0^{3/2}}{\varepsilon^{5}} \big)$.

%%%%%%%%%%%%%%%%%%%%%%%%%%%%%%%%%%%%%%%%%%
%%% 4. Numerical Experiments.
%%%%%%%%%%%%%%%%%%%%%%%%%%%%%%%%%%%%%%%%%%
\beforesec
\section{Numerical experiments}\label{sec:num_exps}
\aftersec
We use two examples to illustrate our algorithm and compare it with existing methods. Our code is implemented in Python 3.6.3, running on  a Linux desktop (3.6GHz Intel Core i7 and 16Gb memory).

\beforesubsec
\subsection{Risk-averse portfolio optimization}\label{subsec:example1}
\aftersubsec
We consider a risk-averse portfolio optimization problem studied in \cite{Markowitz1952}, and recent used in \cite{zhang2019stochastic}:
\begin{equation}\label{eq:portfolio_exam}
\max_{x\in\R^p} \Big\{\Exps{\xi}{h_{\xi}(x)} - \rho\Vars{\xi}{h_{\xi}(x)} \equiv \Exps{\xi}{h_{\xi}(x)} + \rho\Exps{\xi}{h_{\xi}(x)}^2 - \rho\Exps{\xi}{h_{\xi}^2(x)}\Big\},
\end{equation} 
where $\rho > 0$ is a trade-off parameter and $h_{\xi}(x)$ is a reward for the portfolio vector $x$.
Following \cite{zhang2019stochastic}, \eqref{eq:portfolio_exam} can be reformulated into \eqref{eq:com_nlp}, where $\phi_0(u) = u_1 + \rho u_1^2 - \rho u_2$ is smooth, and $\Fb(x,\xi) = (h_{\xi}(x), h_{\xi}^2(x))^{\top}$.
Suppose further that we only consider $N$ periods of time. 
Then we can view $\xi \in \set{1,\cdots, N}$ as a discrete random variable and define $h_{i}(x) := \iprods{r_i, x}$ as a linear reward function, where 
$r_i := (r_{i1},\cdots,r_{ip})^{\top}$ and $r_{ij}$ represents the return per unit of $j$ at time $i$.
We also choose $\Rc(x) := \lambda\norms{x}_1$ as a regularizer to promote sparsity as in \cite{zhang2019stochastic}.

\begin{figure}[ht!]
\begin{center}
\includegraphics[width=1\textwidth]{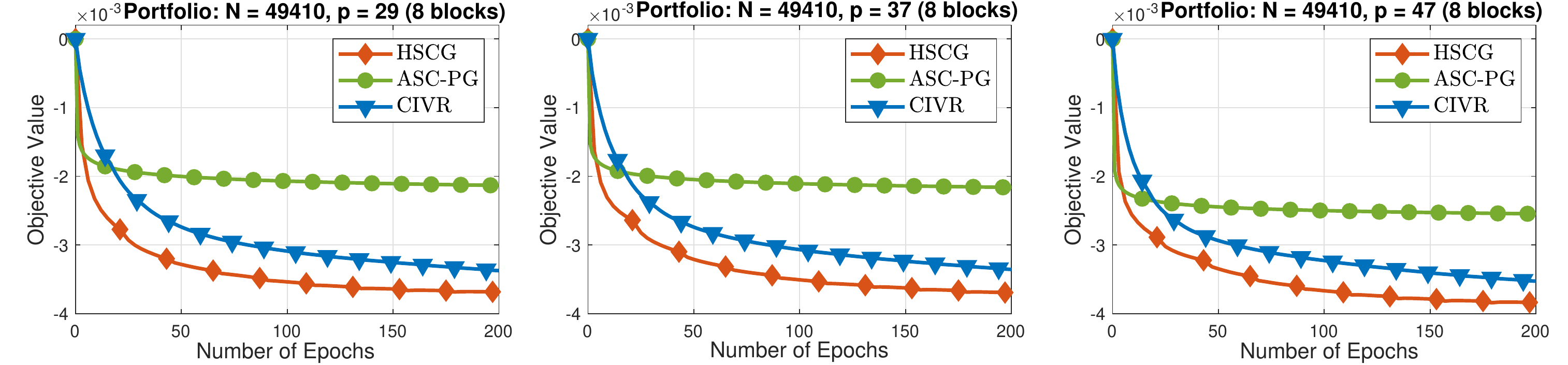}
\caption{
Comparison of three algorithms for solving \eqref{eq:portfolio_exam} on $3$ different datasets.
}
\label{fig:risk_averse}
\end{center}
\end{figure}

We implement our algorithm, abbreviated by \texttt{HSCG} (i.e., Hybrid Stochastic Compositional Gradient for short), and test it on three real-world portfolio datasets, which contain $29$, $37$, and $47$ portfolios, respectively, from the Keneth R. French Data Library \cite{PortfolioData2020}.
We set $\rho := 0.2$ and $\lambda := 0.01$ as in \cite{zhang2019stochastic}.
For comparison, we also implement 2 methods, called \texttt{CIVR} in \cite{zhang2019stochastic} and \texttt{ASC-PG} in \cite{wang2017accelerating}.
The step-size $\eta$ of all algorithms are well tuned from a set of trials $\set{1, 0.5, 0.1, 0.05, 0.01, 0.001,0.0001}$.
The performance of 3 algorithms are shown in Figure \ref{fig:risk_averse} for three datasets using $b:=\rounds{N/8}$ (8 blocks). 

One can observe from Fig. \ref{fig:risk_averse} that both \texttt{HSCG} and  \texttt{CIVR} highly outperform \texttt{ASC-PG} due to their variance-reduced property. 
\texttt{HSCG} is slightly better than  \texttt{CIVR} since it has a flexible step-size $\theta_t$. 
Note that, in theory, \texttt{CIVR} requires a large batch for both function values and Jacobian, which may affect its performance, while \texttt{HSCG} can work with a wide range of batches, including singe sample.

\beforesubsec
\subsection{Stochastic minimax problem}\label{subsec:example2}
\aftersubsec
We consider the following regularized stochastic minimax problem studied, e.g., in \cite{shapiro2002minimax}:
\begin{equation}\label{eq:min_max_stochastic_opt}
\min_{x\in \R^p}\Big\{\max_{1\leq i\leq m}\{\Exps{\xi}{\Fb_i(x, \xi)}\} + \tfrac{\lambda}{2}\norms{x}^2\Big\},
\end{equation}
where $\Fb_i : \R^p\times\Omega\to\R_{+}$ can be taken as the loss function of the $i$-th model.
If we define $\phi_0(u) := \max_{1\leq i\leq m}\{u_i\}$ and $\Rc(x) := \frac{\lambda}{2}\norms{x}^2$, then \eqref{eq:min_max_stochastic_opt} can be reformulated into \eqref{eq:com_nlp}.
Since $u_i \geq 0$, we have $\phi_0(u) := \max_{1\leq i\leq m}\{u_i\} = \norms{u}_{\infty} = \max_{\norms{y}_1 \leq 1}\{\iprods{u, y}\}$, which is nonsmooth.
Therefore, we can smooth $\phi_0$ as $\phi_{\gamma}(u) := \max_{\norms{y}_1 \leq 1}\{\iprods{u, y} - (\gamma/2)\norms{y}^2\}$ using $b(y) := \frac{1}{2}\norms{y}^2$.

In this example, we employ \eqref{eq:min_max_stochastic_opt} to solve a model selection problem in binary classification with nonconvex loss, see, e.g., \cite{zhao2010convex}.
Suppose that we have four $(m=4)$ different nonconvex losses: $\Fb_1(x, \xi) := 1 - \tanh(b\iprods{a,x})$, $\Fb_2(x, \xi) := \log(1 + \exp(-b\iprods{a,x})) - \log(1 + \exp(-b\iprods{a,x}-1))$,  $\Fb_3(x, \xi) := (1 - 1/(\exp(-b\iprods{a,x})+1))^2$, and $\Fb_4(x, \xi) := \log(1 + \exp(-b\iprods{a,x}))$ (see \cite{zhao2010convex} for more details), where $\xi := (a, b)$ represents examples.
We assume that we have $N$ examples of $\xi$.

\begin{figure}[ht!]
\hspace{-2ex}
\centering
\includegraphics[width=1\textwidth]{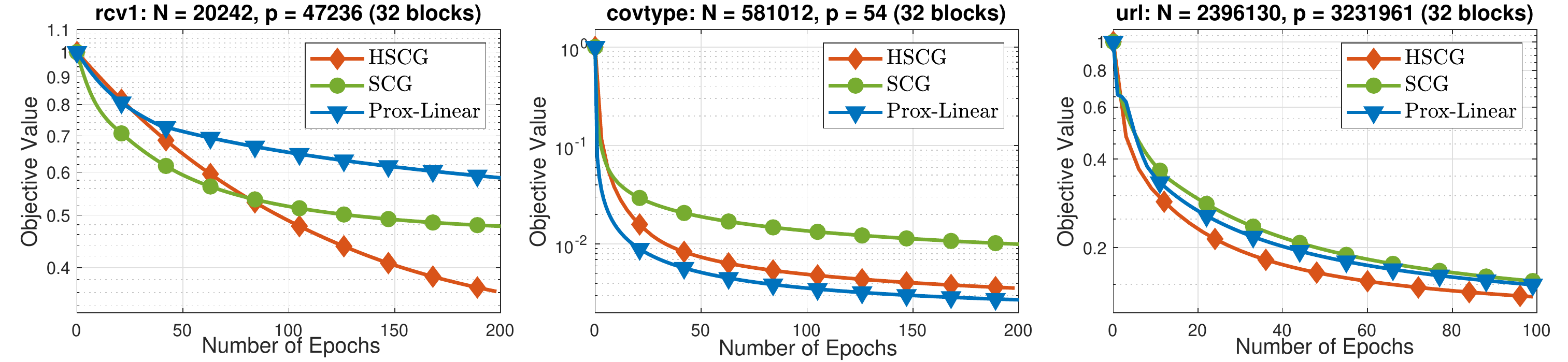}
\hspace{-2ex}
\caption{
Comparison of three algorithms for solving \eqref{eq:min_max_stochastic_opt} on $3$ different datasets.
}
\label{fig:min_max_stochastic_opt}
\end{figure}

We implement three algorithms: \texttt{HSCG}, \texttt{SCG} in \cite{wang2017stochastic}, and \texttt{Prox-Linear} in \cite{zhang2020stochastic}.
We test them on 3 datasets from LIBSVM \cite{CC01a}.
We set $\lambda := 10^{-4}$ and update our $\gamma_t$ parameter as $\gamma_t := \frac{1}{2(t+1)^{1/3}}$. 
The step-size $\eta$ of all algorithms are well tuned from $\{1, 0.5, 0.1, 0.05, 0.01, 0.001,0.0001\}$, and their performance is shown in Figure \ref{fig:min_max_stochastic_opt} for three datasets: \textbf{rcv1}, \textbf{covtype}, and \textbf{url} with $32$ blocks.

One can observe from Figure \ref{fig:min_max_stochastic_opt} that \texttt{HSCG} outperforms \texttt{SCG} and  \texttt{Prox-Linear} on \textbf{rcv1} and \textbf{url}. 
For \textbf{covtype}, since $p$ is very small, allowing us to evaluate the prox-linear operator to a high accuracy, \texttt{Prox-Linear} slightly performs better than ours and much better than \texttt{SCG}. 
Note that solving the subproblem of \texttt{Prox-Linear} is expensive when $p$ is large.
Hence, if $p$ is large, \texttt{Prox-Linear} becomes much slower than  \texttt{HSCG} and \texttt{SCG} in terms of time.
Due to space limit, we refer to Supp. Doc. \ref{apdx:sec:add_num_exam} for further details of experiments and additional results.

\beforesec
\section{Conclusions}\label{sec:concl}
\aftersec
We have proposed a new single loop hybrid variance-reduced SGD algorithm, Algorithm~\ref{alg:A1}, to solve a class of nonconvex-concave saddle-point problems.
The main idea is to combine both smoothing idea \cite{Nesterov2005c} and hybrid SGD approach in \cite{Tran-Dinh2019a} to develop novel algorithms with less tuning effort.
Our algorithm relies on standard assumptions, and can achieve the best-known oracle complexity, and in some cases, the optimal oracle complexity. 
It also has several computational advantages compared to existing methods such as avoiding expensive subproblems, working with both single sample and mini-batches, and using constant and diminishing step-sizes.
We have also proposed a simple restarting variant, Algorithm~\ref{alg:A2}, in Supp. Doc. \ref{apdx:sec:restarting_hSGD} to improve practical performance in the constant step-size case without sacrificing complexity bounds.
We believe that both algorithms and theoretical results are new, even in the smooth case, compared to \cite{tran2020stochastic,zhang2019multi,zhang2020stochastic}.
%Our future plan is to exploit this approach to solve some interesting applications, such as robust optimization and learning, and GANs.

%%%%%%%%%%%%%%%%%%%%%%%%%%%%%%%%%%%%%%%
%%% + References.
%%%%%%%%%%%%%%%%%%%%%%%%%%%%%%%%%%%%%%%
\clearpage
\newpage
\beforesec
\section{Broader Impact}\label{sec:broader_impact}
\aftersec
This work could potentially have positive impact in different fields where nonconvex-concave minimax and nonconvex compositional optimization models as \eqref{eq:min_max_form} and \eqref{eq:com_nlp} are used.
For instance, robust learning, distributionally robust optimization, zero-sum game, and generative adversarial nets (GANs) applications are concrete examples under certain settings.
We emphasize that the nonconvex-concave minimax problem \eqref{eq:min_max_form} studied in this paper remains challenging to solve for global solutions.
Existing methods can only find an approximate KKT (Karush-Kuhn-Tucker) point in general.
This paper proposed new algorithms to tackle a class of nonconvex-concave minimax problems, but they can only guarantee to find an approximate KKT point, which may not be an approximate global solution of the model. 
This could lead to a negative impact if one expects to find an approximate global solution instead of an approximate KKT point without further investigation.
Apart from the above impact, since this paper is a theoretical work, it does not present any other foreseeable societal consequence.

\begin{ack}
The work of Quoc Tran-Dinh and Deyi Liu is partially supported by the National Science Foundation (NSF), grant no. DMS-1619884, and the Office of Naval Research (ONR), grant No. N00014-20-1-2088.
The authors would also like to thank all the anonymous reviewers and the ACs for their constructive comments to improve the paper.
%Use unnumbered first level headings for the acknowledgments. All acknowledgments
%go at the end of the paper before the list of references. Moreover, you are required to declare 
%funding (financial activities supporting the submitted work) and competing interests (related financial activities outside the submitted work). 
%More information about this disclosure can be found at: \url{https://neurips.cc/Conferences/2020/PaperInformation/FundingDisclosure}.
%
%
%Do {\bf not} include this section in the anonymized submission, only in the final paper. You can use the \texttt{ack} environment provided in the style file to autmoatically hide this section in the anonymized submission.
\end{ack}

%\clearpage
%\newpage
\bibliographystyle{plain}

\input{neurips2020_refs}
%\bibliography{/Users/quoctd/Dropbox/E-Books/tran_bibtex_new}

%%%%%%%%%%%%%%%%%%%%%%%%%%%%%%%%%%%%%%%
%%% + Supplementary Documents.
%%%%%%%%%%%%%%%%%%%%%%%%%%%%%%%%%%%%%%%
\clearpage
\newpage
\appendix
\begin{center}
\centering{\textbf{\large Supplementary Document}}

\textbf{\Large 
Hybrid Variance-Reduced SGD Algorithms For 
\vspace{0.5ex}\\ 
Minimax Problems with Nonconvex-Linear Function
}

\vspace{2ex}
\textbf{Quoc Tran-Dinh$^{*}$ \quad\qquad Deyi Liu$^{*}$ \quad\qquad Lam M. Nguyen$^{\dagger}$}\vspace{1ex}\\
  $^{*}$Department of Statistics and Operations Research\\
  The University of North Carolina at Chapel Hill, Chapel Hill, NC 27599 \\
  Emails: \texttt{\{quoctd@email.unc.edu, deyi.liu@live.unc.edu\}} \vspace{0.25ex}\\
$^{\dagger}$IBM Research, Thomas J. Watson Research Center \\ Yorktown Heights, NY10598, USA.\\
Email: \texttt{lamnguyen.mltd@ibm.com}
\end{center}
\vspace{0ex}

%% A. Some Technical Results and Proof of Lemma 2.1.
\beforesec
\section{Some technical results and proof of Lemma~\ref{le:approx_KKT0}}\label{app:A}
\aftersec
In this Supp. Doc., we provide some useful properties of $\phi_0$ in \eqref{eq:phi_func} and its smoothed approximation $\phi_{\gamma}$ defined by \eqref{eq:smoothed_phi} in Section~\ref{sec:math_tools}.
Then we recall and prove some bounds of variance for $\tilde{F}_t$, $\tilde{J}_t$, and $v_t$.
Finally, we prove Lemma~\ref{le:approx_KKT0} in the main text.

\beforesubsec
\subsection{Properties of the smoothed function $\phi_{\gamma}$}\label{apdx:subsec:smooth_properties}
\aftersubsec
Under Assumption~\ref{ass:A2}, $\phi_0$ in \eqref{eq:phi_func} and $\phi_{\gamma}$ defined by \eqref{eq:smoothed_phi} have the following properties. 

%%% Lemma 1.
\begin{lemma}\label{le:properties_of_phi}
Let $\phi_0$ be defined by \eqref{eq:phi_func} and $\phi_{\gamma}$ be defined by \eqref{eq:smoothed_phi}. 
Then, the following statements hold:
\begin{compactitem}
\item[$\mathrm{(a)}$] $\dom{\psi}$ is bounded by $M_{\psi}$ iff $\phi_0$ is $M_{\phi_0}$-Lipschitz continuous with $M_{\phi_0} := M_{\psi}\norms{K}$.
\item[$\mathrm{(b)}$] $\dom{\psi}$ is bounded by $M_{\psi}$ iff $\phi_{\gamma}$ is Lipschitz continuous with $M_{\phi_{\gamma}} := M_{\psi}\norms{K}$.
\item[$\mathrm{(c)}$] $\phi_{\gamma}$ is convex and $L_{\phi_{\gamma}}$-smooth with $L_{\phi_{\gamma}} := \frac{\norms{K}^2}{\gamma + \mu_{\psi}}$.
\item[$\mathrm{(d)}$] It holds that $\phi_{\gamma}(u) \leq \phi_{0}(u) \leq \phi_{\gamma}(u) + \gamma B_{\psi}$ for all $u\in\R^q$, where $\gamma > 0$ and $B_{\psi} := \sup\set{ b(y) \mid y\in\dom{\psi}}$.
In addition, we have $D_{\psi} := \max_{v\in\dom{\psi}}\norms{\nabla{b}(v)} <+\infty$.
\item[$\mathrm{(e)}$] We have $\phi_{\gamma}(u) \leq \phi_{\hat{\gamma}}(u) + (\hat{\gamma} - \gamma)b(y_{\gamma}^{*}(u)) \leq \phi_{\hat{\gamma}}(u) + (\hat{\gamma} - \gamma)B_{\psi}$ for all $\hat{\gamma} \geq \gamma > 0$.
\end{compactitem}
\end{lemma}

%% Proof of Lemma 1.
\begin{proof}
The statement (a) can be found in \cite[Corollary 17.19]{Bauschke2011}.

Since $\nabla{\phi}_{\gamma}(u) = Ky_{\gamma}^{*}(u)$ with $y^{*}_{\gamma}(u)\in\dom{\psi}$, we have $\norms{\nabla{\phi}_{\gamma}(u)} \leq \norms{K}\norms{y^{*}_{\gamma}(u)} \leq M_{\psi}\norms{K}$.
Applying again \cite[Corollary 17.19]{Bauschke2011} we prove (b).

The statement (c) holds due to the well-known Baillon-Haddad theorem \cite[Corollary 18.17]{Bauschke2011}.

The proof of the first part of (d) can be found in \cite{Nesterov2005c}.
Under Assumption~\ref{ass:A2} and the continuous differentiability of $b$, we have $D_{\psi} := \max_{v\in\dom{\psi}}\norms{\nabla{b}(v)} <+\infty$.
%To prove the second part, from the optimality condition $0 \in K^{\top}u - \partial{\psi}(y^{*}_{\gamma}(u)) - \gamma \nabla{b}(y^{*}_{\gamma}(u))$ of  \eqref{eq:smoothed_phi}, we have \deyi{missing something?}

Finally, for any $u$ and $y$, since $s(\gamma;u, y) :=\iprods{u, Ky} - \psi(y) - \gamma b(y)$ is linear in $\gamma$.
Therefore, $\phi_{\gamma}(u) := \max_{y\in\R^n}s(\gamma; u, y)$ is convex in $\gamma$ and $\frac{d}{d\gamma}\phi_{\gamma}(u) = -b(y^{*}_{\gamma}(u)) \leq 0$.
Consequently, we have $\phi_{\gamma}(u) + \frac{d}{d\gamma}\phi_{\gamma}(u)(\hat{\gamma} - \gamma) = \phi_{\gamma}(u) - (\hat{\gamma} - \gamma)b(y^{*}_{\gamma}(u)) \leq \phi_{\hat{\gamma}}(u)$, which implies (e).
\end{proof}
%% End of the proof.

One common example of $\psi$ in Assumption~\ref{ass:A2} is $\psi(x) := \delta_{\Xc}(x)$, the indicator of a nonempty, closed, bounded, and convex set $\Xc$.
For instance, $\Xc := \set{ y \in\R^n \mid \norms{y}_{*} \leq 1}$ is a unit ball in the dual norm $\norm{\cdot}_{*}$ of $\norm{\cdot}$.
Then, we have $\phi_0(u) := \norm{u}$, which is clearly Lipschitz continuous. 
In particular, if $\Xc := \set{ y \in\R^n \mid \norms{y}_{\infty} \leq 1}$, then $\phi_0(u) := \norm{u}_1$.

%%%%% A.1. Variance bounds.
\beforesubsec
\subsection{Key bounds on the variance of estimators}\label{apdx:subsec:variance_bounds}
\aftersubsec
Next, we provide some useful bounds for the estimators $\tilde{F}_t$ and $\tilde{J}_t$ defined in \eqref{eq:est_update}.
The following lemma can be found in \cite{Tran-Dinh2019a}, where we have used the  inequality $2\Exp{\iprods{a, b}} \leq \Exp{\norms{a}^2} + \Exp{\norms{b}^2}$ in the proof, when $a$ and $b$ are not independent.

%%% Lemma 1.
\begin{lemma}\label{lem:F_J_est_var}
Let $\tilde{F}_t$ and $\tilde{J}_t$ be defined by \eqref{eq:est_update}, and $\Fc_t$ be defined by \eqref{eq:filtration}.
Then
\begin{equation}\label{eq:var_est1}
\arraycolsep=0.2em
\begin{array}{lcl}
\Exps{(\Bc_t^1,\Bc_t^2)}{\norms{\tilde{F}_t - F(x_t)}^2}  &\leq & \beta_{t-1}^2 \norms{\tilde{F}_{t-1} - F(x_{t-1})}^2 - \beta_{t-1}^2 \norms{ F(x_t) - F(x_{t-1}) }^2 \vspace{1ex}\\
& &+ {~} \kappa(1-\beta_{t-1})^2 \Exps{\Bc_t^2}{\norms{\Fb(x_t,\zeta_t) - F(x_t)}^2} \vspace{1ex}\\
&& + {~} \frac{\kappa\beta_{t-1}^2}{b_1} \Exps{\xi}{\norms{\Fb(x_t,\xi) - \Fb(x_{t-1},\xi)}^2}, \vspace{1ex}\\ 
\Exps{(\hat{\Bc}_t^1,\hat{\Bc}_t^2)}{\norms{\tilde{J}_t - F'(x_t)}^2}  &\leq & \hat{\beta}_{t-1}^2 \norms{\tilde{J}_{t-1} - F'(x_{t-1})}^2 \vspace{1ex}\\ %- \hat{\beta}_{t-1}^2\norms{F'(x_t) - F'(x_{t-1})}^2 \vspace{1ex}\\
&& + {~} \hat{\kappa}(1-\hat{\beta}_{t-1})^2 \Exps{\hat{\Bc}^2_t}{\norms{\Fb'(x_t,\hat{\zeta}_t) - F'(x_t)}^2 } \vspace{1ex}\\
&& + {~} \frac{\hat{\kappa}\hat{\beta}_{t-1}^2}{\hat{b}_1} \Exps{\hat{\xi}}{\norms{\Fb'(x_t,\hat{\xi}) - \Fb'(x_{t-1},\hat{\xi})}^2}.
\end{array}
\end{equation}
Here, $\kappa = 1$ if $\Bc_t^1$ is independent of $\Bc_t^2$, and $\kappa = 2$, otherwise.
Similarly, $\hat{\kappa} = 1$ if $\hat{\Bc}_t^1$ is independent of $\hat{\Bc}_t^2$, and $\hat{\kappa} = 2$, otherwise.
\end{lemma}

Furthermore, we can bound the variance of the estimator  $v_t$ of $\nabla{\Phi}_{\gamma_t}(x_t)$ defined in \eqref{eq:grad_estimators} as follows.

%%% Lemma 3.3.
\begin{lemma}\label{lem:vt_var}
Let $\Phi_{\gamma}$ and $v_t$ be defined by \eqref{eq:smoothed_compositional_func} and \eqref{eq:grad_estimators}, respectively.
Then, under Assumptions~\ref{ass:A1} and \ref{ass:A2}, we have
\begin{equation}\label{eq:vt_var0}
\Exp{\norms{v_t - \nabla{\Phi}_{\gamma_t}(x_t)}^2} \leq 2M_F^2 L_{\phi_{\gamma_t}}^2 \Exp{\norms{\tilde{F}_t - F(x_t)}^2} + 2M_{\phi_{\gamma_t}}^2 \Exp{\norms{\tilde{J}_t - F'(x_t)}^2}.
\end{equation}
\end{lemma}

%%% Proof of Lemma 3.3.
\begin{proof}
First, by the composition rule of derivatives, we can derive
\begin{equation*}
\arraycolsep=0.1em
\hspace{-2ex}\begin{array}{lcl}
\norm{v_t -\nabla{\Phi}_{\gamma_t}(x_t)}^2 &= & \norms{\tilde{J}_t^{\top} \nabla{\phi}_{\gamma_t}(\tilde{F}_t) - F'(x_t)^{\top} \nabla{\phi}_{\gamma_t}(F(x_t))}^2 \vspace{1ex}\\
&=& \big\Vert \tilde{J}_t^{\top} \nabla{\phi}_{\gamma_t}(\tilde{F}_t) - F'(x_t)^{\top} \nabla{\phi}_{\gamma_t}(\tilde{F}_t) + F'(x_t)^{\top} \nabla{\phi}_{\gamma_t}(\tilde{F}_t) \vspace{1ex}\\
&& \quad - {~} F'(x_t)^{\top} \nabla{\phi}_{\gamma_t}(F(x_t)) \big\Vert^2 \vspace{1ex}\\
&\overset{\tiny(i)}{\leq} & 2\norms{ (\tilde{J}_t - F'(x_t))^{\top} \nabla{\phi}_{\gamma_t}(\tilde{F}_t) }^2  + 2\norms{ F'(x_t)^{\top}\big( \nabla{\phi}_{\gamma_t}(\tilde{F}_t) - \nabla{\phi}_{\gamma_t}(F(x_t) \big) }^2 \vspace{1ex}\\
&\leq& 2\norms{\nabla{\phi}_{\gamma_t}(\tilde{F}_t)}^2\norms{\tilde{J}_t - F'(x_t)}^2  + 2 \norms{\nabla{\phi}_{\gamma_t}(\tilde{F}_t) - \nabla{\phi}_{\gamma_t}(F(x_t))}^2\norms{F'(x_t)}^2 \vspace{1ex}\\
&\overset{\tiny(ii)}{\leq}&  2M_{\phi_{\gamma_t}}^2\norms{\tilde{J}_t - F'(x_t)}^2 + 2L_{\phi_{\gamma_t}}^2M_F^2  \norms{\tilde{F}_t - F(x_t)}^2.
\end{array}
\end{equation*}
Here, we use $\norms{a+b}^2 \leq 2\norms{a}^2 + 2\norms{b}^2$ in \emph{(i)} and the $M_{\phi_{\gamma_t}}$-Lipschitz continuity, $L_{\phi_{\gamma_t}}$-smoothness of $\phi_{\gamma_t}$, and \eqref{eq:F_Lipschitz} in \emph{(ii)}.
Taking expectation over $\Fc_{t+1}$ on both sides the last inequality, we obtain
\begin{equation*} 
\Exp{\norms{v_t - \nabla{\Phi}_{\gamma_t}(x_t)}^2} \leq 2M_F^2 L_{\phi_{\gamma_t}}^2 \Exp{\norms{\tilde{F}_t - F(x_t)}^2} + 2M_{\phi_{\gamma_t}}^2 \Exp{\norms{\tilde{J}_t - F'(x_t)}^2},
\end{equation*}
which proves \eqref{eq:vt_var0}.
\end{proof}
%%% End of proof.

%%%% A.2. The construction of approximate KKT points.
\beforesubsec
\subsection{The construction of approximate KKT points for \eqref{eq:min_max_form}}\label{apdx:le:approx_KKT0}
\aftersubsec
Recall from \eqref{eq:smoothed_compositional_func} that $\Phi_{\gamma}(x) = \phi_{\gamma}(F(x))$ and $\nabla{\Phi}_{\gamma}(x) = F'(x)^{\top}\nabla{\phi_{\gamma}}(F(x))$, where $\phi_{\gamma}$ is defined by \eqref{eq:smoothed_phi}.
We define a smoothed approximation problem of \eqref{eq:com_nlp} as follows:
\begin{equation}\label{eq:smooth_com_nlp}
\min_{x\in\R^p}\Big\{ \Psi_{\gamma}(x) := \Phi_{\gamma}(x) + \Rc(x) \equiv \phi_{\gamma}(F(x)) + \Rc(x)\Big\}.
\end{equation}
Clearly, if $\gamma = 0$, then \eqref{eq:smooth_com_nlp} reduces to \eqref{eq:com_nlp}. The optimality condition of \eqref{eq:smooth_com_nlp} becomes
\begin{equation}\label{eq:smoothed_opt_cond}
0 \in \nabla{\Phi}_{\gamma}(x^{\star}_{\gamma}) + \partial{\Rc}(x^{\star}_{\gamma}) \equiv F'(x^{\star}_{\gamma})^{\top}\nabla{\phi_{\gamma}}(F(x^{\star}_{\gamma})) + \partial{\Rc}(x^{\star}_{\gamma}).
\end{equation}
Here, $x^{\star}_{\gamma}$ is called a stationary point of \eqref{eq:smooth_com_nlp}.
Therefore, an $\varepsilon$-stationary point $\tilde{x}^{*}_{\gamma}$ is defined as 
\begin{equation}\label{eq:app_smoothed_opt_cond}
\Exp{\dist{0, \nabla{\Phi}_{\gamma}(\tilde{x}^{*}_{\gamma}) + \partial{\Rc}(\tilde{x}^{*}_{\gamma})}} \leq \varepsilon.
\end{equation}
Again, the expectation $\Exp{\cdot}$ is taken over all the randomness generated by the model \eqref{eq:smooth_com_nlp} and the algorithm for finding $\tilde{x}^{*}_{\gamma}$.

Alternatively, using the definition of $\phi_{\gamma}$ in \eqref{eq:smoothed_phi}, problem \eqref{eq:smooth_com_nlp} can be written as
\begin{equation}\label{eq:smooth_min_max}
\min_{x\in\R^p}\max_{y\in\R^n}\Big\{ \Rc(x) + \iprods{F(x), Ky} - \psi(y) - \gamma b(y) \Big\}.
\end{equation}
Its optimality condition becomes
\begin{equation}\label{eq:smoothed_minmax_opt}
0 \in \partial{\Rc}(x^{\star}_{\gamma}) + F'(x^{\star}_{\gamma})Ky^{\star}_{\gamma} \quad\text{and}\quad  0 \in K^{\top}F(x^{\star}_{\gamma}) - \partial{\psi}(y^{\star}_{\gamma}) - \gamma\nabla{b}(y^{\star}_{\gamma}).
\end{equation}
Using the definition of $\Ec$ in \eqref{eq:approx_KKT_point}, we have
\begin{equation}\label{eq:e_approx}
\Ec(x^{\star}_{\gamma}, y^{\star}_{\gamma}) := \dist{0, \partial{\Rc}(x^{\star}_{\gamma}) + F'(x^{\star}_{\gamma})Ky^{\star}_{\gamma}} + \dist{0, K^{\top}F(x^{\star}_{\gamma}) - \partial{\psi}(y^{\star}_{\gamma})} \leq \gamma D_{\psi}.
\end{equation}
Here, we use the fact that $\norms{\nabla{b}(y^{\star}_{\gamma})} \leq D_{\psi}$ as stated in Lemma~\ref{le:properties_of_phi}.

Given $\bar{x} \in \dom{\Psi_0}$, let  $\tilde{F}(\cdot)$ and $\tilde{J}(\cdot)$ be a stochastic approximation of $F(\cdot)$ and $F'(\cdot)$, respectively.
We define $(\tilde{x}^{*}_{\gamma}, y^{*}_{\gamma})$ as follows:
\begin{equation}\label{eq:approx_point}
\left\{\begin{array}{lcl}
\tilde{x}^{*}_{\gamma} &:=& \prox_{\eta\Rc}\left(\bar{x} - \eta\widetilde{\nabla}\Phi_{\gamma}(\bar{x})\right), \quad\text{where}\quad \widetilde{\nabla}{\Phi}_{\gamma}(\bar{x}) :=  \tilde{J}(\bar{x})^{\top}\nabla{\phi_{\gamma}}(\tilde{F}(\bar{x})), \vspace{1ex}\\
\tilde{y}^{*}_{\gamma} &:= & y^{*}_{\gamma}(\tilde{F}(\tilde{x}^{*}_{\gamma})) \equiv \argmin_{y\in\R^n}\set{\iprods{K^{\top}\tilde{F}(\tilde{x}^{*}_{\gamma}), y} - \psi(y) - \gamma b(y)},
\end{array}\right.
\end{equation}
Note that $\tilde{x}^{*}_{\gamma}$ only depends on $\bar{x}$, while $\tilde{y}^{*}_{\gamma}$ depends on both $\bar{x}$ and $\tilde{x}^{*}_{\gamma}$.
Hence, we first compute $\tilde{x}^{*}_{\gamma}$ and then compute $\tilde{y}^{*}_{\gamma}$.

The following lemma provides key estimates to prove Lemma~\ref{le:approx_KKT0} in the main text.
%%% Lemma A.2.
\begin{lemma}\label{le:approx_KKT}
Under Assumptions~\ref{ass:A1} and \ref{ass:A2}, for given $\bar{x}$ and $\eta > 0$, $\tilde{x}^{*}_{\gamma}$ defined by \eqref{eq:approx_point} satisfies
\begin{equation}\label{eq:approx_KKT1}
\dist{0, \nabla{\Phi}_{\gamma}(\tilde{x}^{*}_{\gamma}) + \partial{\Rc}(\tilde{x}^{*}_{\gamma})} \leq \left(1 + \eta L_{\Phi_{\gamma}}\right)\norms{\Gc_{\eta}(\bar{x})} + (2 + \eta L_{\Phi_{\gamma}}) \norms{\nabla{\Phi}_{\gamma}(\bar{x}) - \widetilde{\nabla}{\Phi}_{\gamma}(\bar{x})}.
\end{equation}
Let $(\tilde{x}^{*}_{\gamma}, \tilde{y}^{*}_{\gamma})$ be computed by \eqref{eq:approx_point}, and $\Ec(x,y)$ be defined by \eqref{eq:approx_KKT_point}.
Then, we have
\begin{equation}\label{eq:approx_KKT2}
\arraycolsep=0.2em
\begin{array}{lcl}
\Ec(\tilde{x}^{*}_{\gamma}, \tilde{y}^{*}_{\gamma}) &\leq & \left(1 + \eta L_{\Phi_{\gamma}}\right)\norms{\Gc_{\eta}(\bar{x})} + \gamma  D_{\psi} + \norms{K}\norms{F(\tilde{x}^{*}_{\gamma}) - \tilde{F}(\tilde{x}^{*}_{\gamma})} \vspace{1ex}\\
&& + {~} \left(2 + \eta L_{\Phi_{\gamma}}\right)\big[\norms{(\tilde{J}(\bar{x}) - F'(\bar{x}))^{\top}\nabla{\phi_{\gamma}}(\tilde{F}(\bar{x})} + L_{\phi_{\gamma}}M_F\norms{\tilde{F}(\bar{x}) - F(\bar{x})}\big], 
\end{array}
\end{equation}
where  $D_{\psi}$ is defined in Lemma~\ref{le:properties_of_phi}.
\end{lemma}

%%% Proof of Lemma A.2.
\begin{proof}
From \eqref{eq:approx_point}, we have $\bar{x} -  \eta\widetilde{\nabla}{\Phi}_{\gamma}(\bar{x})\in \tilde{x}_{\gamma}^{*} + \eta\partial{\Rc}(\tilde{x}_{\gamma}^{*})$, which is equivalent to
\begin{equation}\label{eq:la2_proof1}
r^{*}_x := \frac{1}{\eta}(\bar{x} - \tilde{x}^{*}_{\gamma}) + \big(\nabla{\Phi}_{\gamma}(\tilde{x}_{\gamma}^{*}) -  \widetilde{\nabla}{\Phi}_{\gamma}(\bar{x})\big) \in \nabla{\Phi}_{\gamma}(\tilde{x}_{\gamma}^{*}) + \partial{\Rc}(\tilde{x}_{\gamma}^{*}).
\end{equation}
We can bound $r^{*}_x$ in \eqref{eq:la2_proof1} as follows:
\begin{equation}\label{eq:la2_proof2}
\begin{array}{lcl}
\norms{r_x^{*}} &\leq & \frac{1}{\eta}\norms{\bar{x} - \tilde{x}^{*}_{\gamma}} + \norms{\nabla{\Phi}_{\gamma}(\tilde{x}_{\gamma}^{*}) -  \nabla{\Phi}_{\gamma}(\bar{x})} + \norms{\nabla{\Phi}_{\gamma}(\bar{x}) - \widetilde{\nabla}{\Phi}_{\gamma}(\bar{x})} \vspace{1ex}\\
&\leq & \frac{1}{\eta}\big(1 + \eta L_{\Phi_{\gamma}}\big)\norms{\tilde{x}^{*}_{\gamma} - \bar{x}} + \norms{\nabla{\Phi}_{\gamma}(\bar{x}) - \widetilde{\nabla}{\Phi}_{\gamma}(\bar{x})}.
\end{array}
\end{equation}
Next, from \eqref{eq:grad_map}, let us define $\bar{x}^{*}_{\gamma} := \bar{x} - \eta\Gc_{\eta}(\bar{x}) = \prox_{\eta\Rc}(\bar{x} - \eta\nabla{\Phi}_{\gamma}(\bar{x}))$.
Then, we have 
\begin{equation}\label{eq:la2_proof3}
\begin{array}{lcl}
\norms{\tilde{x}^{*}_{\gamma} - \bar{x}} &\leq & \norms{\tilde{x}^{*}_{\gamma} - \bar{x}^{*}_{\gamma}} + \norms{\bar{x}^{*}_{\gamma} - \bar{x}} \vspace{1ex}\\
&= & \Vert  \prox_{\eta\Rc}(\bar{x} - \eta\widetilde{\nabla}{\Phi}_{\gamma}(\bar{x})) -  \prox_{\eta\Rc}(\bar{x} - \eta\nabla{\Phi}_{\gamma}(\bar{x}))\Vert + \eta\norms{\Gc_{\eta}(\bar{x})} \vspace{1ex}\\
&\leq& \eta\norms{\widetilde{\nabla}{\Phi}_{\gamma}(\bar{x}) - \nabla{\Phi}_{\gamma}(\bar{x})} + \eta\norms{\Gc_{\eta}(\bar{x})}.
\end{array}
\end{equation}
Substituting this estimate into \eqref{eq:la2_proof2}, we obtain
\begin{equation*} 
\begin{array}{lcl}
\norms{r_x^{*}} &\leq & \big(1 + \eta L_{\Phi_{\gamma}}\big)\norms{\Gc_{\eta}(\bar{x})} + (2 + \eta L_{\Phi_{\gamma}}) \norms{\nabla{\Phi}_{\gamma}(\bar{x}) - \widetilde{\nabla}{\Phi}_{\gamma}(\bar{x})}.
\end{array}
\end{equation*}
Combining this inequality and \eqref{eq:la2_proof1}, we obtain \eqref{eq:approx_KKT1}.

Now, since $\tilde{y}^{*}_{\gamma} = y^{*}_{\gamma}(\tilde{F}(\tilde{x}^{*}_{\gamma}))$, by the optimality condition of \eqref{eq:smoothed_phi}, we have
\begin{equation}\label{eq:la2_proof5}
r^{*}_y := \gamma\nabla{b}(\tilde{y}^{*}_{\gamma}) + K^{\top}(F(\tilde{x}^{*}_{\gamma}) - \tilde{F}(\tilde{x}^{*}_{\gamma})) \in K^{\top}F(\tilde{x}^{*}_{\gamma}) - \partial{\psi}(\tilde{y}^{*}_{\gamma}).
\end{equation}
Utilizing Lemma~\ref{le:properties_of_phi}(d), we can bound $r^{*}_y$ defined by \eqref{eq:la2_proof5} as
\begin{equation*}
\norms{r^{*}_y} \leq  \gamma\norms{\nabla{b}(\tilde{y}^{*}_{\gamma})} + \norms{K}\norms{ F(\tilde{x}^{*}_{\gamma}) - \tilde{F}(\tilde{x}^{*}_{\gamma}) } \leq \gamma D_{\psi} + \norms{K}\norms{ F(\tilde{x}^{*}_{\gamma}) - \tilde{F}(\tilde{x}^{*}_{\gamma}) }.
\end{equation*}
Combining this estimate and \eqref{eq:la2_proof5}, we get
\begin{equation}\label{eq:la2_proof6}
 \dist{0, K^{\top}F(\tilde{x}^{*}_{\gamma}) - \partial{\psi}(\tilde{y}^{*}_{\gamma})} \leq \norms{K}\norms{ F(\tilde{x}^{*}_{\gamma}) - \tilde{F}(\tilde{x}^{*}_{\gamma}) } + \gamma D_{\psi}.
\end{equation}
On the other hand, using the definition of $\widetilde{\nabla}{\Phi}_{\gamma}(\cdot)$ from \eqref{eq:approx_point}, we can show that
\begin{equation*}
\hspace{-2ex}\begin{array}{lcl}
\norms{\widetilde{\nabla}{\Phi}_{\gamma}(\bar{x}) -\nabla{\Phi}_{\gamma}(\bar{x})}  \hspace{-3ex}&=\hspace{-3ex}& \norms{\tilde{J}(\bar{x})^{\top} \nabla{\phi}_{\gamma}(\tilde{F}(\bar{x})) - F'(\bar{x})^{\top} \nabla{\phi}_{\gamma}(F(\bar{x}))} \vspace{1ex}\\
&\leq & \norms{ (\tilde{J}(\bar{x}) - F'(\bar{x}))^{\top} \nabla{\phi}_{\gamma}(\tilde{F}(\bar{x})) }  + \norms{ F'(\bar{x})^{\top}\big( \nabla{\phi}_{\gamma}(\tilde{F}(\bar{x})) - \nabla{\phi}_{\gamma}(F(\bar{x}) \big) } \vspace{1ex}\\
&\leq&  \norms{ (\tilde{J}(\bar{x}) - F'(\bar{x}))^{\top} \nabla{\phi}_{\gamma}(\tilde{F}(\bar{x})) }   +  \norms{\nabla{\phi}_{\gamma}(\tilde{F}(\bar{x})) - \nabla{\phi}_{\gamma}(F(\bar{x}))}\norms{F'(\bar{x})} \vspace{1ex}\\
&\overset{\tiny(i)}{\leq}&  \norms{ (\tilde{J}(\bar{x}) - F'(\bar{x}))^{\top} \nabla{\phi}_{\gamma}(\tilde{F}(\bar{x})) } + L_{\phi_{\gamma}} \norms{F'(\bar{x})} \norms{\tilde{F}(\bar{x}) - F(\bar{x})} \vspace{1ex}\\
&\overset{\tiny\eqref{eq:F_Lipschitz}}{\leq}&  \norms{ (\tilde{J}(\bar{x}) - F'(\bar{x}))^{\top} \nabla{\phi}_{\gamma}(\tilde{F}(\bar{x})) } + L_{\phi_{\gamma}}M_F\norms{\tilde{F}(\bar{x}) - F(\bar{x})}.
\end{array}
\end{equation*}
Here, we have used the $L_{\phi_{\gamma}}$-smoothness of $\phi_{\gamma}$ in \emph{(i)}.

Finally, combining the last estimate, \eqref{eq:approx_KKT1}, and \eqref{eq:la2_proof6}, and using the definition of $\Ec$ from \eqref{eq:approx_KKT_point}, we have
\begin{equation*}
\begin{array}{lcl}
\Ec(\tilde{x}^{*}_{\gamma}, \tilde{y}^{*}_{\gamma}) &:= & \dist{0, \nabla{\Phi}_{\gamma}(\tilde{x}_{\gamma}^{*}) + \partial{\Rc}(\tilde{x}_{\gamma}^{*})} + \dist{0, K^{\top}F(\tilde{x}^{*}_{\gamma}) - \partial{\psi}(\tilde{y}^{*}_{\gamma})} \vspace{1ex}\\
&\leq & \left(1 + \eta L_{\Phi_{\gamma}}\right)\norms{\Gc_{\eta}(\bar{x})} + (2 + \eta L_{\Phi_{\gamma}}) \norms{\nabla{\Phi}_{\gamma}(\bar{x}) - \widetilde{\nabla}{\Phi}_{\gamma}(\bar{x})} \vspace{1ex}\\
&& + {~} \norms{K}\norms{ F(\tilde{x}^{*}_{\gamma}) - \tilde{F}(\tilde{x}^{*}_{\gamma}) } + \gamma D_{\psi} \vspace{1ex}\\
&\leq & \left(1 + \eta L_{\Phi_{\gamma}}\right)\norms{\Gc_{\eta}(\bar{x})} +   \gamma D_{\psi} + \norms{K}\norms{F(\tilde{x}^{*}_{\gamma}) - \tilde{F}(\tilde{x}^{*}_{\gamma})} \vspace{1ex}\\
&& + {~} \left(2 + \eta L_{\Phi_{\gamma}}\right)\big[ \norms{ (\tilde{J}(\bar{x}) - F'(\bar{x}))^{\top} \nabla{\phi}_{\gamma}(\tilde{F}(\bar{x})) } + L_{\phi_{\gamma}}M_F\norms{\tilde{F}(\bar{x}) - F(\bar{x})}\big], 
\end{array}
\end{equation*}
which proves \eqref{eq:approx_KKT2}.
\end{proof}
%%% End of the proof.

%%% Beginning of Proof.
\begin{proof}[\textbf{The proof of Lemma~\ref{le:approx_KKT0}}]
For notational simplicity, we drop the subscript $T$ in this proof.
Since $M_{\phi_{\gamma}} = M_{\psi}\norms{K}$ and $L_{\phi_{\gamma}} = \frac{\norms{K}^2}{\gamma+\mu_{\psi}}$, using the conditions in Lemma~\ref{le:approx_KKT0} and \eqref{eq:approx_oralce10}, we can derive from \eqref{eq:approx_KKT2} after taking the full expectation that
\begin{equation*}
\arraycolsep=0.1em
\begin{array}{lcl}
\Exp{\Ec(\tilde{x}^{*}_{\gamma}, \tilde{y}^{*}_{\gamma})} &\leq & \left(1 + \eta L_{\Phi_{\gamma}}\right)\Exp{\norms{\Gc_{\eta}(\bar{x})}} +  \left(2 + \eta L_{\Phi_{\gamma}}\right)\Exp{ \norms{ (\tilde{J}(\bar{x}) - F'(\bar{x}))^{\top} \nabla{\phi}_{\gamma}(\tilde{F}(\bar{x})) } } \vspace{1ex}\\
&& + {~} \norms{K}\Exp{\norms{F(\tilde{x}^{*}_{\gamma}) - \tilde{F}(\tilde{x}^{*}_{\gamma})}} +  \left(2 + \eta L_{\Phi_{\gamma}}\right) \frac{\norms{K}^2M_F}{\mu_{\psi}+\gamma}\Exp{\norms{\tilde{F}(\bar{x}) - F(\bar{x})}}  +  \gamma D_{\psi}.
\end{array}
\end{equation*}
Now, by the Jensen inequality $\Exp{\norms{\Gc_{\eta}(\bar{x})}} \leq \big(\Exp{\norms{\Gc_{\eta}(\bar{x})}^2} \big)^{1/2} \leq \varepsilon$.
In addition, by \eqref{eq:approx_oralce10}, we also have  $0 < \gamma \leq c_2\epsilon$, $\Exp{ \norms{ (\tilde{J}(\bar{x}) - F'(\bar{x}))^{\top} \nabla{\phi}_{\gamma}(\tilde{F}(\bar{x})) } } \leq \varepsilon$, $\Exp{\norms{F(\tilde{x}^{*}_{\gamma}) - \tilde{F}(\tilde{x}^{*}_{\gamma})}} \leq\varepsilon$, and $\frac{1}{\mu_{\psi}+\gamma}\Exp{\norms{\tilde{F}(\bar{x}) - F(\bar{x})}} \leq \varepsilon$.
By the update rule of $\eta$ in Theorems~\ref{th:convergence2_scvx}, \ref{th:convergence2_scvx_diminishing}, \ref{th:convergence2}, and \ref{th:nonsmooth_diminishing}, we have $\eta L_{\Phi_{\gamma}} = \frac{2}{3+\theta} \leq \frac{2}{3}$ since $\theta \in (0, 1]$.
Substituting these expressions into the last inequality, we finally arrive at
\begin{equation*}
\Exp{\Ec(\tilde{x}^{*}_{\gamma}, \tilde{y}^{*}_{\gamma})} \leq  (1 + \tfrac{2}{3})\varepsilon + c_2D_{\psi}\varepsilon + \norms{K}\varepsilon + (2 + \tfrac{2}{3})(1 + \norms{K}^2M_F)\varepsilon,
\end{equation*}
which is exactly \eqref{eq:approx_KKT2_main}.
\end{proof}
%%% End of proof.

%%% 2.1. Estimating Variance of Stochastic Estimator
\beforesec
\section{Convergence analysis of Algorithm~\ref{alg:A1} in Section \ref{sec:alg_and_theory}}
\aftersec
This Supp. Doc. provides the full analysis of Algorithm~\ref{alg:A1}, including convergence rates and oracle complexity for both strongly convex and non-strongly convex cases of $\psi$ (or equivalently, the smoothness and the nonsmoothness of $\phi_0$, respectively).

\beforesubsec
\subsection{Preparing technical results}
\aftersubsec
Let us first recall and prove some technical results to prepare for our convergence analysis.

%%% Lemma 3.5.
\begin{lemma}
Let $\set{x_t}$ be generated by Algorithm~\ref{alg:A1}, $L_{\Phi_{\gamma_t}}$ be defined by \eqref{eq:Phi_smoothness}, and $B_{\psi}$ be given in Lemma~\ref{le:properties_of_phi}.
Then, under Assumptions~\ref{ass:A1} and \ref{ass:A2}, for any $\eta_t > 0$ and $\theta_t \in [0, 1]$, we have
\begin{equation}\label{eq:lem4_psi_bound}
\hspace{-1ex}\begin{array}{lcl}
\Exp{\Psi_{\gamma_t}(x_{t+1})} {\!\!\!\!}&\leq {\!\!\!\!}& \Exp{\Psi_{\gamma_{t-1}}(x_t)}  + \frac{\theta_t \left(1 + L_{\Phi_{\gamma_t}}^2\eta_t^2\right) }{2L_{\Phi_{\gamma_t}}}\Exp{\norms{\nabla{\Phi}_{\gamma_t}(x_t) - v_t}^2} + (\gamma_{t-1}-\gamma_t)B_{\psi}\vspace{1ex}\\
&& - {~} \frac{L_{\Phi_{\gamma_t}}\eta_t^2\theta_t}{4}\Exp{\norms{\Gc_{\eta_t}(x_t)}^2} - \frac{\theta_t}{2}\left(\frac{2}{\eta_t} - L_{\Phi_{\gamma_t}}\theta_t - 2L_{\Phi_{\gamma_t}}\right)\Exp{\norms{\hat{x}_{t+1} - x_t}^2}.
\end{array}\hspace{-6ex}
\end{equation}
\end{lemma}

%%% Proof of Lemma 3.5.
\begin{proof}
Following the same line of proof of \cite[Lemma 5]{Tran-Dinh2019a}, we can show that
\begin{equation*}
\begin{array}{lcl}
\Exp{\Psi_{\gamma_t}(x_{t+1})} &\leq& \Exp{\Psi_{\gamma_t}(x_t)} + \frac{\theta_t\left(1 + L_{\Phi_{\gamma_t}}^2\eta_t^2\right)}{2L_{\Phi_{\gamma_t}}}\Exp{\norms{\nabla{\Phi}_{\gamma_t}(x_t) - v_t}^2}  \vspace{1ex}\\
&& - {~} \frac{L_{\Phi_{\gamma_t}}\eta_t^2\theta_t}{4}\Exp{\norms{\Gc_{\eta_t}(x_t)}^2} - \frac{\theta_t}{2}\left(\frac{2}{\eta_t} - L_{\Phi_{\gamma_t}}\theta_t - 2L_{\Phi_{\gamma_t}} \right)\Exp{\norms{\hat{x}_{t+1} - x_t}^2}.
\end{array}
\end{equation*}
Finally, since $ \Exp{\Psi_{\gamma_t}(x_t)} \leq  \Exp{\Psi_{\gamma_{t-1}}(x_t)} + (\gamma_{t-1}-\gamma_t)B_{\psi}$ due to Lemma~\ref{le:properties_of_phi}(e), substituting this expression into the last inequality, we obtain \eqref{eq:lem4_psi_bound}.
\end{proof}
%%% End of the proof.

%%% 3.2. One-iteration analysis
\noindent\textbf{The Lyapunov function:}
To analyze Algorithm~\ref{alg:A1}, we introduce the following Lyapunov function:
\begin{equation}\label{eq:Lyapunov_func}
V_{\gamma_{t-1}}(x_t) := \Exp{\Psi_{\gamma_{t-1}}(x_t)} + \frac{\alpha_t}{2}\Exp{\norms{\tilde{F}_t - F(x_t)}^2} + \frac{\hat{\alpha}_t}{2}\Exp{\norms{\tilde{J}_t - F'(x_t)}^2},
\end{equation}
where $\alpha_t > 0$ and $\hat{\alpha}_t > 0$ are given parameters, and the expectation is taken over $\Fc_{t+1}$.
Lemma \ref{le:descent_property} provides a key bound to estimate convergence rates and complexity bounds.

%%% Lemma 3.5.
\begin{lemma}\label{le:descent_property}
Let $\set{x_t}$ be generated by Algorithm~\ref{alg:A1}, and $V_{\gamma_t}$ be the Lyapunov function defined by \eqref{eq:Lyapunov_func}.
Suppose further that the following conditions hold:
\begin{equation}\label{eq:para_cond}
\hspace{0ex}\left\{\hspace{-2ex}\begin{array}{ll}
&\frac{2}{\eta_t} \geq L_{\Phi_{\gamma_t}}\theta_t + 2L_{\Phi_{\gamma_t}}  + \frac{\kappa M_F^2\beta_t^2\theta_t\alpha_{t+1}}{b_1} + \frac{\hat{\kappa}L_F^2\hat{\beta}_t^2\theta_t\hat{\alpha}_{t+1}}{\hat{b}_1} \vspace{1ex}\\
&2M_F^2 L_{\phi_{\gamma_t}}^2\theta_t\Big(\frac{1 + L_{\Phi_{\gamma_t}}^2\eta_t^2}{L_{\Phi_{\gamma_t}}}\Big) + \alpha_{t+1}\beta_t^2  \leq \alpha_t 
\quad\text{and}\quad 
2M_{\phi_{\gamma_t}}^2\theta_t\Big(\frac{1 + L_{\Phi_{\gamma_t}}^2\eta_t^2}{L_{\Phi_{\gamma_t}}}\Big)  +  \hat{\alpha}_{t+1}\hat{\beta}_t^2 \leq \hat{\alpha}_t.
\end{array}\right.
\hspace{-6ex}
\end{equation}
Then, for all $t\geq 0$, one has
\begin{equation}\label{eq:Lyapunov_key}
\hspace{-1ex}
\begin{array}{lcl}
V_{\gamma_t}(x_{t+1}) &\leq& V_{\gamma_{t-1}}(x_t) - \frac{L_{\Phi_{\gamma_t}}\eta_t^2\theta_t}{4}\Exp{\norms{\Gc_{\eta_t}(x_t)}^2} + \frac{\kappa(1-\beta_t)^2\alpha_{t+1}\sigma_F^2}{b_2} + \frac{\hat{\kappa}(1-\hat{\beta}_t)^2\hat{\alpha}_{t+1}\sigma_J^2}{\hat{b}_2} \vspace{1ex}\\
&&+ {~} (\gamma_{t-1}-\gamma_t)B_{\psi}.
\end{array}
\hspace{-1ex}
\end{equation}
\end{lemma}

%%% Proof of Lemma B.2.
\begin{proof}
First of all, by combining \eqref{eq:vt_var0} and \eqref{eq:lem4_psi_bound}, we obtain
\begin{equation}\label{eq:lmb2_proof1}
\arraycolsep=0.1em
\hspace{-1ex}\begin{array}{lcl}
\Exp{\Psi_{\gamma_t}(x_{t+1})} &\leq& \Exp{\Psi_{\gamma_{t-1}}(x_t)}  -  \frac{\theta_t}{2}\left(\frac{2}{\eta_t} - L_{\Phi_{\gamma_t}}\theta_t - 2L_{\Phi_{\gamma_t}} \right)\Exp{\norms{\hat{x}_{t+1} - x_t}^2}\vspace{1ex}\\
&& - {~} \frac{L_{\Phi_{\gamma_t}}\eta_t^2\theta_t}{4}\Exp{\norms{\Gc_{\eta_t}(x_t)}^2} + (\gamma_{t-1}-\gamma_t)B_{\psi} \vspace{1ex}\\
&& + {~} \theta_t\Big(\frac{1+L_{\Phi_{\gamma_t}}^2\eta_t^2}{L_{\Phi_{\gamma_t}}}\Big)\left(M_F^2L_{\phi_{\gamma_t}}^2 \Exp{\norms{\tilde{F}_t - F(x_t)}^2} +M_{\phi_{\gamma_t}}^2 \Exp{\norms{\tilde{J}_t - F'(x_t)}^2}\right).
\end{array}\hspace{-4ex}
\end{equation}
Due to the mini-batch estimators in \eqref{eq:est_update}, it is well-known that 
\begin{equation*}
\begin{array}{lclcl}
 \Exps{\Bc_t^2}{\norms{\Fb(x_t,\zeta_t) - F(x_t)}^2} &= & \Exp{\big\Vert\tfrac{1}{b_2}\sum_{\zeta_i\in\Bc_t^2}\Fb(x_t,\zeta_i) - F(x_t)\big\Vert^2} &\leq & \frac{\sigma_F^2}{b_2} \vspace{1ex}\\
 \Exps{\hat{\Bc}_t^2}{\norms{\Fb'(x_t, \hat{\zeta}_t) - F'(x_t)}^2} &= & \Exp{\big\Vert\frac{1}{\hat{b}_2}\sum_{\hat{\zeta}_i\in\hat{\Bc}^2}\Fb'(x_t,\hat{\zeta}_i) - F'(x_t)\big\Vert^2 }  &\leq& \frac{\sigma_J^2}{\hat{b}_2}.
\end{array}
\end{equation*}
Substituting these bounds and $x_{t+1} - x_t = \theta_t(\hat{x}_{t+1} - x_t)$ into \eqref{eq:var_est1} and taking full expectation the resulting inequality over $\Fc_{t+1}$, we obtain
\begin{equation*} 
\begin{array}{lcl}
\Exp{\norms{\tilde{F}_{t+1} - F(x_{t+1})}^2}  &\leq & \beta_t^2 \Exp{\norms{\tilde{F}_{t} - F(x_{t})}^2} + \frac{\kappa\beta_t^2\theta_t^2M_F^2}{b_1}\Exp{\norms{\hat{x}_{t+1}-x_t}^2} + \frac{\kappa(1-\beta_t)^2\sigma_F^2}{b_2} \vspace{1ex}\\
\Exp{\norms{\tilde{J}_{t+1}  - F'(x_{t+1})}^2}  &\leq & \hat{\beta_t}^2\Exp{\norms{\tilde{J}_{t} - F'(x_{t})}^2} + \frac{\hat{\kappa}\hat{\beta}_t^2\theta_t^2L_F^2}{\hat{b}_1} \Exp{\norms{\hat{x}_{t+1}-x_t}^2} +  \frac{\hat{\kappa}(1-\hat{\beta}_t)^2\sigma_J^2}{\hat{b}_2}.
\end{array}
\end{equation*}
Multiplying these inequalities by $\alpha_{t+1} > 0$ and $\hat{\alpha}_{t+1} > 0$, respectively, and adding the results to \eqref{eq:lmb2_proof1}, we can further derive
\begin{equation*}
\arraycolsep=0.2em
\begin{array}{lcl}
V_{\gamma_t}(x_{t+1}) &\overset{\tiny\eqref{eq:Lyapunov_func}}{:=} & \Exp{\Psi_{\gamma_t}(x_{t+1})} + \frac{\alpha_{t+1}}{2}\Exp{\norms{\tilde{F}_{t+1} - F(x_{t+1})}^2} + \frac{\hat{\alpha}_{t+1}}{2}\Exp{\norms{\tilde{J}_{t+1} - F'(x_{t+1})}^2} \vspace{1ex}\\
&\leq& \Exp{\Psi_{\gamma_{t-1}}(x_t)} + \left[ M_F^2L_{\phi_{\gamma_t}}^2\theta_t\Big(\frac{1 + L_{\Phi_{\gamma_t}}^2\eta_t^2}{L_{\Phi_{\gamma_t}}} \Big) + \frac{\alpha_{t+1}\beta_t^2}{2} \right] \Exp{\norms{\tilde{F}_t - F(x_t)}^2} \vspace{1ex}\\
&& + \left[M_{\phi_{\gamma_t}}^2 \theta_t \Big(\frac{1 + L_{\Phi_{\gamma_t}}^2\eta_t^2}{L_{\Phi_{\gamma_t}}} \Big) + \frac{\hat{\alpha}_{t+1}\hat{\beta}_t^2}{2}\right] \Exp{\norms{\tilde{J}_t - F'(x_t)}^2} - \frac{L_{\Phi_{\gamma_t}}\eta_t^2\theta_t}{4}\Exp{\norms{\Gc_{\eta_t}(x_t)}^2} \vspace{1ex}\\
&& - {~}  \frac{\theta_t}{2}\left( \frac{2}{\eta_t} - L_{\Phi_{\gamma_t}}\theta_t  - 2L_{\Phi_{\gamma_t}}  - \frac{\kappa M_F^2\beta_t^2\theta_t\alpha_{t+1}}{b_1} -  \frac{\hat{\kappa}L_F^2\hat{\beta}_t^2\theta_t\hat{\alpha}_{t+1}}{\hat{b}_1} \right)\Exp{\norms{\hat{x}_{t+1} - x_t}^2} \vspace{1ex}\\
&& + {~} \frac{\kappa(1-\beta_t)^2\alpha_{t+1}\sigma_F^2}{b_2} + \frac{\hat{\kappa}(1-\hat{\beta}_t)^2\hat{\alpha}_{t+1}\sigma_J^2}{\hat{b}_2} + (\gamma_{t-1}-\gamma_t)B_{\psi}.
\end{array}
\end{equation*}
Let us choose $\alpha_t > 0$ and $\hat{\alpha}_t > 0$ and impose three conditions as in \eqref{eq:para_cond}, i.e.:
\begin{equation*}
\left\{\begin{array}{ll}
&\frac{2}{\eta_t} \geq L_{\Phi_{\gamma_t}}\theta_t + 2L_{\Phi_{\gamma_t}} + \frac{\kappa M_F^2\beta_t^2\theta_t\alpha_{t+1}}{b_1} + \frac{\hat{\kappa}L_F^2\hat{\beta}_t^2\theta_t\hat{\alpha}_{t+1}}{\hat{b}_1},\vspace{1ex}\\
&2M_F^2L_{\phi_{\gamma_t}}^2\theta_t \Big( \frac{1 + L_{\Phi_{\gamma_t}}^2\eta_t^2}{L_{\Phi_{\gamma_t}}} \Big) + \alpha_{t+1}\beta_t^2  \leq \alpha_t, 
\quad\text{and}\quad
2M_{\phi_{\gamma_t}}^2\theta_t \Big(\frac{1 + L_{\Phi_{\gamma_t}}^2\eta_t^2}{L_{\Phi_{\gamma_t}}} \Big)  +  \hat{\alpha}_{t+1}\hat{\beta}_t^2 \leq \hat{\alpha}_t.
\end{array}\right.
\end{equation*}
Then, by using \eqref{eq:Lyapunov_func}, the last inequality can be further upper bounded as 
\begin{equation*}
\hspace{-1ex}
\arraycolsep=0.3em
\begin{array}{lcl}
V_{\gamma_t}(x_{t+1}) & \leq & V_{\gamma_{t-1}}(x_t) - \frac{L_{\Phi_{\gamma_t}}\eta_t^2\theta_t}{4}\Exp{\norms{\Gc_{\eta_t}(x_t)}^2} + \frac{\kappa(1-\beta_t)^2\alpha_{t+1}\sigma_F^2}{b_2} \vspace{1ex}\\
&& + {~} \frac{\hat{\kappa}(1-\hat{\beta}_t)^2\hat{\alpha}_{t+1}\sigma_J^2}{\hat{b}_2} + (\gamma_{t-1}-\gamma_t)B_{\psi},
\hspace{-1ex}
\end{array}
\end{equation*}
which proves \eqref{eq:Lyapunov_key}.
\end{proof}
%%% End of proof.

\beforesubsec
\subsection{A general key bound for Algorithm~\ref{alg:A1}}
\aftersubsec
Now, we are ready to prove one key result, Theorem~\ref{thm:comp_D1},  for oracle complexity analysis  of Algorithm~\ref{alg:A1}.
To simplify our expressions, let us introduce the following notations in advance:
\begin{equation}\label{eq:D_new_quatities}
\left\{\begin{array}{lcl}
\omega_t &:= & \frac{\theta_t}{L_{\Phi_{\gamma_t}}} \quad\text{and}\quad \Sigma_T  :=  \sum_{t=0}^T\omega_t, \vspace{1ex}\\
\Theta_t & := & \frac{M_F^2L_{\phi_{\gamma_t}}^2\sqrt{26b_1\hat{b}_1}}{3\big(\kappa M_F^4L_{\phi_{\gamma_t}}^2\hat{b}_1 + \hat{\kappa} M_{\phi_{\gamma_t}}^2L_F^2b_1\big)^{1/2}}, \vspace{1ex}\\
\Pi_0 & := & \frac{\sqrt{26b_1\hat{b}_1}}{3\big(\hat{b}_1\kappa M_F^4L_{\phi_{\gamma_0}}^2 + b_1\hat{\kappa} L_F^2M_{\phi_{\gamma_0}}^2\big)^{1/2}}\left(\frac{\kappa M_F^2L_{\phi_{\gamma_0}}^2\sigma_F^2}{b_0} + \frac{\hat{\kappa} M_{\phi_{\gamma_0}}^2\sigma_J^2}{\hat{b}_0}\right), \vspace{1ex}\\
\Gamma_t &:= & \frac{\sqrt{26b_1\hat{b}_1}}{3\big(\hat{b}_1\kappa M_F^4L_{\phi_{\gamma_t}}^2 + b_1\hat{\kappa} L_F^2M_{\phi_{\gamma_t}}^2\big)^{1/2}}\left(\frac{\kappa M_F^2L_{\phi_{\gamma_t}}^2\sigma_F^2}{b_2} + \frac{\hat{\kappa} M_{\phi_{\gamma_t}}^2\sigma_J^2}{\hat{b}_2}\right).
\end{array}\right.
\end{equation}

%%% Theorem 3.1.
\begin{theorem}\label{thm:comp_D1}
Suppose that Assumptions~\ref{ass:A1} and \ref{ass:A2} hold, and $\omega_t$, $\Sigma_T$, $\Theta_t$, $\Pi_0$, and $\Gamma_t$ are defined by \eqref{eq:D_new_quatities}.
Let $\set{x_t}_{t=0}^T$ be generated by Algorithm~\ref{alg:A1} using the following step-sizes:
\begin{equation}\label{eq:D_para_config}
\theta_t := \frac{3L_{\Phi_{\gamma_t}} \big[ b_1\hat{b}_1(1-\beta_t) \big]^{1/2}}{\sqrt{26}(\kappa M_F^4L_{\phi_{\gamma_t}}^2\hat{b}_1 + \hat{\kappa}M_{\phi_{\gamma_t}}^2L_F^2b_1)^{1/2}} \quad \text{and}\quad \eta_t  := \frac{2}{L_{\Phi_{\gamma_t}}(3 + \theta_t)},
\end{equation}
where $\beta_t, \hat{\beta}_t \in (0, 1]$ are chosen such that $\beta_t = \hat{\beta}_t$, $0 \leq \gamma_{t+1} \leq \gamma_t$, and
\begin{equation}\label{eq:D_beta_cond}
\arraycolsep=0.2em
\left\{ \begin{array}{ll}
& \frac{\beta_t^2(1-\beta_t)}{\Theta_t^2} \leq \frac{1-\beta_{t+1}}{\Theta_{t+1}^2} \leq \frac{1-\beta_t}{\Theta_t^2},  \vspace{1ex}\\
& \beta_t > \max\set{0, 1 - \tfrac{26}{9L_{\Phi_{\gamma_t}}^2}\Big(\tfrac{\kappa M_F^4L_{\phi_{\gamma_t}}^2}{b_1} + \tfrac{\hat{\kappa} L_F^2M_{\phi_{\gamma_t}}^2}{\hat{b}_1}\Big)}.
 \end{array}\right.
\end{equation}
Let $\bar{x}_T$ be randomly chosen between $\set{x_0,\cdots, x_T}$ such that $\Prob{\bar{x}_T = x_t} = \frac{\omega_t}{\Sigma_T}$, and $\bar{\eta}_T$ be corresponding to $\eta_t$ of $\bar{x}_T$.
Then, the following estimate holds:
\begin{equation}\label{eq:D_convergence1}
\displaystyle \Exp{\norms{\Gc_{\bar{\eta}_T}(\bar{x}_T)}^2} \leq  \displaystyle\frac{16}{\Sigma_T}\Big(\Exp{\Psi_0(x_0) - \Psi^{\star}_0}  + \gamma_T B_{\psi} \Big) +  \frac{8\Pi_0}{\Sigma_T\sqrt{1-\beta_0}} +  \displaystyle \frac{16}{\Sigma_T}\sum_{t=0}^T\frac{\Gamma_{t+1}(1-\beta_t)^2}{\sqrt{1 - \beta_{t+1}}}.  
\end{equation}
\end{theorem}
%%% End of Theorem 3.1.

%%% Proof of Theorem 3.1.
\begin{proof}[\textbf{The proof of Theorem~\ref{thm:comp_D1}}]
First, the conditions in \eqref{eq:para_cond} can be simplified as follows:
\begin{equation}\label{eq:D_para_cond1}
\hspace{-1ex}\left\{\hspace{-3ex}\begin{array}{llcl}
&L_{\Phi_{\gamma_t}}\theta_t + 2L_{\Phi_{\gamma_t}} + \big(\frac{\kappa M_F^2\beta_t^2\alpha_{t+1}}{b_1} + \frac{\hat{\kappa}L_F^2\hat{\beta_t}^2\hat{\alpha}_{t+1}}{\hat{b}_1}\big)\theta_t & \leq & \frac{2}{\eta_t}, \vspace{1ex}\\
&2M_F^2L_{\phi_{\gamma_t}}^2(1 + L_{\Phi_{\gamma_t}}^2\eta_t^2)\theta_t & \leq & L_{\Phi_{\gamma_t}}(\alpha_t -\beta_t^2\alpha_{t+1}), \vspace{1ex}\\
&2M_{\phi_{\gamma_t}}^2(1+L_{\Phi_{\gamma_t}}^2\eta_t^2)\theta_t  &\leq & L_{\Phi_{\gamma_t}}(\hat{\alpha}_t -\hat{\beta}_t^2\hat{\alpha}_{t+1}).
\end{array}\right.\hspace{-6ex}
\end{equation}
Let us update $\eta_t := \frac{2}{(3+\theta_t)L_{\Phi_{\gamma_t}}}$ as \eqref{eq:D_para_config}.
Since $\theta_t \in (0, 1]$, we have 
\begin{equation*}
\frac{1}{2L_{\Phi_{\gamma_t}}} \leq \eta_t < \frac{2}{3L_{\Phi_{\gamma_t}}} \quad \text{and}\quad 1 \leq  1 + L_{\Phi_{\gamma_t}}^2\eta_t^2 < \frac{13}{9}.
\end{equation*}
Next, let us choose $\gamma_t$, $\beta_t$, $\hat{\beta}_t$, $\alpha_t$, and $\hat{\alpha}_t$ such that
\begin{equation}\label{eq:alpha_cond}
\hat{\beta}_t = \beta_t \in (0, 1], \quad \hat{\alpha}_t = \frac{M_{\phi_{\gamma_t}}^2}{M_F^2L_{\phi_{\gamma_t}}^2}\alpha_t, \quad \frac{M_{\phi_{\gamma_{t+1}}}}{L_{\phi_{\gamma_{t+1}}}} \leq \frac{M_{\phi_{\gamma_t}}}{L_{\phi_{\gamma_t}}}, \quad\text{and} \quad 0 < \alpha_t \leq \alpha_{t+1} \leq \frac{\alpha_t}{\beta_t}.
\end{equation}
Then, we have  
\begin{equation*}
\begin{array}{llcl}
& \alpha_t - \alpha_{t+1}\beta_t^2 & \geq & \alpha_t(1-\beta_t) > 0, \vspace{1ex}\\
\quad\text{and}\quad & \hat{\alpha}_t - \hat{\beta}_t^2\hat{\alpha}_{t+1} & = &  \frac{M_{\phi_{\gamma_t}}^2}{M_F^2L_{\phi_{\gamma_t}}^2}\alpha_t - \beta_t^2\frac{M_{\phi_{\gamma_{t+1}}}^2}{M_F^2L_{\phi_{\gamma_{t+1}}}^2}\alpha_{t+1} \geq  \frac{M_{\phi_{\gamma_t}}^2}{M_F^2L_{\phi_{\gamma_t}}^2}(\alpha_t - \beta_t^2\alpha_{t+1}) \vspace{1ex}\\
&& \geq &  \frac{M_{\phi_{\gamma_t}}^2}{M_F^2L_{\phi_{\gamma_t}}^2}(1-\beta_t)\alpha_t = (1-\beta_t)\hat{\alpha}_t > 0.
\end{array}
\end{equation*}
By using the last two inequalities, we can show that the conditions in \eqref{eq:D_para_cond1} hold, if we have
\begin{equation}\label{eq:D_para_cond2}
\begin{array}{ll}
& 0 < \theta_t \leq \frac{9L_{\Phi_{\gamma_t}}\alpha_t(1 - \beta_t)}{26M_F^2L_{\phi_{\gamma_t}}^2}, \quad\quad 0 < \theta_t \leq \frac{9L_{\Phi_{\gamma_t}}\hat{\alpha}_t(1-\beta_t)}{26M_{\phi_{\gamma_t}}^2}, \vspace{1ex}\\
\text{and}\quad & 0  < \theta_t \leq L_{\Phi_{\gamma_t}}\left(\frac{\kappa M_F^2\alpha_{t}}{b_1} + \frac{\hat{\kappa}L_F^2\hat{\alpha}_{t}}{\hat{b}_1}\right)^{-1}.
\end{array}
\end{equation}
Therefore, the three conditions in \eqref{eq:D_para_cond2} hold if we choose
\begin{equation*}
\frac{\alpha_t(1 - \beta_t)}{M_F^2L_{\phi_{\gamma_t}}^2} = \frac{\hat{\alpha}_t(1 - \beta_t)}{M_{\phi_{\gamma_t}}^2}\quad \text{and}\quad 
\left(\frac{\kappa M_F^2}{b_1} + \frac{\hat{\kappa}L_F^2M_{\phi_{\gamma_t}}^2}{M_F^2L_{\phi_{\gamma_t}}^2\hat{b}_1}\right)\alpha_t   = \frac{26M_F^2L_{\phi_{\gamma_t}}^2}{9\alpha_t(1 -\beta_t)}.
\end{equation*}
These conditions show that we can choose
\begin{equation*}
\begin{array}{lcl}
\alpha_t :=  \frac{\Theta_t}{\sqrt{1-\beta_t}}\quad \text{and}\quad \hat{\alpha}_t  :=  \frac{M_{\phi_{\gamma_t}}^2\Theta_t}{M_F^2L_{\phi_{\gamma_t}}^2\sqrt{1-\beta_t}},
\quad\text{where}\quad \Theta_t := \frac{M_F^2L_{\phi_{\gamma_t}}^2\sqrt{26b_1\hat{b}_1}}{3\big(\kappa M_F^4L_{\phi_{\gamma_t}}^2\hat{b}_1 + \hat{\kappa}M_{\phi_{\gamma_t}}^2L_F^2b_1\big)^{1/2}}.
\end{array} 
\end{equation*}
Clearly, this $\Theta_t$ is exactly given by \eqref{eq:D_new_quatities}.
With this choice of $\alpha_t$ and $\hat{\alpha}_t$, we obtain
\begin{equation*} 
0 < \theta_t \leq \bar{\theta}_t := \frac{9L_{\Phi_{\gamma_t}}\Theta_t\sqrt{(1-\beta_t)}}{26M_F^2L_{\phi_{\gamma_t}}^2} = \frac{3L_{\Phi_{\gamma_t}}\sqrt{b_1\hat{b}_1(1-\beta_t)}}{\sqrt{26}(\kappa M_F^4L_{\phi_{\gamma_t}}^2\hat{b}_1 + \hat{\kappa} M_{\phi_{\gamma_t}}^2L_F^2b_1)^{1/2}}.
\end{equation*}
We then choose $\theta_t := \bar{\theta}_t$ at the upper bound as in \eqref{eq:D_para_config}.

Now, to guarantee that $0 < \bar{\theta}_t \leq 1$, we impose the following condition as in \eqref{eq:D_beta_cond}, i.e.:
\begin{equation*}
\beta_t > \max\set{0, 1 - \tfrac{26}{9L_{\Phi_{\gamma_t}}^2}\left(\tfrac{\kappa M_F^4L_{\phi_{\gamma_t}}^2}{b_1} + \tfrac{\hat{\kappa} L_F^2M_{\phi_{\gamma_t}}^2}{\hat{b}_1}\right)}.
\end{equation*}
Due to the choice of $\alpha_t$, the condition $\alpha_t \leq \alpha_{t+1} \leq \frac{\alpha_t}{\beta_t}$ in \eqref{eq:alpha_cond}  is equivalent to
\begin{equation*}
\frac{\beta_t^2(1-\beta_t)}{\Theta_t^2} \leq \frac{1-\beta_{t+1}}{\Theta_{t+1}^2} \leq \frac{1-\beta_t}{\Theta_t^2},
\end{equation*}
which is the first condition of \eqref{eq:D_beta_cond}.
Moreover, since $M_{\phi_{\gamma_t}} = M_{\psi}\norms{K}$ and $L_{\phi_{\gamma_t}} = \frac{\norms{K}^2}{\mu_{\psi} + \gamma_t}$ due to Lemma~\ref{le:properties_of_phi}, the third condition of \eqref{eq:alpha_cond} reduces to $\gamma_{t+1} \leq \gamma_t$, which is one of the conditions in Theorem~\ref{thm:comp_D1}.

Next, under the choice of $\alpha_t$ and $\hat{\alpha}_t$, and $\eta_t \geq \frac{1}{2L_{\Phi_{\gamma_t}}}$, \eqref{eq:Lyapunov_key} implies
\begin{equation}\label{eq:D_proof26}
\hspace{-2ex}\begin{array}{ll}
&\frac{\theta_t}{16L_{\Phi_{\gamma_t}}}\Exp{\norms{\Gc_{\eta_t}(x_t)}^2} \leq  V_{\gamma_{t-1}}(x_t) - V_{\gamma_t}(x_{t+1})  + (\gamma_{t-1} - \gamma_t)B_{\psi} \vspace{1ex}\\
&+ {~} \frac{\sqrt{26b_1\hat{b}_1}}{3\big(\hat{b}_1\kappa M_F^4L_{\phi_{\gamma_{t+1}}}^2 + b_1 \hat{\kappa}L_F^2M_{\phi_{\gamma_{t+1}}}^2\big)^{1/2}}\left(\frac{\kappa M_F^2L_{\phi_{\gamma_{t+1}}}^2\sigma_F^2}{b_2} + \frac{\hat{\kappa} M_{\phi_{\gamma_{t+1}}}^2\sigma_J^2}{\hat{b}_2}\right)\frac{(1-\beta_t)^2}{(1 - \beta_{t+1})^{1/2}}.
\end{array}\hspace{-6ex}
\end{equation}
Note that since $\Psi_{\gamma_0}(x_0) \leq \Psi_0(x_0)$ due to Lemma~\ref{le:properties_of_phi}, and $\gamma_{-1} = \gamma_0$ by convention, we have
\begin{equation}\label{eq:D_proof27}
\arraycolsep=0.2em
\hspace{-1ex}\begin{array}{lcl}
V_{\gamma_0}(x_0) &= & \Exp{\Psi_{\gamma_0}(x_0)} + \frac{\alpha_0}{2}\Exp{\norms{\tilde{F}_0 - F(x_0)}^2} + \frac{\hat{\alpha}_0}{2}\Exp{\norms{\tilde{J}_0 - F'(x_0)}^2} \vspace{1ex}\\
&\leq&  \Exp{\Psi_{0}(x_0)} + \frac{\sqrt{26b_1\hat{b}_1}}{6\big(\hat{b}_1 \kappa M_F^4 L_{\phi_{\gamma_0}}^2 + b_1\hat{\kappa} L_F^2M_{\phi_{\gamma_0}}^2\big)^{1/2}} \left(\frac{\kappa M_F^2L_{\phi_{\gamma_0}}^2\sigma_F^2}{b_0} + \frac{\hat{\kappa} M_{\phi_{\gamma_0}}^2\sigma_J^2}{\hat{b}_0}\right)\frac{1}{(1-\beta_0)^{1/2}}.
\end{array}\hspace{-4ex}
\end{equation}
Moreover, by Lemma~\ref{le:properties_of_phi}(d), we have
\begin{equation}\label{eq:D_V_to_Psi}
V_{\gamma_T}(x_{T+1}) \geq   \Exp{\Psi_{\gamma_T}(x_{T+1})} \geq \Exp{\Psi_{0}(x_{T+1})} - \gamma_T B_{\psi} \geq \Psi^{\star}_0  - \gamma_T B_{\psi}.
\end{equation}
Let us define $\Gamma_t$ and $\Pi_0$ as \eqref{eq:D_new_quatities}, i.e.:
\begin{equation*} 
\left\{\begin{array}{lcl}
\Gamma_t &:= & \frac{\sqrt{26b_1\hat{b}_1}}{3\big(\hat{b}_1\kappa M_F^4L_{\phi_{\gamma_t}}^2 + b_1\hat{\kappa} L_F^2M_{\phi_{\gamma_t}}^2\big)^{1/2}}\left(\frac{\kappa M_F^2L_{\phi_{\gamma_t}}^2\sigma_F^2}{b_2} + \frac{\hat{\kappa} M_{\phi_{\gamma_t}}^2\sigma_J^2}{\hat{b}_2}\right), \vspace{1ex}\\
\Pi_0 & := & \frac{\sqrt{26b_1\hat{b}_1}}{3\big(\hat{b}_1\kappa M_F^4L_{\phi_{\gamma_0}}^2 + b_1 \hat{\kappa} L_F^2M_{\phi_{\gamma_0}}^2\big)^{1/2}}\left(\frac{\kappa M_F^2L_{\phi_{\gamma_0}}^2\sigma_F^2}{b_0} + \frac{\hat{\kappa} M_{\phi_{\gamma_0}}^2\sigma_J^2}{\hat{b}_0}\right).
\end{array}\right.
\end{equation*}
Then, summing up \eqref{eq:D_proof26} from $t:=0$ to $t := T$, and using these expressions, \eqref{eq:D_proof27}, and \eqref{eq:D_V_to_Psi}, we get
\begin{equation*}
\sum_{t=0}^T \frac{\theta_t}{16L_{\Phi_{\gamma_t}}}\Exp{\norms{\Gc_{\eta_t}(x_t)}^2} \leq \Exp{\Psi_{0}(x_0) - \Psi^{\star}_0}   + \gamma_T B_{\psi}  + \sum_{t=0}^T\frac{\Gamma_{t+1}(1-\beta_t)^2}{(1 - \beta_{t+1})^{1/2}} +  \frac{\Pi_0}{2(1-\beta_0)^{1/2}}.
\end{equation*}
Dividing this inequality by $\frac{\Sigma_T}{16}$, where $\Sigma_T := \sum_{t=0}^T\omega_t \equiv \sum_{t=0}^T\frac{\theta_t}{L_{\Phi_{\gamma_t}}}$, we obtain
\begin{equation*}
\arraycolsep=0.2em
\hspace{-1ex}\begin{array}{lcl}
\displaystyle\frac{1}{\Sigma_T}\sum_{t=0}^T\omega_t\Exp{\norms{\Gc_{\eta_t}(x_t)}^2} & \leq & \displaystyle\frac{16}{\Sigma_T}\Big(\Exp{\Psi_0(x_0) - \Psi^{\star}_0}  + \gamma_T B_{\psi} \Big) +  \frac{8\Pi_0}{\Sigma_T(1-\beta_0)^{1/2}} \vspace{1ex}\\
&& + {~} \displaystyle \frac{16}{\Sigma_T}\sum_{t=0}^T\frac{\Gamma_{t+1}(1-\beta_t)^2}{(1 - \beta_{t+1})^{1/2}}.  
\end{array}\hspace{-4ex}
\end{equation*}
Finally, due to the choice of $\bar{x}_T$ and $\bar{\eta}_T$, we have $\frac{1}{\Sigma_T}\sum_{t=0}^T\omega_t\Exp{\norms{\Gc_{\eta_t}(x_t)}^2} = \Exp{\norms{\Gc_{\bar{\eta}_T}(\bar{x}_T)}^2}$.
This relation together with the above estimate prove \eqref{eq:D_convergence1}.
\end{proof}
%%% End of the proof.

%%% B.3. The smooth case with constant step-size.
\beforesubsec
\subsection{The proof of Theorem~\ref{th:convergence2_scvx}: The smooth case with constant step-size}\label{apdx:subsec:th:convergence2_scvx}
\aftersubsec
Now, we prove our first main result in the main text.

%%% Proof of Theorem 3.2.
\begin{proof}[\textbf{The proof of Theorem~\ref{th:convergence2_scvx} in the main text}]
First, since $\mu_{\psi} = 1 > 0$, we can set $\gamma_t := 0$ for all $t\geq 0$.
That means, we do not need to smooth $\phi_0$ in \eqref{eq:com_nlp}.
Hence, from \eqref{eq:D_new_quatities}, $\Theta_t = \Theta_0 =  \frac{M_F^2L_{\phi_0}\sqrt{26 b_1\hat{b}_1}}{3\big(\kappa M_F^4L_{\phi_0}^2\hat{b}_1 {~} + {~}  \hat{\kappa} M_{\phi_0}^2L_F^2b_1\big)^{1/2}}$ and $\frac{\omega_t}{\Sigma_T} = \frac{\theta_t}{\sum_{t=0}^T\theta_t}$, where $L_{\Phi_0}$ is defined by \eqref{eq:constant_defs}.

Next, given a batch size $b > 0$, let us choose the mini-batch sizes  $b_0 := c_0\hat{b}_0 > 0$, $\hat{b}_1 = \hat{b}_2 := b > 0$, and $b_1 = b_2 := c_0b$ for some $c_0 > 0$.
We also choose a constant step-size $\theta_t := \theta \in (0, 1]$ and a constant weight $\beta_t := \beta \in (0, 1]$ for all $t \geq 0$.
We also recall $P$, $Q$, and $L_{\Phi_0}$ defined by \eqref{eq:constant_defs}.

With this configuration, the first condition of \eqref{eq:D_beta_cond} and $0 \leq \gamma_{t+1} \leq \gamma_t$ are automatically satisfied, while the second one becomes
\begin{equation}\label{eq:D_cond_beta}
\beta > \max\set{0, 1 - \tfrac{26}{9c_0L_{\Phi_0}^2b}\big(\kappa M_F^4\norms{K}^4 {~} + {~}  c_0\hat{\kappa} \norms{K}^2L_F^2M_{\psi}^2 \big)} = \max\set{0, 1 - \tfrac{P^2}{L_{\Phi_0}^2 b} }.
\end{equation}
Moreover, we also obtain from \eqref{eq:D_new_quatities}, \eqref{eq:D_para_config}, and \eqref{eq:constant_defs} that
\begin{equation*}
\left\{\begin{array}{lclcl}
\theta_t &= & \theta = \frac{3L_{\Phi_0}\sqrt{c_0b(1-\beta)}}{\sqrt{26}(\kappa M_F^4\norms{K}^4  {~} + {~} c_0\hat{\kappa}\norms{K}^2M_{\psi}^2L_F^2)^{1/2}}  
& \overset{\tiny \eqref{eq:constant_defs}}{=} & \frac{L_{\Phi_0}[b(1-\beta)]^{1/2}}{P}, \vspace{1ex}\\
\Gamma_t &= & \Gamma = \frac{\sqrt{26}\left( \kappa M_F^2\norms{K}^4\sigma_F^2 {~} + {~} c_0\hat{\kappa}\norms{K}^2M_{\psi}^2\sigma_J^2\right)}{3\sqrt{c_0b}\big(\kappa M_F^4\norms{K}^4 {~} + {~} c_0\hat{\kappa}\norms{K}^2L_F^2M_{\psi}^2\big)^{1/2}} 
& \overset{\tiny \eqref{eq:constant_defs}}{=}  & \frac{Q}{P\sqrt{b}},  \vspace{1ex}\\
\Pi_0 & = & \frac{\sqrt{26b}\left(\kappa M_F^2\norms{K}^4\sigma_F^2 {~} + {~} c_0\hat{\kappa}\norms{K}^2M_{\psi}^2\sigma_J^2\right)}{3\sqrt{c_0} \hat{b}_0\big(\kappa M_F^4\norms{K}^4 {~} + {~} c_0\hat{\kappa}\norms{K}^2L_F^2M_{\psi}^2\big)^{1/2}} 
& \overset{\tiny \eqref{eq:constant_defs}}{=} & \frac{Q\sqrt{b}}{P\hat{b}_0}, \vspace{1ex}\\
\Sigma_T &= & \sum_{t=0}^T\frac{\theta}{L_{\Phi_0}} = \frac{\theta(T+1)}{L_{\Phi_0}} & = & \frac{(T+1)[b(1-\beta)]^{1/2}}{P}. 
\end{array}\right.
\end{equation*}
Furthermore, with these expressions of $\Gamma_t$, $\Pi_0$, and $\Sigma_T$, \eqref{eq:D_convergence1} reduces to
\begin{equation*} 
\hspace{-1ex}\begin{array}{lcl}
\Exp{\norms{\Gc_{\eta}(\bar{x}_T)}^2}  &\leq & \frac{16P}{(T+1)[b(1-\beta)]^{1/2}}\Exp{\Psi_0(x_0) - \Psi_0^{\star}}  +  \frac{8Q}{\hat{b}_0(T+1)(1-\beta)} + \frac{16Q(1-\beta)}{b}.
\end{array}\hspace{-6ex}
\end{equation*}
Trading-off the term $\frac{1}{\hat{b}_0(1-\beta)(T+1)} + \frac{2(1-\beta)}{b}$ over $\beta \in (0, 1]$, we obtain $\beta := 1 - \frac{b^{1/2}}{[\hat{b}_0(T+1)]^{1/2}}$, which has shown in \eqref{eq:para_config0}.
In this case, $\theta_t = \theta = \frac{L_{\Phi_0}[b(1-\beta)]^{1/2}}{P} = \frac{ L_{\Phi_0} b^{3/4} }{P[\hat{b}_0(T+1)]^{1/4}}$ as shown in \eqref{eq:para_config0}.

Now, let us choose $\hat{b}_0 := c_1^2[b(T+1)]^{1/3}$ for some $c_1 > 0$.
Then, the last inequality leads to
\begin{equation*} 
\hspace{-1ex}\begin{array}{ll}
\Exp{\norms{\Gc_{\eta}(\bar{x}_T)}^2} {\!\!\!\!}&\leq \frac{16P\sqrt{c_1}}{ [b(T+1)]^{2/3}}\big[\Psi_0(x_0) - \Psi_0^{\star}\big] +  \frac{24Q}{2c_1[b(T+1)]^{2/3}}.
\end{array}\hspace{-6ex}
\end{equation*}
Hence, if we define $\Delta_0$ as in \eqref{eq:convergence_rate1_b}, i.e.: 
\begin{equation*}
\Delta_0 := 16P\sqrt{c_1} \big[\Psi_0(x_0)  - \Psi_0^{\star}\big]  +  \frac{24Q}{c_1},
\end{equation*}
then we obtain from the last inequality that \eqref{eq:convergence_rate1_b} holds, i.e.:
\begin{equation*}
\Exp{\norms{\Gc_{\eta}(\bar{x}_T)}^2} \leq \frac{ \Delta_0}{[b(T+1)]^{2/3}}.
\end{equation*}
Consequently, for a given tolerance $\varepsilon > 0$, to obtain $\Exp{\norms{\Gc_{\eta}(\bar{x}_T)}^2}  \leq \varepsilon^2$, we need at most $T := \big\lfloor \frac{\Delta_0^{3/2}}{b\varepsilon^3}\big\rfloor$ iterations.
In this case, the total number of function evaluations $\Fb(x_t,\xi)$ is at most
\begin{equation*}
\Tc_F := b_0 + (T+1)(2b_1 + b_2) = c_0c_1^2[b(T+1)]^{1/3} + 3c_0(T+1)b = \frac{c_0c_1^2\Delta_0^{1/2}}{\varepsilon} + \frac{3c_0\Delta_0^{3/2}}{\varepsilon^3}.
\end{equation*}
Alternatively, the total number of Jacobian evaluations $\Fb'(x_t,\xi)$ is at most
\begin{equation*}
\Tc_J := \hat{b}_0 + (T+1)(2\hat{b}_1 + \hat{b}_2) = c_1^2[b(T+1)]^{1/3} + 3(T+1)b = \frac{c_1^2\Delta_0^{1/2}}{\varepsilon} + \frac{3\Delta_0^{3/2}}{\varepsilon^3}.
\end{equation*}
Finally, since $\beta := 1 - \frac{b^{1/2}}{[\hat{b}_0(T+1)]^{1/2}}$, the condition \eqref{eq:D_cond_beta} leads to $\frac{b^{1/2}}{[\hat{b}_0(T+1)]^{1/2}} < \frac{P^2}{L_{\Phi_0}^2 b}$, which is equivalent to $\frac{\hat{b}_0(T+1)}{b^3} > \frac{L_{\Phi_0}^4}{P^4}$ as shown in Theorem~\ref{th:convergence2_scvx}. 
\end{proof}
%%% End of the proof.

%%%%% Theorem B.2.
\beforesubsec
\subsection{The proof of Theorem~\ref{th:convergence2_scvx_diminishing}: The smooth case with diminishing step-size}\label{apdx:subsec:th:convergence2_scvx_diminishing}
\aftersubsec

%%% The proof of Theorem B.2.
\begin{proof}[\textbf{The proof of Theorem~\ref{th:convergence2_scvx_diminishing} in the main text}]
Similar to the proof of Theorem~\ref{th:convergence2_scvx}, with $\mu_{\psi} = 1 > 0$, we set $\gamma_t = 0$.
Hence, we obtain $\Theta_t = \Theta_0 =  \frac{M_F^2L_{\phi_0}\sqrt{26b_1\hat{b}_1}}{3\big(\kappa M_F^4L_{\phi_0}^2\hat{b}_1 + \hat{\kappa} M_{\phi_0}^2L_F^2b_1\big)^{1/2}}$ and $\frac{\omega_t}{\Sigma_T} = \frac{\theta_t}{\sum_{t=0}^T\theta_t}$.

Next, given a mini-batch size $b > 0$, let us choose the mini-batch sizes $b_0 := c_0\hat{b}_0$, $\hat{b}_1 = \hat{b}_2 := b$, and $b_1 = b_2 := c_0b > 0$ for some $c_0 > 0$.
With these choices, the condition \eqref{eq:D_beta_cond} becomes
\begin{equation}\label{eq:D_cond_beta2}
\beta_t^2(1-\beta_t) \leq 1-\beta_{t+1} \leq 1 - \beta_t  \text{ and } 
\beta_t > \max\set{0, 1 - \tfrac{26}{9c_0L_{\Phi_0}^2b}\big(c_0\kappa M_F^4L_{\phi_0}^2 +  \hat{\kappa} L_F^2M_{\phi_0}^2\big)}.
\end{equation}
Moreover, from \eqref{eq:D_new_quatities} and \eqref{eq:D_para_config}, we have
\begin{equation*}
\left\{\begin{array}{lclcl}
\theta_t &= &  \frac{3L_{\Phi_0}\sqrt{c_0b(1-\beta_t)}}{\sqrt{26}(\kappa M_F^4\norms{K}^4 + c_0\hat{\kappa}\norms{K}^2M_{\psi}^2L_F^2)^{1/2}} & \overset{\tiny\eqref{eq:constant_defs}}{=} & \frac{L_{\Phi_0}[b(1-\beta_t)]^{1/2}}{P}, \vspace{1ex}\\
\Gamma_t &= & \Gamma = \frac{\sqrt{26}\left( \kappa M_F^2\norms{K}^4\sigma_F^2 {~} + {~} c_0\hat{\kappa}\norms{K}^2M_{\psi}^2\sigma_J^2\right)}{3\sqrt{c_0b}\big(\kappa M_F^4\norms{K}^4 {~} + {~} c_0\hat{\kappa}\norms{K}^2L_F^2M_{\psi}^2\big)^{1/2}} & \overset{\tiny\eqref{eq:constant_defs}}{=} & \frac{Q}{P\sqrt{b}}, \vspace{1ex}\\
\Pi_0 & = & \frac{\sqrt{26b}\left( \kappa M_F^2\norms{K}^4\sigma_F^2 {~} + {~} c_0\hat{\kappa}\norms{K}^2M_{\psi}^2\sigma_J^2\right)}{3\sqrt{c_0}\hat{b}_0\big( \kappa M_F^4\norms{K}^4 {~} + {~} c_0\hat{\kappa}\norms{K}^2L_F^2M_{\psi}^2\big)^{1/2}} & \overset{\tiny\eqref{eq:constant_defs}}{=} & \frac{Q\sqrt{b}}{P\hat{b}_0} , \vspace{1ex}\\
\Sigma_T &= & \sum_{t=0}^T\omega_t = \sum_{t=0}^T\frac{\theta_t}{L_{\Phi_0}} &= & \frac{\sqrt{b}}{P}\sum_{t=0}^T\sqrt{1-\beta_t}.
\end{array}\right.
\end{equation*}
Furthermore, with these expressions of $\Gamma_t$, $\Pi_0$, and $\Sigma_T$, \eqref{eq:D_convergence1} reduces to
\begin{equation}\label{eq:proof100_a} 
\arraycolsep=0.2em
\hspace{-1ex}\begin{array}{lcl}
\frac{1}{\sum_{t=0}^T\theta_t}\sum_{t=0}^T\theta_t\Exp{\norms{\Gc_{\eta_t}(x_t)}^2} 
& \leq & \frac{16P}{\sqrt{b}\sum_{t=0}^T\sqrt{1-\beta_t}} \big[ \Psi_0(x_0) - \Psi^{\star}_0 \big]  + \frac{8Q}{\hat{b}_0\sqrt{1-\beta_0}\sum_{t=0}^T\sqrt{1-\beta_t}} \vspace{1ex}\\
&& + {~} \frac{16Q}{b\sum_{t=0}^T\sqrt{1-\beta_t}}\sum_{t=0}^T\frac{(1-\beta_t)^2}{(1-\beta_{t+1})^{1/2}}. 
\end{array}\hspace{-4ex}
\end{equation}
Let us choose $\beta_t := 1- \frac{1}{(t+2)^{2/3}} \in (0, 1)$ as in \eqref{eq:para_config0_a0}.
Then, it is easy to check that $\beta_t^2(1-\beta_t) \leq 1 - \beta_{t+1} \leq 1 - \beta_t$ after a few elementary calculations.

Moreover, we have $\theta_t :=  \frac{L_{\Phi_0}\sqrt{b}}{P(t+2)^{1/3}}$ as \eqref{eq:para_config0_a0}.
In addition, one can easily show that
\begin{equation*}
\left\{\begin{array}{ll}
&\sum_{t=0}^T\sqrt{1-\beta_t} = \sum_{t=0}^{T}\frac{1}{(t+2)^{1/3}} \geq \int_2^{T+3}\frac{ds}{s^{1/3}} = \frac{3}{2}[(T+3)^{2/3} - 2^{2/3}], \vspace{1ex}\\
&\sum_{t=0}^T\frac{(1-\beta_t)^2}{\sqrt{1-\beta_{t+1}}} = \sum_{t=0}^T\frac{(t+3)^{1/3}}{(t+2)^{4/3}} \leq \sum_{t=0}^{T}\frac{1}{(t+1)}  \leq 1 + \log(T+1).
\end{array}\right.
\end{equation*}
Here, we use the fact that $\int_t^{t+1}r(s)ds \leq r(t) \leq \int_{t-1}^tr(s)ds$ for a nonnegative and monotonically decreasing function $r$.

Substituting these estimates and $\sqrt{1-\beta_0} = \frac{1}{2^{1/3}}$ into \eqref{eq:proof100_a}, we eventually obtain
\begin{equation*} 
\arraycolsep=0.2em
\hspace{-1ex}\begin{array}{lcl}
\frac{1}{\sum_{t=0}^T\theta_t}\sum_{t=0}^T\theta_t\Exp{\norms{\Gc_{\eta_t}(x_t)}^2} 
& \leq & \frac{32P}{3\sqrt{b}\big[(T+3)^{2/3} - 2^{2/3}\big]} \big[ \Psi_0(x_0) - \Psi^{\star}_0 \big] \vspace{1ex}\\
&& + {~}  \frac{16Q}{3\big[(T+3)^{2/3} - 2^{2/3}\big]}\left[\frac{2^{1/3}}{\hat{b}_0} + \frac{2(1+\log(T+1))}{b}\right].
\end{array}\hspace{-4ex}
\end{equation*}
Combining this inequality and $\frac{1}{\sum_{t=0}^T\theta_t}\sum_{t=0}^T\theta_t\Exp{\norms{\Gc_{\eta_t}(x_t)}^2} = \Exp{\norms{\Gc_{\bar{\eta}_T}(\bar{x}_T)}^2}$, we have proved \eqref{eq:convergence_rate1_b_a0} for $T \geq 0$.
\end{proof}
%%% End of the proof.

%%% B.4. The non-smooth case with constant step-size 
\beforesubsec
\subsection{The proof of Theorem \ref{th:convergence2}: The nonsmooth case with constant step-size}\label{apdx:subsec:th:convergence2}
\aftersubsec

%%% Beginning of proof of Theorem 3.2.
\begin{proof}[\textbf{The proof of Theorem \ref{th:convergence2} in the main text}]
Since $\mu_{\psi} = 0$, let us fix the smoothness parameter $\gamma_t = \gamma > 0$ and the weights $\beta_t = \hat{\beta}_t = \beta \in (0, 1]$ for all $t\geq 0$.
By Lemma~\ref{le:properties_of_phi}, we have 
\begin{equation*}
M_{\phi_{\gamma}} = M_{\psi}\norms{K}, \quad L_{\phi_{\gamma}} = \frac{\norms{K}^2}{\gamma}, \quad\text{and}\quad L_{\Phi_{\gamma}} =  L_FM_{\psi}\norms{K} + \frac{M_F^2\norms{K}^2}{\gamma}.
\end{equation*}
Given  batch sizes $b > 0$ and $\hat{b}_0 > 0$, for some $c_0 > 0$, let us also choose the mini-batch sizes as
\begin{equation*}
\hat{b}_1 = \hat{b}_2 := b, \quad  b_1 = b_2 := \frac{c_0b}{\gamma^2},  \quad\text{and}\quad b_0 := \frac{c_0\hat{b}_0}{\gamma^2}.
\end{equation*}
Recall that $P$, $Q$, and $L_{\Phi_{\gamma}}$ are defined by \eqref{eq:constant_defs}.
In this case, the quantities in \eqref{eq:D_new_quatities} become
\begin{equation*} 
\arraycolsep=0.2em
\left\{\begin{array}{lclcl}
\Theta_t & := & \Theta = \frac{M_F^2L_{\phi_{\gamma}}\sqrt{26b_1\hat{b}_1}}{3\big(\kappa M_F^4L_{\phi_{\gamma}}^2\hat{b}_1 + \hat{\kappa}M_{\phi_{\gamma}}^2L_F^2b_1\big)^{1/2}} = \frac{\sqrt{26c_0b}M_F^2\norms{K}^2}{3\gamma (\kappa M_F^4\norms{K}^4 + c_0\hat{\kappa}\norms{K}^2M_{\psi}^2L_F^2)^{1/2}} & \overset{\tiny\eqref{eq:constant_defs}}{=} & \frac{M_F^2\norms{K}^2b^{1/2}}{\gamma P}, \vspace{1ex}\\
\Gamma_t &:= & \Gamma = \frac{\sqrt{26b_1\hat{b}_1}}{3\big(\hat{b}_1\kappa M_F^4L_{\phi_{\gamma}}^2 + b_1\hat{\kappa}L_F^2M_{\phi_{\gamma}}^2\big)^{1/2}}\left(\frac{\kappa M_F^2L_{\phi_{\gamma}}^2\sigma_F^2}{b_2} + \frac{\hat{\kappa} M_{\phi_{\gamma}}^2\sigma_J^2}{\hat{b}_2}\right) & \overset{\tiny\eqref{eq:constant_defs}}{=} & \frac{Q}{P\sqrt{b}}, \vspace{1ex}\\
\Pi_0 & := & \frac{\sqrt{26b_1\hat{b}_1}}{3\big(\hat{b}_1\kappa M_F^4L_{\phi_{\gamma}}^2 + b_1\hat{\kappa} L_F^2M_{\phi_{\gamma}}^2\big)^{1/2}}\left(\frac{\kappa M_F^2L_{\phi_{\gamma}}^2\sigma_F^2}{b_0} + \frac{\hat{\kappa} M_{\phi_{\gamma}}^2\sigma_J^2}{\hat{b}_0}\right) 
& \overset{\tiny\eqref{eq:constant_defs}}{=} & \frac{Q\sqrt{b}}{P\hat{b}_0}. 
\end{array}\right.
\end{equation*}
Furthermore, the step-sizes in \eqref{eq:D_para_config} also become
\begin{equation*} 
\left\{\begin{array}{lcl}
\theta_t & := & \theta = \frac{3L_{\Phi_{\gamma}} [ b_1\hat{b}_1(1-\beta)]^{1/2}}{\sqrt{26}(\kappa M_F^4L_{\phi_{\gamma}}^2\hat{b}_1 + \hat{\kappa} M_{\phi_{\gamma}}^2L_F^2b_1)^{1/2}}
 \overset{\tiny\eqref{eq:constant_defs}}{=}  \frac{L_{\Phi_{\gamma}}[b(1-\beta)]^{1/2}}{P}, \vspace{1ex}\\
\eta_t  & := & \eta =  \frac{2}{L_{\Phi_{\gamma}}(3 + \theta)}.
\end{array}\right.
\end{equation*}
Therefore, we have $\omega_t := \frac{\theta}{L_{\Phi_{\gamma}}}$ and
\begin{equation*}
\begin{array}{lcl}
\Sigma_T  :=  \sum_{t=0}^T\omega_t = \frac{\theta (T+1)}{L_{\Phi_{\gamma}}} = \frac{(T+1)[b(1-\beta)]^{1/2}}{P}.
\end{array}
\end{equation*}
Substituting these expressions into \eqref{eq:D_convergence1}, we can further derive
\begin{equation}\label{eq:proof26}
\begin{array}{lcl}
\displaystyle \Exp{\norms{\Gc_{\eta}(\bar{x}_T)}^2} &\leq&  \frac{16P}{(T+1)[b(1-\beta)]^{1/2}}\Big(\Exp{\Psi_0(x_0) - \Psi^{\star}_0}  + \gamma B_{\psi} \Big) \vspace{1ex}\\
&& + {~} 8Q \left[ \frac{1}{\hat{b}_0(1-\beta)(T+1)}  + \frac{2(1 - \beta)}{b} \right].
\end{array}
\end{equation}
From the last term of \eqref{eq:proof26}, we can choose $\beta$ as $\beta = 1 - \frac{b^{1/2}}{[\hat{b}_0(T+1)]^{1/2}}$.
In this case, \eqref{eq:proof26} reduces to
\begin{equation}\label{eq:proof26b}
\displaystyle \Exp{\norms{\Gc_{\eta}(\bar{x}_T)}^2} \leq  \frac{16P\hat{b}_0^{1/4}}{[b(T+1)]^{3/4}}\Big(\Exp{\Psi_0(x_0) - \Psi^{\star}_0}  + \gamma B_{\psi} \Big)  + \frac{24Q}{[b\hat{b}_0(T+1)]^{1/2}}.
\end{equation}
Clearly, from \eqref{eq:proof26b},  to achieve the best convergence rate, we need to choose $\hat{b}_0 := c_1^2[b(T+1)]^{1/3}$.
Then, since we choose $0 < \gamma \leq 1$ and $\Exp{\Psi_{0}(x_0)} = \Psi_0(x_0)$, \eqref{eq:proof26b} can be overestimated as 
\begin{equation*}
\Exp{\norms{\Gc_{\eta}(\bar{x}_T)}^2} \leq \frac{\hat{\Delta}_0}{[b(T+1)]^{2/3}},
\end{equation*}
which proves \eqref{eq:key_est5}, where $\hat{\Delta}_0$ is defined by \eqref{eq:key_est5}, i.e.:  
\begin{equation*}
\begin{array}{l}
\hat{\Delta}_0 := 16P\sqrt{c_1}\big( \Psi_{0}(x_0) - \Psi^{\star}_0 + B_{\psi}\big) + \frac{24Q}{c_1}.
\end{array}
\end{equation*}
Now, for any tolerance $\varepsilon > 0$, to obtain $\Exp{\norms{\Gc_{\eta}(\bar{x}_T)}^2}\leq \varepsilon^2$, we require at most $T := \left\lfloor \frac{\hat{\Delta}_0^{3/2}}{b\varepsilon^3}\right\rfloor$ iterations.
In this case, the total number of function evaluations $\Tc_F$ is at most
\begin{equation*}
\begin{array}{lcl}
\Tc_F &:= & b_0 + (T+1)(2b_1 + b_2) = \frac{c_0}{\gamma^2}c_1^2[b(T+1)]^{1/3} + \frac{3c_0}{\gamma^2}[b(T+1)] =  \frac{c_0c_1^2\hat{\Delta}_0^{1/2}}{\gamma^2\varepsilon} +  \frac{3c_0\hat{\Delta}_0^{3/2}}{\gamma^2\varepsilon^3}.
\end{array}
\end{equation*}
Alternatively, the total number of Jacobian evaluations $\Tc_J$ is at most
\begin{equation*}
\begin{array}{lcl}
\Tc_J & := & \hat{b}_0 + (T+1)(2\hat{b}_1 + \hat{b}_2) = c_1[b(T+1)]^{1/3} + 3b(T+1) =  \frac{c_1^2\hat{\Delta}_0^{1/2}}{\varepsilon} +  \frac{3\hat{\Delta}_0^{3/2}}{\varepsilon^3}. 
\end{array}
\end{equation*}
If we choose $\gamma := c_2\varepsilon$ for some $c_2 > 0$, then 
\begin{equation*}
\Tc_F :=    \frac{c_0c_1^2\hat{\Delta}_0^{1/2}}{c_2^2\varepsilon^3} +  \frac{3c_0\hat{\Delta}_0^{3/2}}{c_2^2\varepsilon^5} = \BigO{ \frac{\hat{\Delta}_0^{3/2}}{\varepsilon^5}},
\end{equation*}
which proves the last statement.
\end{proof}
%%% End of proof.

%%% B.6. The nonsmooth case with diminishing step-size.
\beforesubsec
\subsection{The proof of Theorem~\ref{th:nonsmooth_diminishing}: The nonsmooth case with diminishing step-size}\label{apdx:subsec:th:convergence2_diminishing}
\aftersubsec

%%% The proof of Theorem B.3.
\begin{proof}[\textbf{The proof of Theorem~\ref{th:nonsmooth_diminishing} in the main text}]
Using the fact that $\mu_{\psi} = 0$, from Lemma~\ref{le:properties_of_phi}, we have 
\begin{equation*}
M_{\phi_{\gamma_t}} = M_{\psi}\norms{K}, \quad L_{\phi_{\gamma_t}} = \frac{\norms{K}^2}{\gamma_t}, \quad\text{and}\quad L_{\Phi_{\gamma_t}} =  L_FM_{\psi}\norms{K} + \frac{M_F^2\norms{K}^2}{\gamma_t},
\end{equation*}
where $\gamma_t > 0$, which will be appropriately updated.
Moreover, let us choose $b_0 := \frac{c_0\hat{b}_0}{\gamma_0^2}$, $\hat{b}_1 = \hat{b}_2 := b $, and $b_1^t = b_2^t := \frac{c_0b}{\gamma_{t}^2} > 0$, for some $b > 0$ and $c_0 > 0$.
We also recall $P$, $Q$, and $L_{\Phi_{\gamma}}$ from \eqref{eq:constant_defs}.

With these expressions, the quantities defined by \eqref{eq:D_new_quatities} and \eqref{eq:D_para_config} become
\begin{equation*}
\arraycolsep=0.2em
\left\{\begin{array}{lclcl}
\theta_t &:= & \frac{3L_{\Phi_{\gamma_t}} [ b_1^t\hat{b}_1(1-\beta_t) ]^{1/2}}{\sqrt{26}(\kappa M_F^4L_{\phi_{\gamma_t}}^2\hat{b}_1 + \hat{\kappa} M_{\phi_{\gamma_t}}^2L_F^2b_1^t)^{1/2}} 
& \overset{\tiny\eqref{eq:constant_defs}}{=} &  \frac{L_{\Phi_{\gamma_t}}[b(1-\beta_t)]^{1/2}}{P}, \vspace{1ex}\\
\Theta_t & := & \frac{M_F^2L_{\phi_{\gamma_t}}\sqrt{26b_1^t\hat{b}_1}}{3\big(\kappa M_F^4L_{\phi_{\gamma_t}}^2\hat{b}_1 + \hat{\kappa} M_{\phi_{\gamma_t}}^2L_F^2b_1^t\big)^{1/2}}  
 & \overset{\tiny\eqref{eq:constant_defs}}{=} & \frac{M_F^2\norms{K}^2b^{1/2}}{\gamma_t P}, \vspace{1ex}\\
\Gamma_{t} &:= & \frac{\sqrt{26b_1^t\hat{b}_1}}{3\big(\hat{b}_1\kappa M_F^4L_{\phi_{\gamma_{t}}}^2 + b_1^t\hat{\kappa} L_F^2M_{\phi_{\gamma_{t}}}^2\big)^{1/2}}\left(\frac{\kappa M_F^2L_{\phi_{\gamma_{t}}}^2\sigma_F^2}{b^t_2} + \frac{\hat{\kappa} M_{\phi_{\gamma_{t}}}^2\sigma_J^2}{\hat{b}_2}\right) 
 & \overset{\tiny\eqref{eq:constant_defs}}{=} & \frac{Q}{P\sqrt{b}}, \vspace{1ex}\\
\Pi_0 & := & \frac{\sqrt{26b_1^0\hat{b}_1}}{3\big(\hat{b}_1 \kappa M_F^4L_{\phi_{\gamma_0}}^2 + b_1^0\hat{\kappa} L_F^2M_{\phi_{\gamma_0}}^2\big)^{1/2}}\left(\frac{\kappa M_F^2L_{\phi_{\gamma_0}}^2\sigma_F^2}{b_0} + \frac{\hat{\kappa} M_{\phi_{\gamma_0}}^2\sigma_J^2}{\hat{b}_0}\right) 
 & \overset{\tiny\eqref{eq:constant_defs}}{=} & \frac{Q\sqrt{b}}{P\hat{b}_0}.
\end{array}\right.
\end{equation*}
Let us choose $\beta_t := 1- \frac{1}{(t+2)^{2/3}} \in (0, 1)$ and $\gamma_t := \frac{1}{(t+2)^{1/3}}$ as in \eqref{eq:para_config0_a2}.
Then, it is easy to check that 
\begin{equation*}
\frac{\beta_t^2(1-\beta_t)}{\Theta_t^2} \leq \frac{1 - \beta_{t+1}}{\Theta_{t+1}^2} \leq \frac{1 - \beta_t}{\Theta_t^2}.
\end{equation*}
In addition, as before, one can show that
\begin{equation*}
\left\{\begin{array}{ll}
&\sum_{t=0}^T\sqrt{1-\beta_t} = \sum_{t=0}^{T}\frac{1}{(t+2)^{1/3}} \geq \int_2^{T+3}\frac{ds}{s^{1/3}} = \frac{3}{2}[(T+3)^{2/3} - 2^{2/3}], \vspace{1ex}\\
&\sum_{t=0}^T\frac{(1-\beta_t)^2}{\sqrt{1-\beta_{t+1}}} = \sum_{t=0}^T\frac{(t+3)^{1/3}}{(t+2)^{4/3}} \leq \sum_{t=0}^{T}\frac{1}{(t+1)}  \leq 1 + \log(T+1).
\end{array}\right.
\end{equation*}
Using these estimates, we can easily prove
\begin{equation*}
\left\{\begin{array}{lcl}
\Sigma_T  :=  \sum_{t=0}^T\omega_t & = &  \frac{\sqrt{b}}{P} \sum_{t=0}^T\sqrt{1-\beta_t}  \geq  \frac{3\sqrt{b}[(T+3)^{2/3} - 2^{2/3} ]}{2P}, \vspace{1ex}\\
\sum_{t=0}^T\frac{\Gamma_{t+1}(1-\beta_t)^2}{\sqrt{1-\beta_{t+1}}} &\leq &   \frac{Q[1 + \log(T+1)]}{P\sqrt{b}}
\end{array}\right.
\end{equation*}
Substituting these inequalities into \eqref{eq:D_convergence1} and using $\sqrt{1-\beta_0} = \frac{1}{2^{1/3}}$, we further upper bound
\begin{equation*} 
\arraycolsep=0.2em
\begin{array}{lcl}
\Exp{\norms{\Gc_{\eta}(\bar{x}_T)}^2} & \leq & \frac{32P}{3\sqrt{b}[(T+3)^{2/3} - 2^{2/3}]}\Big( \Psi_0(x_0) - \Psi^{\star}_0  + \frac{B_{\psi}}{(T+2)^{1/3}} \Big) \vspace{1ex}\\
&& + {~}   \frac{16Q}{3[(T+3)^{2/3} - 2^{2/3}]}\left(\frac{2^{1/3}}{\hat{b}_0} + \frac{2(1 + \log(T+1))}{b}\right),
\end{array}
\end{equation*}
which proves \eqref{eq:convergence_rate1_b_a2}.
\end{proof}
%%% End of Proof.

%%%% D. Convergence Analysis of Restarting Variant, Algorithm~2.
\beforesec
\section{Restarting variant of Algorithm~\ref{alg:A1} and its convergence and complexity}\label{apdx:sec:restarting_hSGD}
\aftersec
In this Supp. Doc., we propose a simple restarting variant, Algorithm~\ref{alg:A2},  of Algorithm~\ref{alg:A1}, prove its convergence, and estimate its oracle complexity bounds for both smooth $\phi_0$ and nonsmooth $\phi_0$ in \eqref{eq:com_nlp}.
For simplicity of our analysis, we only consider the constant step-size case, and omit the diminishing step-size analysis.

%%% 3.4. Restarting variants.
\beforesubsec
\subsection{Restarting variant}
\aftersubsec
\textbf{Motivation:}
Since the constant step-size $\theta$ in \eqref{eq:para_config0} of Theorem~\ref{th:convergence2_scvx} and \eqref{eq:choice_of_para_ncvx} of Theorem~\ref{th:convergence2} depends on the number of iterations $T$.
Clearly, if $T$ is large, then $\theta$ is small. 
To avoid using small step-size $\theta$, we can restart Algorithm~\ref{alg:A1} by frequently resetting its initial point and parameters after $T$ iterations.
This variant is described in Algorithm~\ref{alg:A2}.
Algorithm~\ref{alg:A2} has two loops, where each iteration $s$ of the outer loop is called the $s$-th stage.
Unlike  the outer loop in other variance-reduced methods relying on SVRG or SARAH estimators  from the literature, which is mandatory to guarantee convergence, our outer loop is optional, since without it, Algorithm~\ref{alg:A2} reduces to Algorithm~\ref{alg:A1}, and it still converges. 

\begin{algorithm}[hpt!]\caption{(Restarting Variant of Algorithm~\ref{alg:A1})}\label{alg:A2}
\normalsize
\begin{algorithmic}[1]
   \State{\bfseries Inputs:} An arbitrarily initial point $\tilde{x}^{0} \in\dom{F}$, and a fixed number of iterations $T$.
   \vspace{0.65ex}   
   \State\hspace{0ex}\label{A2_step:o4}{\bfseries For $s := 1,\cdots, S$ do}
   \vspace{0.5ex}   
   \State\hspace{3ex}\label{A2_step:i1} Run Algorithm~\ref{alg:A1} for $T$ iterations starting from $x_0^{(s)} := \tilde{x}^{s-1}$.
   \vspace{0.5ex}   
   \State\hspace{3ex}\label{A2_step:i2} Set $\tilde{x}^{s} := x^{(s)}_{T+1}$ as the last iterate of Algorithm~\ref{alg:A1}.
   \vspace{0.5ex}   
   \State\hspace{0ex}{\bfseries EndFor}
   \vspace{0.5ex}   
   \State\hspace{0ex}\label{step:o5}\textbf{Output:} Choose $\bar{x}_N$ randomly from $\sets{x_t^{(s)}}_{t=0\to T}^{s=1\to S}$ such that $\Prob{\bar{x}_N = x_t^{(s)}} = \frac{\theta_t}{S\sum_{j=0}^T\theta_j}$. 
\end{algorithmic}
\end{algorithm}

\beforesubsec
\subsection{The smooth case $\phi_0$ with constant step-size}
\aftersubsec
The smoothness of $\phi_0$ is equivalent to the $\mu_{\psi}$-strong convexity of $\psi$ in \eqref{eq:min_max_form}.
The following theorem states convergence rate and estimates oracle complexity of Algorithm~\ref{alg:A2}.

\begin{theorem}\label{th:convergence3}
Suppose that Assumptions~\ref{ass:A1} and \ref{ass:A2} hold, $\psi$ is strongly convex $($i.e., $\mu_{\psi} = 1 > 0$$)$, and $P$, $Q$, and $L_{\Phi_0}$ are defined by \eqref{eq:constant_defs}.
Let $\sets{x^{(s)}_t}_{t=0\to T}^{s=1\to S}$ be  generated by Algorithm~\ref{alg:A2} using $\gamma := 0$, $b_0 := c_0\hat{b}_0$, $b_1 = b_2  := c_0b$, $\hat{b}_1 = \hat{b}_2 = b$ for some $c_0 > 0$ and given batch sizes $b > 0$ and $\hat{b}_0 > 0$, and the parameter configuration \eqref{eq:para_config0}.
Then, the following estimate holds
\begin{equation}\label{eq:restart_var}
\Exp{\norms{\Gc_{\eta}(\bar{x}_N)}^2} \leq \frac{16P\hat{b}_0^{1/4}}{S[b(T+1)]^{3/4}}\big[\Psi_0(\tilde{x}^{0}) - \Psi_0^{\star}\big] + \frac{24Q}{[\hat{b}_0b(T+1)]^{1/2}},
\end{equation}
where $\bar{x}_N$ is uniformly randomly chosen from $\sets{x^{(s)}_t}_{t=0\to T}^{s=1\to S}$.

Given $\varepsilon > 0$, if we choose $T  := \rounds{\frac{48Q}{b\varepsilon^2}}$ and $\hat{b}_0 := \rounds{\frac{48Q}{\varepsilon^2}}$, then after at most $S := \rounds{\frac{8P}{\varepsilon\sqrt{3Q}}}$ outer iterations, we obtain $\Exp{\norms{\Gc_{\eta}(\bar{x}_N)}^2} \leq\varepsilon^2$.
Consequently, the total number of function evaluations $\Tc_F$ and the total number of Jacobian evaluations $\Tc_J$ are at most $\Tc_F = \Tc_J :=  \rounds{\frac{400 P\sqrt{3Q}}{\varepsilon^3}}$.
\end{theorem}

Theorem~\ref{th:convergence3} holds for any mini-batch $b$ such that $1 \leq b \leq \frac{48Q}{\varepsilon^2}$, which is different from, e.g., \cite{zhang2019multi}, where the complexity result holds under large batches.
Moreover, the total oracle calls $\Tc_F$ and $\Tc_J$ are independent of $b$.
In this case, the weight $\beta$  and the step-size $\theta$ become
\begin{equation*}
\beta :=  1 - \frac{b\varepsilon^2}{48Q}
\quad\text{and}\quad
\theta :=   \frac{bL_{\Phi_0}}{4P\varepsilon\sqrt{3Q}}.
\end{equation*}
Clearly, if $b$ is large, then our step-size $\theta$ is also large.

%%% Beginning of the proof of Theorem 3.3.
\begin{proof}[\textbf{The proof of Theorem~\ref{th:convergence3}: Restarting variant}]
Since $\gamma := 0$, $\hat{b}_1 = \hat{b}_2 := b$ and $b_1 = b_2 := c_0b$, 
from \eqref{eq:D_proof26}, using the superscript ``$^{(s)}$'' for the outer iteration $s$, and $P$ and $Q$ from \eqref{eq:constant_defs},  we have
\begin{equation*} 
\frac{\theta}{16L_{\Phi_{0}}}\Exp{\norms{\Gc_{\eta}(x_t^{(s)})}^2} \leq  V_{0}(x_t^{(s)}) - V_{0}(x_{t+1}^{(s)}) + \frac{Q(1-\beta)^{3/2}}{Pb^{1/2}},
\end{equation*}
Summing up this inequality from $t :=0$ to $t :=T$, and using the fact that $\tilde{x}^{s-1} := x_0^{(s)}$ and $\tilde{x}^{s} := x_{T+1}^{(s)}$, we get
\begin{equation*} 
\frac{\theta}{16L_{\Phi_{0}}}\sum_{t=0}^T\Exp{\norms{\Gc_{\eta}(x_t^{(s)})}^2} \leq  V_{0}(\tilde{x}^{s-1}) - V_{0}(\tilde{x}^{s}) + \frac{Q(T+1)(1-\beta)^{3/2}}{Pb^{1/2}}.
\end{equation*}
Using the choice $b_0 := c_0\hat{b}_0$, similar to the proof of \eqref{eq:D_proof27}, we can show that
\begin{equation*} 
\hspace{-1ex}\begin{array}{lcl}
V_0(\tilde{x}^{s-1}) &= & \Exp{\Psi_0(\tilde{x}^{s-1})} + \frac{\alpha}{2}\Exp{\norms{\tilde{F}_0^{(s)} - F(\tilde{x}^{s-1})}^2} + \frac{\hat{\alpha}}{2}\Exp{\norms{\tilde{J}_0^{(s)} - F'(\tilde{x}^{s-1})}^2} \vspace{1ex}\\
&\leq&  \Exp{\Psi_{0}(\tilde{x}^{s-1})} +\frac{Q b^{1/2}}{2P \hat{b}_0\sqrt{1-\beta}}.
\end{array}\hspace{-4ex}
\end{equation*}
Using this estimate and and $V_{0}(\tilde{x}^{s}) \geq \Psi_0(\tilde{x}^{s})$ into above inequality, we can further derive
\begin{equation*} 
\begin{array}{lcl}
\frac{1}{(T+1)}\sum_{t=0}^T\Exp{\norms{\Gc_{\eta}(x_t^{(s)})}^2} &\leq &  \frac{16L_{\Phi_0}}{\theta(T+1)}\left[\Psi_{0}(\tilde{x}^{s-1}) - \Psi_{0}(\tilde{x}^{s})\right] + \frac{16 Q L_{\Phi_0}(1-\beta)^{3/2}}{P\theta b^{1/2}} \vspace{1ex}\\
&& + {~} \frac{8Q L_{\Phi_0} b^{1/2}}{P\theta(T+1)\hat{b}_0\sqrt{1-\beta}}.
\end{array}
\end{equation*}
Due to the choice of $b_1$ and $\hat{b}_1$, it follows from \eqref{eq:para_config0} that $\beta := 1 - \frac{b^{1/2}}{[\hat{b}_0(T+1)]^{1/2}}$ and $\theta :=  \frac{L_{\Phi_0}b^{3/4}}{P[\hat{b}_0(T+1)]^{1/4}}$.
Therefore, the last inequality becomes
\begin{equation*} 
\frac{1}{(T+1)}\sum_{t=0}^T\Exp{\norms{\Gc_{\eta}(x_t^{(s)})}^2} \leq  \frac{16P\hat{b}_0^{1/4}}{[b(T+1)]^{3/4}}\left[\Psi_{0}(\tilde{x}^{s-1}) - \Psi_{0}(\tilde{x}^{s})\right] +  \frac{24Q}{[\hat{b}_0b(T+1)]^{1/2}}.
\end{equation*}
Summing up this inequality from $s := 1$ to $s := S$ and multiplying the result by $\frac{1}{S}$, we get
\begin{equation*} 
\hspace{-1ex}
\frac{1}{S(T+1)}\sum_{s=1}^S\sum_{t=0}^T\Exp{\norms{\Gc_{\eta}(x_t^{(s)})}^2} \leq   \frac{16P\hat{b}_0^{1/4}}{S[b(T+1)]^{3/4}} \big[\Psi_0(\tilde{x}^{0}) - \Psi_0(\tilde{x}^{S})\big]  +  \frac{24Q}{[\hat{b}_0b(T+1)]^{1/2}}.
\hspace{-2ex}
\end{equation*}
Substituting $\Psi_0(\tilde{x}^{S}) \geq \Psi_0^{\star}$ into the last inequality, and using the fact that $\Exp{\norms{\Gc_{\eta}(\bar{x}_N)}^2} = \frac{1}{S(T+1)}\sum_{s=1}^S\sum_{t=0}^T\Exp{\norms{\Gc_{\eta}(x_t^{(s)})}^2}$, we obtain 
\begin{equation*} 
\hspace{-1ex}
\begin{array}{lcl}
\Exp{\norms{\Gc_{\eta}(\bar{x}_N)}^2} &= & \frac{1}{S(T+1)}\sum_{s=1}^S\sum_{t=0}^T\Exp{\norms{\Gc_{\eta}(x_t^{(s)})}^2} \vspace{1ex}\\
& \leq &  \frac{16P\hat{b}_0^{1/4}}{S[b(T+1)]^{3/4}} \big[\Psi_0(\tilde{x}^{0}) - \Psi_0^{\star}\big]  + \frac{24Q}{[\hat{b}_0b(T+1)]^{1/2}},
\end{array}
\hspace{-2ex}
\end{equation*}
which is exactly \eqref{eq:restart_var}.

Now, for a given tolerance $\varepsilon > 0$, to obtain $\Exp{\norms{\Gc_{\eta}(\bar{x}_K)}^2} \leq\varepsilon^2$, we need to impose
\begin{equation*}
 \frac{16P \hat{b}_0^{1/4}}{S[b(T+1)]^{3/4}} = \frac{\varepsilon^2}{2}   \quad \text{and}\quad \frac{24Q }{[ \hat{b}_0b(T+1)]^{1/2}} = \frac{\varepsilon^2}{2}.
\end{equation*}
This condition leads to $N = S(T+1) = \frac{32P [ \hat{b}_0(T+1)]^{1/4}}{b^{3/4}\varepsilon^2}$ and $\hat{b}_0b(T+1) = \frac{48^2 Q^2}{\varepsilon^4}$.
Hence, the total number of iterations is $N := S(T+1) = \frac{32P [\hat{b}_0b(T+1)]^{1/4}}{b\varepsilon^2} = \frac{128 P\sqrt{3Q}}{b\varepsilon^3}$.

Clearly, to optimize the oracle complexity, we need to choose $T + 1 := \frac{48Q}{b\varepsilon^2}$, then $\hat{b}_0 := \frac{48Q}{\varepsilon^2}$ and $S := \frac{8P}{\sqrt{3Q}\varepsilon}$.
In this case, the total number of function evaluations is at most
\begin{equation*}
\Tc_F := b_0S + 3bS(T+1) = \frac{48Q}{\varepsilon^2}\cdot \frac{8P}{\sqrt{3Q}\varepsilon} + 3bN = \frac{16P\sqrt{3Q}}{\varepsilon^3} + \frac{384 P\sqrt{3Q}}{\varepsilon^3} = \frac{400 P\sqrt{3Q}}{\varepsilon^3}.
\end{equation*}
This is also the total number of Jacobian evaluations $\Tc_J$.
\end{proof}
%%% End of the proof.

\beforesubsec
\subsection{The nonsmooth $\phi_0$ with constant step-size}
\aftersubsec
Finally, we prove the convergence of Algorithm~\ref{alg:A2} when $\psi$ is non-strongly convex (i.e., $\phi_0$ in \eqref{eq:com_nlp} is possibly nonsmooth).

%%% Theorem 3.4.
\begin{theorem}\label{th:convergence2b}
Assume that Assumptions~\ref{ass:A1} and \ref{ass:A2} hold, $\psi$ in \eqref{eq:min_max_form} is non-strongly convex $($i.e., $\mu_{\psi} = 0$$)$, and $P$, $Q$, and $L_{\Phi_{\gamma}}$ are defined by \eqref{eq:constant_defs}.
Let $\sets{x_t^{(s)}}_{t=0\to T}^{s=1\to S}$ be generated by Algorithm~\ref{alg:A2} after $N := S(T+1)$ iterations using:
\begin{equation}\label{eq:choice_of_para_ncvx_2b}
\left\{\begin{array}{ll}
&b_1 = b_2 := \frac{2c_0b\hat{R}_0}{\varepsilon^2},  \quad \hat{b}_1 = \hat{b}_2 := b, \quad b_0 := \frac{4c_0\hat{R}_0^2}{\varepsilon^4}, \quad \hat{b}_0 := \frac{2\hat{R}_0}{\varepsilon^2}, \vspace{1ex}\\
&\gamma := \frac{\varepsilon}{\sqrt{2\hat{R}_0}},\quad\text{and}\quad \beta := 1  - \frac{b\varepsilon^2}{2\hat{R}_0}.
\end{array}\right.
\end{equation}
where $\varepsilon > 0$ is a given tolerance\footnote{The batch sizes and $T$ in this paper must be integer, but for simplicity, we do not write their rounding form.}, and
\begin{equation}\label{eq:R0_def5} 
R_0  := 16\big[ \Psi_{0}(\tilde{x}^{0}) - \Psi^{\star}  + B_{\psi}\big] \quad\text{and}\quad \hat{R}_0 := 24Q.
\end{equation}
Then, if we choose $T := \rounds{\frac{2\hat{R}_0}{\varepsilon^2}}$, then after at most $S := \rounds{\frac{\sqrt{2}R_0}{b\varepsilon\sqrt{\hat{R}_0}}}$ outer iterations, we obtain $\bar{x}_T$ such that $\Exp{\norms{\Gc_{\eta}(\bar{x}_T)}^2}\leq\varepsilon^2$.

Consequently, the total number of function evaluations $\Tc_F$ and the total number of Jacobian evaluations $\Tc_J$ are respectively at most
\begin{equation*}
\begin{array}{l}
\Tc_F :=  \frac{4\sqrt{2}c_0R_0\hat{R}_0^{3/2}(3 + b^{-1})}{\varepsilon^5} = \BigO{\frac{R_0\hat{R}_0^{3/2}}{\varepsilon^{5}}}
\quad\text{and}\quad
\Tc_J :=  \frac{2\sqrt{2}R_0\hat{R}_0^{1/2}(3 + b^{-1})}{\varepsilon^3} = \BigO{\frac{R_0\hat{R}_0^{1/2}}{\varepsilon^{3}}}.
\end{array}
\end{equation*}
\end{theorem}

%%% Remark 1.
\begin{remark}
Note that we do not need to choose the batch sizes and parameters depending on $R_0$ as in \eqref{eq:choice_of_para_ncvx_2b}, which is unknown since $\Psi_0^{\star}$ is unknown, but they are proportional to $R_0$.
In this case, the complexity bounds in Theorem~\ref{th:convergence2b} will only be shifted by a constant factor. 

As we can see from Theorem~\ref{th:convergence2b}, the number of outer iterations $S$ is divided by the batch size $b$.
However, the terms $\frac{12\sqrt{2}c_0R_0\hat{R}_0^{3/2}}{\varepsilon^5}$ and $\frac{6\sqrt{2}R_0\hat{R}_0^{1/2}}{\varepsilon^3}$ are independent of $b$ and dominate the complexity bounds in both $\Tc_F$ and $\Tc_J$, respectively. 
\end{remark}

%%% Proof of Theorem 3.4.
\begin{proof}[\textbf{The proof of Theorem~\ref{th:convergence2b}}]
Let us first choose $\hat{b}_1 = \hat{b}_2 := b$, $b_1 = b_2 := \frac{c_0b}{\gamma^2}$, and $b_0 := \frac{c_0\hat{b}_0}{\gamma^2}$.
With the same line as the proof of  \eqref{eq:proof26}, we can show that
\begin{equation*} 
\hspace{-1ex}\begin{array}{lcl}
\frac{1}{(T+1)}\sum_{t=0}^T\Exp{\norms{\Gc_{\eta}(x_t^{(s)})}^2} &\leq& \frac{16P}{(T+1)[b(1-\beta)]^{1/2}}\big[\Exp{\Psi_{0}(x_0^{(s)})} - \Exp{\Psi_0(x_{T+1}^{(s)})}  + \gamma B_{\psi}\big] \vspace{1ex}\\
&& + {~} 8Q \left[ \frac{1}{\hat{b}_0(1-\beta)(T+1)}  + \frac{2(1 - \beta)}{b} \right].
\end{array}\hspace{-2ex}
\end{equation*}
Here, we use the superscript ``$^{(s)}$'' to present the outer iteration $s$.
Moreover, instead of $\Psi_0^{*}$, we keep $\Psi_0(x^{(s)}_{T+1})$ from \eqref{eq:D_V_to_Psi}.
Now, using the fact that $\tilde{x}^{s-1} = x_0^{(s)}$ and $\tilde{x}^{s} = x_{T+1}^{(s)}$, we can further derive from the above inequality that
\begin{equation*} 
\hspace{-1ex}\begin{array}{lcl}
\frac{1}{(T+1)}\sum_{t=0}^T\Exp{\norms{\Gc_{\eta}(x_t^{(s)})}^2} &\leq& \frac{16P}{(T+1)[b(1-\beta)]^{1/2}}\big[\Exp{\Psi_{0}(\tilde{x}^{s-1})} - \Exp{\Psi_0(\tilde{x}^{s})}  + \gamma B_{\psi}\big] \vspace{1ex}\\
&& + {~} 8Q \left[ \frac{1}{\hat{b}_0(1-\beta)(T+1)}  + \frac{2(1 - \beta)}{b} \right]. 
\end{array}\hspace{-2ex}
\end{equation*}
Summing up this inequality from $s:=1$ to $s := S$, and multiplying the result by $\frac{1}{S}$, and then using $0 < \gamma \leq 1$, $\Exp{\Psi_{0}(\tilde{x}^{0})} = \Psi_{0}(\tilde{x}^{0})$,  $\Psi_0(\tilde{x}^{S}) \geq \Psi_0^{\star} > -\infty$, and $ \Exp{\norms{\Gc_{\eta}(\bar{x}_N)}^2} =  \frac{1}{S(T+1)}\sum_{s=1}^S\sum_{t=0}^T \Exp{\norms{\Gc_{\eta}(x_t^{(s)})}^2}$, we arrive at
\begin{equation*} 
\hspace{-1ex}\begin{array}{lcl}
\Exp{\norms{\Gc_{\eta}(\bar{x}_N)}^2} &= & \frac{1}{(T+1)S}\sum_{s=1}^S \sum_{t=0}^T\Exp{\norms{\Gc_{\eta}(x_t^{(s)})}^2} \vspace{1ex}\\
&\leq& \frac{16P}{S(T+1)[b(1-\beta)]^{1/2}}\big[ \Psi_{0}(\tilde{x}^{0}) - \Psi^{\star}  +  B_{\psi}\big] \vspace{1ex}\\
&& + {~} 8Q \left[ \frac{1}{\hat{b}_0(1-\beta)(T+1)}  + \frac{2(1 - \beta)}{b} \right].
\end{array}\hspace{-2ex}
\end{equation*}
Next, let us choose $\beta := 1 - \frac{b}{(T+1)}$ and $\hat{b}_0 := (T+1)$.
Then, the above estimate becomes
\begin{equation*} 
\Exp{\norms{\Gc_{\eta}(\bar{x}_N)}^2} \leq \frac{16P}{bS(T+1)^{1/2}}\big[ \Psi_{0}(\tilde{x}^{0}) - \Psi^{\star}  + B_{\psi}\big] + \frac{24Q}{T+1}.
\end{equation*}
Let us define $R_0$ and $\hat{R}_0$ as in \eqref{eq:R0_def5}, i.e.:
\begin{equation*} 
\begin{array}{l}
 R_0  :=  16P\big[ \Psi_{0}(\tilde{x}^{0}) - \Psi^{\star}  + B_{\psi}\big]  \quad\text{and}\quad
\hat{R}_0  := 24Q.
\end{array}
\end{equation*}
In this case, for a given tolerance $\varepsilon > 0$, to achieve $\Exp{\norms{\Gc_{\eta}(\bar{x}_N)}^2} \leq\varepsilon^2$, we can impose
\begin{equation*}
\frac{R_0}{bS(T+1)^{1/2}} = \frac{\varepsilon^2}{2}
\quad\text{and}\quad
\frac{\hat{R}_0}{(T+1)} = \frac{\varepsilon^2}{2}.
\end{equation*}
These conditions lead to $T+1 = \frac{2\hat{R}_0}{\varepsilon^2}$ and $S := \frac{2R_0}{b(T+1)^{1/2}\varepsilon^2} = \frac{\sqrt{2}R_0}{b\varepsilon\sqrt{\hat{R}_0}}$.
Let us also choose $\gamma :=  \frac{\varepsilon}{\sqrt{2\hat{R}_0}}$.
Then, we also obtain the parameters as in \eqref{eq:choice_of_para_ncvx_2b}, i.e.:
\begin{equation*} 
\left\{\begin{array}{ll}
&b_1 = b_2 := \frac{2c_0b\hat{R}_0}{\varepsilon^2},  \quad \hat{b}_1 = \hat{b}_2 := b, \quad b_0 := \frac{4c_0\hat{R}_0^2}{\varepsilon^4}, \quad \hat{b}_0 := \frac{2\hat{R}_0}{\varepsilon^2}, \vspace{1ex}\\
&\gamma := \frac{\varepsilon}{\sqrt{2\hat{R}_0}},\quad\text{and}\quad \beta := 1  - \frac{b\varepsilon^2}{2\hat{R}_0}.
\end{array}\right.
\end{equation*}
The total number $\Tc_F$ of function evaluations $\Fb(x_t^{(s)},\xi_t)$ is at most
\begin{equation*}
\Tc_F := S[b_0 + (T+1)(2b_1 + b_2)] = \frac{\sqrt{2}R_0}{b\varepsilon\sqrt{\hat{R}_0}}\Big[ \frac{4c_0\hat{R}_0^2}{\varepsilon^4} +  \frac{2\hat{R}_0}{\varepsilon^2} \frac{6c_0b\hat{R}_0}{\varepsilon^2} \Big] = \frac{4\sqrt{2}c_0R_0\hat{R}_0^{3/2}}{\varepsilon^5}\left(\frac{1}{b} + 3\right).
\end{equation*}
The total number $\Tc_J$ of Jacobian evaluations $\Fb'(x_t^{(s)},\xi_t)$ is at most
\begin{equation*}
\Tc_J := S[\hat{b}_0 + (T+1)(2\hat{b}_1 + \hat{b}_2)] = \frac{\sqrt{2}R_0}{b\varepsilon\sqrt{\hat{R}_0}}\Big[ \frac{2\hat{R}_0}{\varepsilon^2} +  \frac{6b\hat{R}_0}{\varepsilon^2} \Big] = \frac{2\sqrt{2}R_0\hat{R}_0^{1/2}}{b\varepsilon^3} + \frac{6\sqrt{2}R_0\hat{R}_0^{1/2}}{\varepsilon^3}.
\end{equation*}
These prove the last statement of Theorem~\ref{th:convergence2b}.
\end{proof}
%%% End of the proof.

%%%%%%%%%%%%%%%%%%%%%%%%%%%%%%%%%%%%%
%%% D. Additional numerical examples
%%%%%%%%%%%%%%%%%%%%%%%%%%%%%%%%%%%%%
\beforesec
\section{Experiment setup and additional experiments}\label{apdx:sec:add_num_exam}
\aftersec
This Supp. Doc. provides the details of configuration for our experiments in Section~\ref{sec:num_exps}, and presents more numerical experiments to support our algorithms and theoretical results.
As mentioned in the main text, all the algorithms used in this paper have been implemented in Python 3.6.3., running on a Linux desktop (3.6GHz Intel Core i7 and 16Gb memory).

Let us provide more details of our experiment configuration.
We shorten the name of our algorithm, either Algorithm~\ref{alg:A1} or Algorithm~\ref{alg:A2}, by Hybrid Stochastic Compositional Gradient, and  abbreviate it by \texttt{HSCG} for both cases.
We have implemented \texttt{CIVR} in \cite{zhang2019stochastic} and \texttt{ASC-PG} in \cite{wang2017accelerating} to compare the smooth case of $\phi_0$.
For the nonsmooth case of $\phi_0$, we have implemented two other algorithms,  \texttt{SCG} in \cite{wang2017stochastic}, and \texttt{Prox-Linear} in \cite{tran2020stochastic,zhang2020stochastic}.
While \texttt{SCG} only works for smooth $\phi_0$, we have smoothed it as in our method, and used the estimator as well as the algorithm in \cite{wang2017stochastic}, but update the smoothness parameter as in our method.
We also omit comparison in terms of time since  \texttt{Prox-Linear} becomes slower if $p$ is large due to its expensive subproblem for evaluating the prox-linear operator.
We only compare these algorithms in terms of epoch (i.e., the number of data passes).

Since both \texttt{CIVR} and \texttt{ASC-PG} are double loop, to be fair, we compare them with our restarting variant, Algorithm~\ref{alg:A2}.
To compare with \texttt{SCG} and \texttt{Prox-Linear}, we simply use Algorithm~\ref{alg:A1} since \texttt{SCG} has single loop.
Since  \texttt{Prox-Linear} requires to solve a nonsmooth convex subproblem, we have implemented a first-order primal-dual method in \cite{Chambolle2011} to solve it.
This algorithm has shown its efficiency in our test.

Note that the batch size $b$ is determined as $b := \rounds{\frac{N}{n_b}}$, where $N$ is the number of data points, and $n_b$ is the number of blocks.
In our experiments, we have varied the number of blocks $n_b$ to observe the performance of these algorithms.
Since we want to obtain the best performance, instead of using their theoretical step-sizes, we have carefully tuned the step-size $\eta$ of three algorithms in a given set of candidates $\set{1, 0.5, 0.1, 0.05, 0.01, 0.001,0.0001}$. 
For our algorithms, we have another step-size $\theta_t$, which is also flexibly chosen from $\sets{0.1, 0.5, 1}$.
For the nonsmooth case, we update our smoothness parameter as $\gamma_t := \frac{1}{2(t+1)^{1/3}}$, which is proportional to the value in Theorems~\ref{th:convergence2_scvx_diminishing} and \ref{th:nonsmooth_diminishing}. 

To further compare our algorithms with their competitors, we provide in the following subsections additional experiments for the two problems in the main text.

\beforesubsec
\subsection{Risk-averse portfolio optimization: Additional experiments}\label{apdx:D_example1} 
\aftersubsec
Figure~\ref{fig:risk_averse} in the main text has shown the performance of three algorithms on three different datasets using $8$ blocks, i.e., $n_b = 8$.
Unfortunately, since ASC-PG does not work well when the number of blocks is larger than $8$, we skip showing it in our comparison.
To obverse more performance of  \texttt{HSCG} and \texttt{CIVR}, we have increased the number of blocks $n_b$ from $8$ to $32$, $64$, and $128$.
The convergence of the two algorithms is shown in Figure \ref{fig:supp_risk}. 
As we can observe, \texttt{HSCG} remains slightly better than  \texttt{CIVR} if $n_b = 32$ or $64$.
When $n_b=128$, \texttt{CIVR} improves its performance and is slightly better than \texttt{HSCG}.

\begin{figure}[ht!]
\begin{center}
\hspace{-2ex}
\includegraphics[width=1\textwidth]{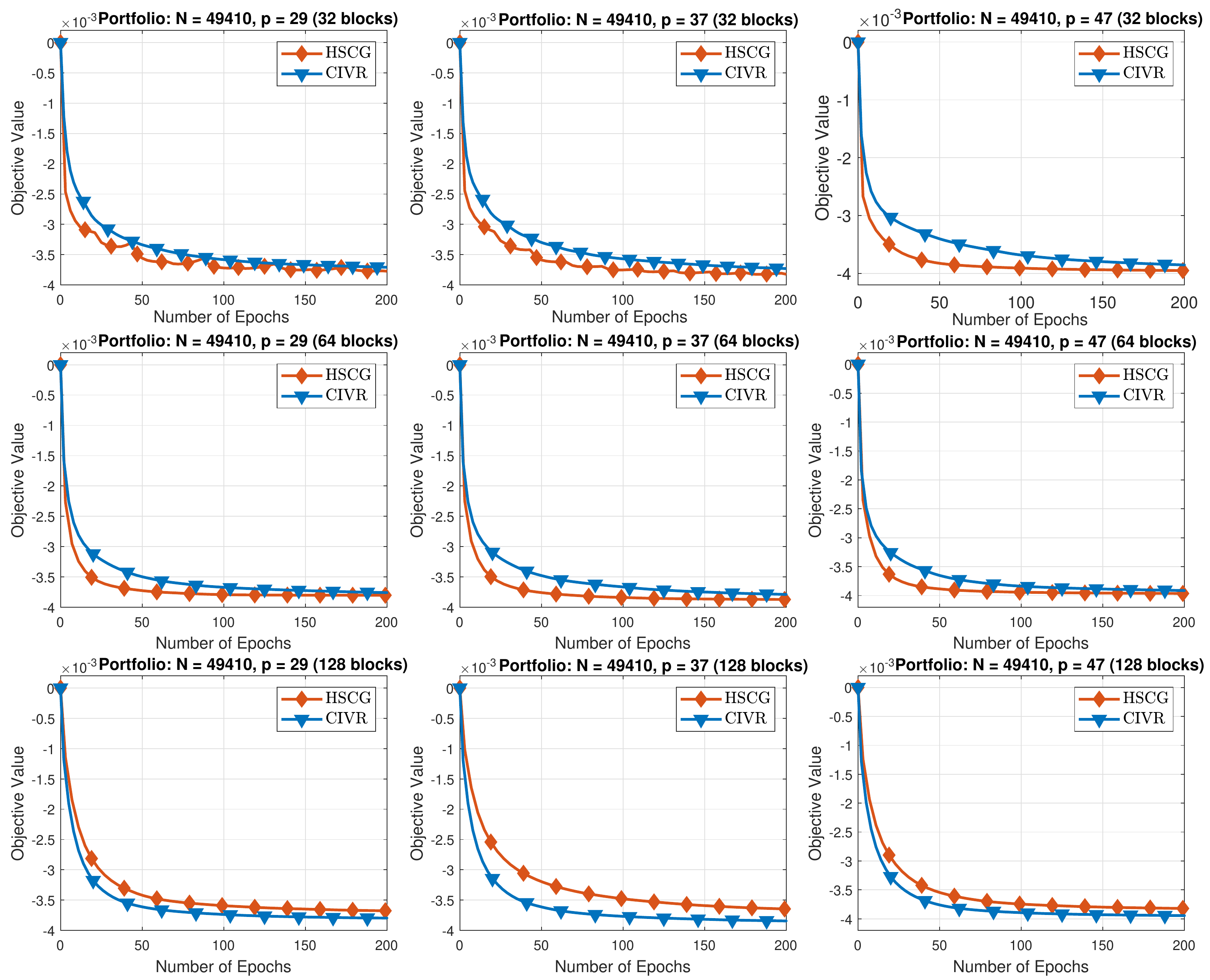}
\hspace{-2ex}
\caption{
Comparison of two algorithms for solving \eqref{eq:portfolio_exam} on larger blocks.
}
\label{fig:supp_risk}
\end{center}
\end{figure}

\beforesubsec
\subsection{Stochastic minimax problem: Additional experiments}\label{apdx:D_example2} 
\aftersubsec
For the stochastic minimax problem \eqref{eq:min_max_stochastic_opt},  Figure~\ref{fig:min_max_stochastic_opt} has shown the progress of the objective values  of three algorithms on three different datasets.
Figure~\ref{fig:min_max_stochastic_opt2} simultaneously shows both the objective values and the gradient mapping norms of this experiment.

\begin{figure}[ht!]
\hspace{-2ex}
\centering
\includegraphics[width=1\textwidth]{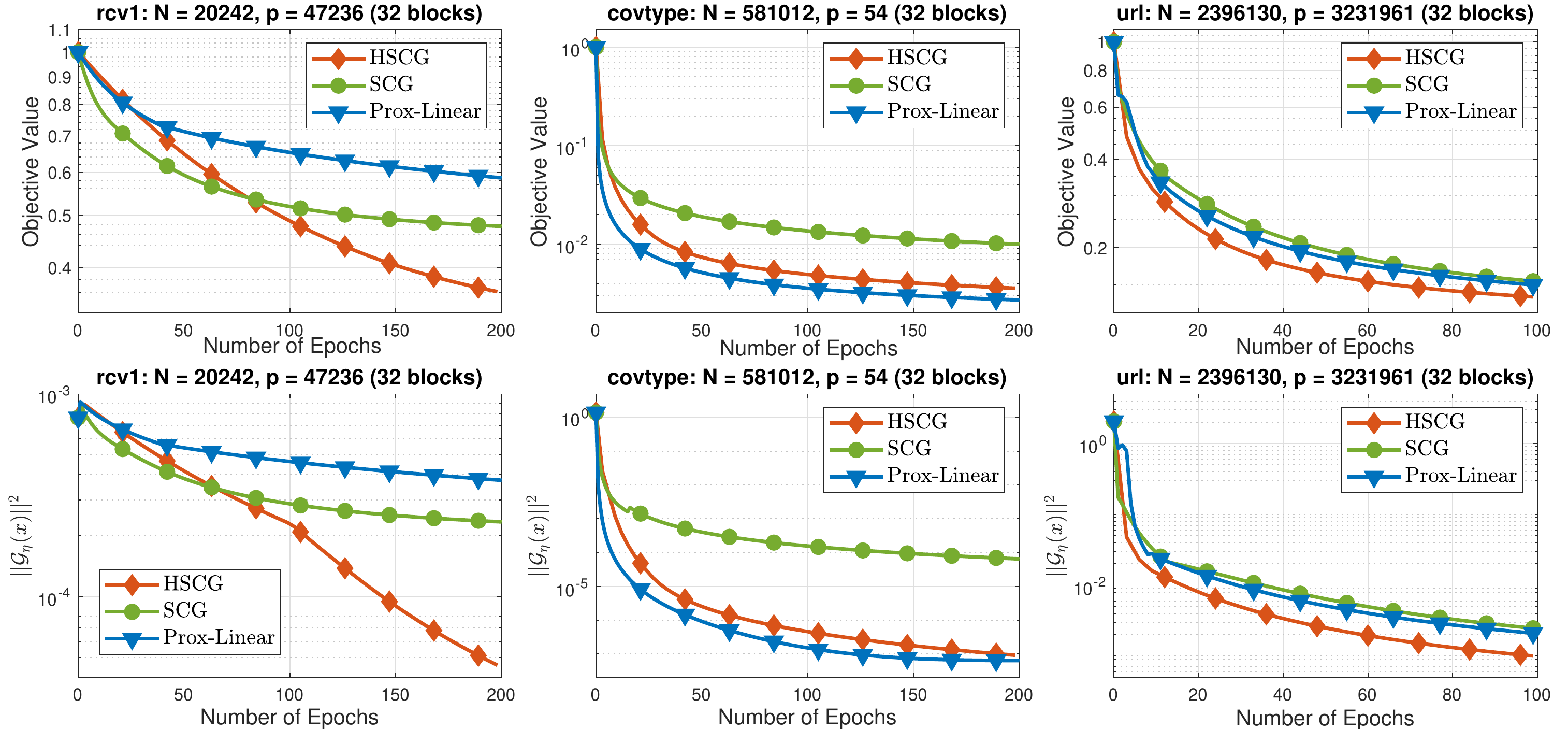}
\hspace{-2ex}
\caption{
Comparison of three algorithms for solving \eqref{eq:min_max_stochastic_opt} on $3$ different datasets in Figure~\ref{fig:min_max_stochastic_opt} with both objective values and gradient mapping norms.
}
\label{fig:min_max_stochastic_opt2}
\end{figure}

Now, let us keep the same configuration as in Figure~\ref{fig:min_max_stochastic_opt}, but run one more case, where the number of blocks is increased to $n_b = 64$.
The results are shown in Figure \ref{fig:supp_minmax_2}.

We again see that \texttt{HSCG} still highly outperforms the other two methods: \texttt{SCG} and \texttt{Prox-Linear} on \textbf{rcv1}.
For \textbf{url}, \texttt{HSCG} is still slightly better than \texttt{Prox-Linear} as we have observed in Figure~\ref{fig:min_max_stochastic_opt}.
However, for \textbf{covtype}, again,  \texttt{Prox-Linear} shows a better performance than the other two competitors. 
Note that since $p=54$ in this dataset, we can solve the subproblem in  \texttt{Prox-Linear} up to a high accuracy without incurring too much computational cost.
Therefore, the inexactness of evaluating the prox-linear operator does not really affect the performance in this example.

\begin{figure}[ht!]
\begin{center}
\hspace{-2ex}
\includegraphics[width=1\textwidth]{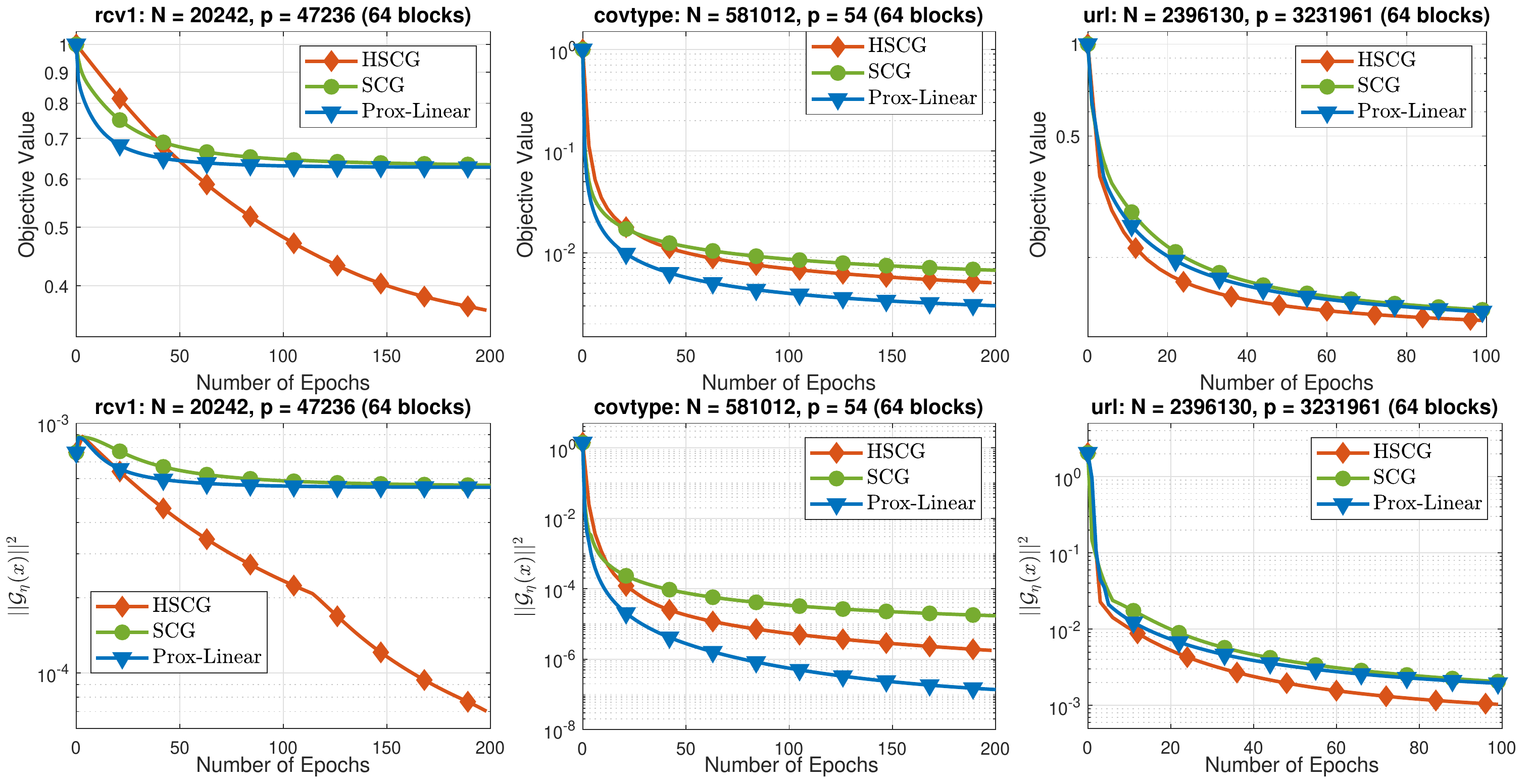}
\hspace{-2ex}
\caption{
Comparison of three algorithms for solving \eqref{eq:min_max_stochastic_opt} on 64 blocks.
}
\label{fig:supp_minmax_2}
\end{center}
\end{figure}

Finally, we test three algorithms: \texttt{HSCG}, \texttt{SCG}, and \texttt{Prox-Linear} on other three datasets:  \textbf{w8a}, \textbf{phishing}, and \textbf{mushrooms} from LIBSVM \cite{CC01a}.
We use the same number of blocks $n_b = 32$, and the results are reported in Figure~\ref{fig:supp_minmax_1}.
Figure~\ref{fig:supp_minmax_1} shows that  \texttt{HSCG} highly outperforms both \texttt{SCG} and \texttt{Prox-Linear} on \textbf{w8a} and  \textbf{phishing}.
However,  \texttt{Prox-Linear} becomes better than the other two on the \textbf{mushrooms} dataset.

\begin{figure}[ht!]
\begin{center}
\hspace{-2ex}
\includegraphics[width=1\textwidth]{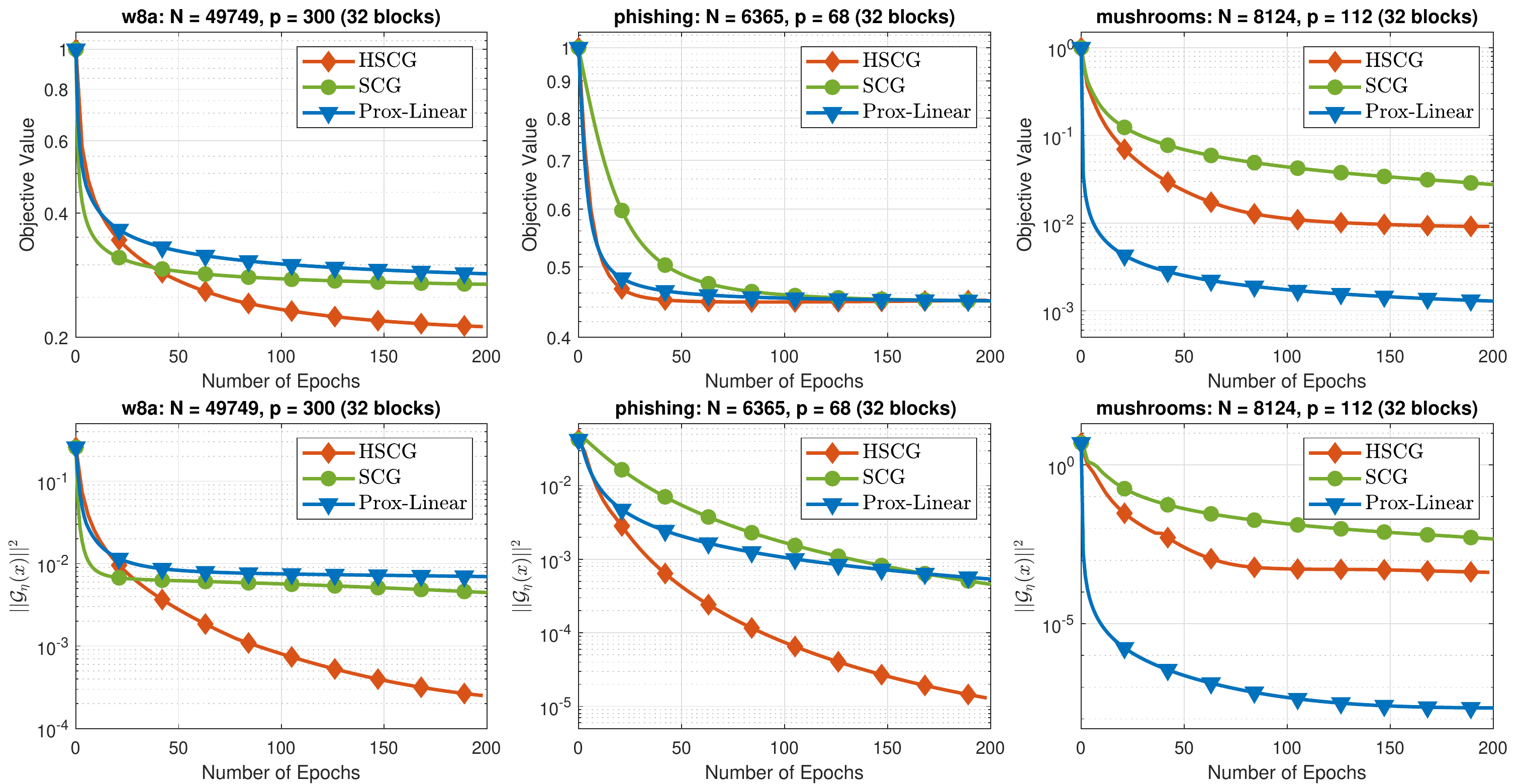}
\hspace{-2ex}
\caption{
Comparison of three algorithms for solving \eqref{eq:min_max_stochastic_opt} on three more different datasets.
}
\label{fig:supp_minmax_1}
\end{center}
\end{figure}

\end{document}

%% file: my_preamble.tex
\usepackage{amsthm}
\usepackage{amsmath}
\usepackage{amssymb}
\usepackage{graphicx}
\usepackage{color}
\usepackage{algorithm}
\usepackage[usenames,dvipsnames]{xcolor}
\usepackage{paralist}
\usepackage{algorithm}
\usepackage{algpseudocode}
\usepackage{multicol} % Used for the two-column layout of the document
\usepackage{caption} % Custom captions under/above floats in tables or figures
\usepackage{hyperref} % For hyperlinks in the PDF

\theoremstyle{plain}
\newtheorem{theorem}{Theorem}[section]

\newtheorem{lemma}{Lemma}[section]

\theoremstyle{definition}
\newtheorem{definition}{Definition}[section]
\newtheorem{assumption}{Assumption}[section]
\newtheorem{remark}{Remark}[section]

\newcommand{\R}{\mathbb{R}}
\newcommand{\Rext}{\R\cup\{+\infty\}}

\newcommand{\set}[1]{\left\{#1\right\}}
\newcommand{\sets}[1]{\{#1\}}
\newcommand{\norm}[1]{\left\Vert#1\right\Vert}
\newcommand{\norms}[1]{\Vert#1\Vert}

\newcommand{\prox}{\mathrm{prox}}

\newcommand{\argmin}{\mathrm{arg}\!\displaystyle\min}

\newcommand{\dom}[1]{\mathrm{dom}(#1)}

\newcommand{\zero}[1]{{\boldsymbol{0}}}
\newcommand{\Exps}[2]{\mathbb{E}_{#1}\big[#2\big]}
\newcommand{\Vars}[2]{\mathrm{Var}_{#1}\big[#2\big]}
\newcommand{\Prob}[1]{\mathbf{Prob}\left(#1\right)}

\newcommand{\Fb}{\mathbf{F}}

\newcommand{\Bc}{\mathcal{B}}

\newcommand{\Gc}{\mathcal{G}}
\newcommand{\Xc}{\mathcal{X}}

\newcommand{\Ec}{\mathcal{E}}

\newcommand{\Rc}{\mathcal{R}}

\newcommand{\Tc}{\mathcal{T}}

\newcommand{\Fc}{\mathcal{F}}

\newcommand{\iprods}[1]{\langle #1\rangle}

\newcommand{\Exp}[1]{\mathbb{E}\big[#1\big]}

\newcommand{\dist}[1]{\mathrm{dist}\left(#1\right)}

\newcommand{\BigO}[1]{\mathcal{O}\left(#1\right)}

\newcommand{\rounds}[1]{\big\lfloor{#1}\big\rfloor}

\newcommand{\mytxtbi}[1]{\textbf{\textit{#1}}}

\newcommand{\beforesubsec}{\vspace{0ex}}
\newcommand{\aftersubsec}{\vspace{0ex}}
\newcommand{\beforesec}{\vspace{0ex}}
\newcommand{\aftersec}{\vspace{0ex}}